\documentclass[11pt]{amsart}

\usepackage{amsmath,amssymb,amsthm,amsfonts,amscd,flafter,
graphicx,verbatim,pinlabel,mathrsfs}
\usepackage[all]{xy}
\usepackage{epstopdf}
\usepackage[colorlinks=false]{hyperref}
\usepackage[all]{hypcap}

\usepackage{collectbox}



\title{On the equivalence of contact invariants in sutured Floer homology theories}

\author[John A. Baldwin]{John A. Baldwin}
\address{Department of Mathematics \\ Boston College}
\email{john.baldwin@bc.edu}

\author[Steven Sivek]{Steven Sivek}
\address{Department of Mathematics \\ Imperial College London}
\email{s.sivek@imperial.ac.uk}

\thanks{JAB was  supported by NSF Grants DMS-1104688, DMS-1406383, and NSF CAREER Grant DMS-1454865. }
\thanks{SS was supported by  NSF Grants DMS-1204387 and DMS-1506157.}

\def\D{{\mathbb{D}}}
\def\T{{\mathbb{T}}}

\def\Q{{\mathbb{Q}}}
\def\P{{\mathbb{P}}}

\newcommand\ob{\mathfrak{ob}}
\newcommand\bx{\mathbf{x}}

\newcommand\by{\mathbf{y}}

\newcommand\bcc{\mathbf{c}}

\newcommand\sfh{\mathit{SFH}}
\newcommand\sfc{\mathit{SFC}}
\newcommand\hfp{\mathit{HF}^+}

\newcommand\cfp{\mathit{CF}^+}

\newcommand\zz{\mathbb{Z}}
\newcommand\zzt{\mathbb{Z}/2\mathbb{Z}}

\newcommand\Sc{\text{Spin}^c}
\newcommand\SA{\text{A}}
\newcommand\spc{\mathfrak{s}}
\newcommand\spt{\mathfrak{t}}
\newcommand\Ta{\mathbb{T}_{\alpha}}
\newcommand\Tb{\mathbb{T}_{\beta}}
\newcommand\Tc{\mathbb{T}_{\gamma}}

\newcommand\ul{\underline}
\newcommand\ba{\boldsymbol{\alpha}}
\newcommand\bb{\boldsymbol{\beta}}
\newcommand\bc{\boldsymbol{\gamma}}

\newcommand\ssm{\smallsetminus}

\newcommand\HFKh{\widehat{\mathit{HFK}}}

\newcommand\SHM{\mathit{SHM}}
\newcommand\SHMt{\underline{\SHM}}
\newcommand\KHMt{\underline{\mathit{KHM}}}
\newcommand\SFH{\mathit{SFH}}
\newcommand\HMtoc{\HMto_{\bullet}}
\newcommand\HF{\mathit{HF}}
\newcommand\HM{\mathit{HM}}

\newcommand\maxtb{\overline{tb}}
\newcommand\maxsl{\overline{sl}}

    \RequirePackage{rotating}                   
    \def\HMto{%
       \setbox0=\hbox{$\widehat{\mathit{HM}}$}
       \setbox1=\hbox{$\mathit{HM}$}
       \dimen0=1.1\ht0
       \advance\dimen0 by 1.17\ht1
       \smash{\mskip2mu\raise\dimen0\rlap{%
          \begin{turn}{180}
              {$\widehat{\phantom{\mathit{HM}}}$}
           \end{turn}} \mskip-2mu    
                \mathit{HM}
    }{\vphantom{\widehat{\mathit{HM}}}}{}}

\newtheorem{theorem}{Theorem}[section]
\newtheorem{lemma}[theorem]{Lemma}

\newtheorem{corollary}[theorem]{Corollary}
\newtheorem{proposition}[theorem]{Proposition}

\theoremstyle{definition}

\newtheorem{remark}[theorem]{Remark}

\begin{document}
\begin{abstract} 
We recently defined an invariant of contact manifolds with convex boundary in Kronheimer and Mrowka's sutured monopole Floer homology theory. Here, we prove that there is an isomorphism between sutured monopole Floer homology and sutured Heegaard Floer homology which identifies our invariant with the contact class defined by Honda, Kazez and Mati{\'c} in the latter theory.  One consequence is that the Legendrian invariants in knot Floer homology behave functorially with respect to Lagrangian concordance. In particular, these invariants provide computable and effective obstructions to the existence of such concordances. Our work also provides the first proof which does not rely on Giroux's correspondence that  Honda, Kazez and Mati{\'c}'s contact class is well-defined up to isomorphism.
\end{abstract}

\maketitle

\section{Introduction}
\label{sec:intro}

The purpose of this article is to establish an equivalence between two invariants of contact 3-manifolds with boundary---one defined using Heegaard Floer homology and the other using monopole Floer homology. Our equivalence fits naturally into the  ongoing program of establishing connections between the many different instantiations of Floer theory. In addition to  the  theoretical appeal of  such connections,   an equivalence between invariants in different Floer homological  settings allows one to combine  the intrinsic advantages of  the  different settings, often with interesting topological or geometric consequences. This principle is illustrated nicely by Taubes's isomorphism between monopole Floer homology and embedded contact homology \cite{taubes1,taubes2,taubes3,taubes4,taubes5}, whose first step, a correspondence between monopoles and Reeb orbit sets, proved the Weinstein conjecture for 3-manifolds \cite{taubeswc}. 

Our work provides another  illustration of this principle.
One of the primary advantages of Heegaard Floer homology is its computability. On the other hand, monopole Floer homology enjoys a certain functoriality  with respect to exact symplectic cobordism which has yet to be proven in the Heegaard Floer setting. The equivalence described in this paper enables us to combine these advantages to give a new, computable obstruction to the existence of Lagrangian concordances between Legendrian knots, a subject of much recent interest. Another application of our equivalence is a Giroux-correspondence-free proof that the contact invariant in sutured Heegaard Floer homology is well-defined up to isomorphism. 

Below, we describe our equivalence and its applications in more detail. We then outline the proof. We work in characteristic 2 throughout this paper.

\subsection{Our equivalence}

Let us first recall  the   invariants of \emph{closed} contact 3-manifolds defined by Kronheimer and Mrowka and by Ozsv{\'a}th and Szab{\'o}. Suppose $(Y,\xi)$ is a closed contact 3-manifold, and $\eta$ is an oriented curve in $Y$. In \cite{km,kmosz}, Kronheimer and Mrowka associate  to such data a class \[c_{\HM}(\xi)\in\HMtoc(-Y,\spc_\xi;\Gamma_{\eta})\footnote{In \cite{km, kmosz}, this class is denoted by $\psi$.}\] in the monopole Floer homology of $-Y$ with local coefficients. Here, $\spc_\xi$ is the $\Sc$ structure on $-Y$ associated with $\xi$, and $\Gamma_{\eta}$ is a local  system on the  monopole Floer configuration space with fiber a Novikov ring $\Lambda$. In \cite{osz1}, Ozsv{\'a}th and Szab{\'o}   likewise define a class \[c_{\HF}(\xi)\in\hfp(-Y,\spc_\xi;\Gamma_{\eta})\] in  Heegaard Floer homology with local coefficients, but by very different means. Remarkably, these two invariants are equivalent. This is made precise in the theorem below, which follows from Taubes' work \cite{taubes1,taubes2,taubes3,taubes4,taubes5} together with work of Colin, Ghiggini, and Honda \cite{cgh3,cgh4,cgh5} on the isomorphism between Heegaard Floer homology and embedded contact homology.

\begin{theorem}[Taubes, Colin--Ghiggini--Honda]
\label{thm:closedequiv}
For every $\spc\in\Sc({-}Y)$, there is an isomorphism of $\Lambda[U]$-modules
\[\Phi_\spc:\hfp({-}Y,\spc;\Gamma_{\eta})\to\HMtoc(-Y,\spc;\Gamma_{\eta}),\]
such that $\Phi_{\spc_\xi}(c_{\HF}(\xi))=c_{\HM}(\xi)$.
\end{theorem}

\begin{remark}
 Kutluhan, Lee, and Taubes also proved this isomorphism \cite{klt1,klt2,klt3,klt4,klt5}.  We need to specifically use the Taubes and Colin-Ghiggini-Honda isomorphisms in this work, however, because they identify the contact invariants $c_{\HF}(\xi)$ and $c_{\HM}(\xi)$ while Kutluhan-Lee-Taubes do not.
\end{remark}

This article sets out to establish a similar equivalence for invariants of contact 3-manifolds \emph{with  boundary}, or, more precisely, what we call \emph{sutured contact manifolds}. These are triples $(M,\Gamma,\xi)$ where $(M,\xi)$ is a contact 3-manifold with convex boundary and  dividing set $\Gamma\subset \partial M$. In \cite{hkm4}, Honda, Kazez, and Mati{\'c} associate to such data a class \[c_{\HF}(\xi)\in\SFH(-M,-\Gamma)\footnote{In \cite{hkm4}, this class is denoted by $EH$.}\] in the sutured Heegaard Floer homology of $(-M,-\Gamma)$ which, in a sense, generalizes Ozsv{\'a}th and Szab{\'o}'s invariant of closed contact manifolds (it generalizes the \emph{hat} version of Ozsv{\'a}th and Szab{\'o}'s invariant). In \cite{bsSHM}, we gave a similar generalization of Kronheimer and Mrowka's invariant. Ours assigns to a sutured contact manifold  a class 
\[c_{\HM}(\xi)\in\SHMt(-M,-\Gamma)\footnote{In \cite{bsSHM}, this class is denoted by $\psi$.}\] in a version of sutured monopole Floer homology with local coefficients.\footnote{In fact, our invariant can be made to take values in the ``natural" refinement of $\SHMt$  defined in \cite{bs3}.} Our main theorem is the following, settling a conjecture made in \cite[Conjecture 1.9]{bsSHM}.

\begin{theorem}[Main Theorem]
\label{thm:main}
There is an isomorphism of $\Lambda$-modules
\[\SFH(-M,-\Gamma)\otimes \Lambda\to\SHMt(-M,-\Gamma),\]
sending $c_{\HF}(\xi)\otimes 1$ to $c_{\HM}(\xi)$.
\end{theorem}

\subsection{Applications}

The main topological application of Theorem \ref{thm:main} discussed in this paper is to the study of Lagrangian concordance initiated by Chantraine in \cite{chantraine}. Recall that for Legendrian knots $K_1,K_2\subset (Y,\xi)$, $K_1$ is \emph{Lagrangian concordant} to $K_2$---written $K_1\prec K_2$---if there is an embedded Lagrangian cylinder $C$ in the symplectization of $(Y,\xi)$ such that \begin{align*}C\,\cap\, ((-\infty,-T)\times Y) &=  (-\infty,-T)\times K_1,\\
C\,\cap\, ((T,\infty)\times Y) &=  (T, \infty)\times K_2\end{align*}
 for some $T>0$. Two Legendrian knots related by Lagrangian concordance must have the same classical invariants (Thurston-Bennequin and rotation numbers) \cite{chantraine}. A challenging problem, therefore, which has attracted a lot of recent attention, is to find tools for deciding whether two knots with the same classical invariants are Lagrangian concordant. Note that this is more difficult than the already formidable task of deciding whether two smoothly isotopic knots with the same classical invariants are Legendrian isotopic, though many known Legendrian isotopy invariants are in fact Lagrangian concordance obstructions, see e.g.\ \cite{cdrgg-gokova}.

In \cite{bsLeg}, we defined a Legendrian invariant  which assigns to a Legendrian knot $K\subset (Y,\xi)$ a class \[\mathcal{L}_{\HM}(K)\in \KHMt(-Y,K)\] in  monopole knot Floer homology with local coefficients. It is defined in terms of a certain contact structure $\xi_K$ on the sutured knot complement $Y(K)$ with two meridional sutures,
\begin{equation}\label{eqn:leg}\mathcal{L}_{\HM}(K):=c_{\HM}(\xi_K)\in \SHMt(-Y(K))=:\KHMt(-Y,K).\end{equation}
Furthermore, we used the functoriality  of $c_{\HM}$ under exact symplectic cobordism  (a feature whose analogue  in Heegaard Floer homology has not been established independently of Theorem \ref{thm:closedequiv}) to show that $\mathcal{L}_{\HM}$ behaves functorially under Lagrangian concordance, as follows.

\begin{theorem}[Baldwin--Sivek]
\label{thm:functcon}
If $K_1,K_2\subset (Y,\xi)$ are Legendrian knots with $K_1\prec K_2$, then there is a map \[\KHMt(-Y,K_2)\to \KHMt(-Y,K_1),\] sending $\mathcal{L}_{\HM}(K_2)$ to $\mathcal{L}_{\HM}(K_1).$
\end{theorem}

In this way, the class $\mathcal{L}_{\HM}$  provides an obstruction to the existence of Lagrangian concordances between Legendrian knots. Unfortunately, this class is not easily computable. A much more computable Legendrian invariant is that defined by Lisca, Ozsv{\'a}th, Stipsicz, and Szab{\'o} in \cite{lossz}. Theirs   takes the form of a class \[\mathcal{L}_{\HF}(K)\subset \HFKh(-Y,K)\] in Heegaard knot Floer homology. Though originally defined in terms of open books for $(Y,\xi)$, Stipsicz and V{\'e}rtesi discovered in \cite{sv} that it can also be formulated as
\[\mathcal{L}_{\HF}(K)=c_{\HF}(\xi_K)\in \SFH(-Y(K))=\HFKh(-Y,K).\] In fact, their work was the inspiration for our subsequent definition of $\mathcal{L}_{\HM}$. The equivalence below then follows  immediately from Theorem \ref{thm:main}.

\begin{theorem}
\label{thm:equivleg}
There is an isomorphism of $\Lambda$-modules
\[\HFKh(-Y,K)\otimes\Lambda\to\KHMt(-Y,K),\]
sending $\mathcal{L}_{\HF}(K)\otimes 1$ to $\mathcal{L}_{\HM}(K)$. \qed
\end{theorem}

\begin{remark}
The second author defined a similar but different invariant of Legendrian knots in $\KHMt$ in \cite{sivek}. In fact, it was his construction that inspired our series \cite{bsSHM, bsSHI, bsLeg}. It is still unknown how his Legendrian invariant is related to $\mathcal{L}_{\HM}$ and therefore to $\mathcal{L}_{\HF}$.
\end{remark}

A Legendrian invariant is said to be \emph{effective} if it can distinguish smoothly isotopic Legendrian knots with the same classical invariants. The invariant $\mathcal{L}_{\HF}$ is effective in that there are Legendrian knots as above for which the invariant vanishes for one but not for the other.  Theorem \ref{thm:equivleg} then implies that $\mathcal{L}_{\HM}$ is effective as well, resolving \cite[Conjecture 1.1]{bsLeg}.

\begin{theorem}
The invariant $\mathcal{L}_{\HM}$ is effective. \qed
\end{theorem}

\begin{remark}
The classes $\mathcal{L}_{\HF}$ and $\mathcal{L}_{\HM}$ are  invariant under negative Legendrian stabilization, and therefore provide  invariants of transverse knots as well. Theorem \ref{thm:equivleg}, combined with computations in Heegaard Floer homology \cite{lossz}, implies that $\mathcal{L}_{\HM}$ is also an effective transverse knot invariant, in the sense that it can distinguish smoothly isotopic transverse knots with the same self-linking numbers.
\end{remark}

An even more striking consequence of Theorems \ref{thm:functcon} and \ref{thm:equivleg} is that the invariant $\mathcal{L}_{\HF}$ is also functorial under Lagrangian concordance, as follows.

\begin{theorem}
\label{thm:functhf}
If $K_1,K_2\subset (Y,\xi)$ are Legendrian with $K_1\prec K_2$, then there is a map \[\HFKh(-Y,K_2)\otimes\Lambda\to \HFKh(-Y,K_1)\otimes\Lambda,\] sending $\mathcal{L}_{\HF}(K_2)\otimes 1$ to $\mathcal{L}_{\HF}(K_1)\otimes 1.$ \qed
\end{theorem}

\begin{remark}We remark that decorated smooth concordances induce similar maps on knot Floer homology, as defined and studied by Juh\'asz \cite{juhasz-cobordisms} and Juh\'asz-Marengon \cite{juhasz-marengon}.  It is not clear how those maps are related to ours.
\end{remark}

Once again, the value of establishing this functoriality in the Heegaard Floer setting has to do with the relative computability of invariants in that setting. In fact, before the discovery of $\mathcal{L}_{\HF}$,  Ozsv{\'a}th, Szab{\'o}, and Thurston defined in \cite{oszt} an intrinsically computable invariant of Legendrian knots in the tight contact structure $(S^3,\xi_{std})$ using the grid diagram formulation of knot Floer homology. Their invariant assigns to a Legendrian $K\subset (S^3,\xi_{std})$ a class 
\[\Theta_{\HF}(K)\in \HFKh(-S^3,K).\] It was  shown in \cite{bvv} that these two Heegaard Floer invariants are equivalent where they overlap, per the following theorem.

\begin{theorem}[Baldwin--Vela-Vick--V{\'e}rtesi] 
\label{thm:equivleginv}For any Legendrian knot $K\subset (S^3,\xi_{std}),$ there is an automorphism \[\HFKh(-S^3,K)\to\HFKh(-S^3,K),\] sending $\mathcal{L}_{\HF}(K)$ to $\Theta_{\HF}(K)$.
\end{theorem}

Combined with Theorem \ref{thm:functhf}, this theorem implies that the invariant $\Theta_{\HF} = \mathcal{L}_{\HF}$ provides an entirely computable obstruction to the existence of Lagrangian concordances between Legendrian knots in $(S^3,\xi_{std})$, as follows. 

\begin{theorem}
\label{thm:grid-concordance}
If $K_1$ and $K_2$ are Legendrian knots in $(S^3,\xi_{std})$ with  \[\Theta_{\HF}(K_1) \neq 0\,\textrm{ and }\,\Theta_{\HF}(K_2) = 0,\] then there is no Lagrangian concordance from $K_1$ to $K_2$. \qed
\end{theorem}
As mentioned by the authors of \cite{bvv}, proving a result like Theorem \ref{thm:grid-concordance} was a major part of their motivation for establishing the equivalence described in Theorem \ref{thm:equivleginv}. 

It is easy to find examples demonstrating the effectiveness of the obstruction in Theorem \ref{thm:grid-concordance}. In particular, there are infinitely many pairs  $(K_1,K_2)$ of smoothly isotopic Legendrian knots with the same classical invariants which satisfy the hypotheses of Theorem \ref{thm:grid-concordance} \cite{ost, not, kng, bald9}.   The results of \cite{ost} imply that such $K_1$ and $K_2$ are not Legendrian isotopic, whereas Theorem \ref{thm:grid-concordance} implies the much stronger fact that $K_1$ is not Lagrangian concordant to $K_2$.

It bears mentioning that the Legendrian contact homology  DGA of Chekanov and Eliashberg \cite{yasha8} enjoys a similar sort of functoriality under Lagrangian concordance \cite{ehk}. However, it can be    difficult to apply  this DGA obstruction in practice. Consider, for example, the two Legendrian representatives $K_1$ and $K_2$ of $m(10_{132})$ with $(tb,r) = (1,0)$ described by Ng, Ozsv{\'a}th, and Thurston in \cite{not}. One can show that $K_1$ is not Lagrangian concordant to $K_2$ by showing that the DGA is trivial for $K_2$ while nontrivial for $K_1$. But proving this nontriviality is tricky as the DGA for $K_1$ does not even admit any nontrivial finite-dimensional representations \cite{sivek2}. By contrast, it is quite easy to check that $\Theta_{\HF}$ vanishes $K_2$ but not for $K_1$, and in so doing, apply the Heegaard Floer obstruction in Theorem \ref{thm:grid-concordance}.

Another advantage of $\Theta_{\HF}$ is that it is preserved under negative Legendrian stabilization, whereas the Legendrian contact homology DGA is trivial for stabilized knots. In particular, for $K_1$ and $K_2$ satisfying the hypotheses of  Theorem \ref{thm:grid-concordance}, we may  also conclude   that no negative stabilization of $K_1$ is Lagrangian concordant to any negative stabilization of $K_2$. The DGA, by contrast, cannot tell us anything about Lagrangian concordances between stabilized knots. 

In Section \ref{sec:concex}, we give another demonstration of  our obstruction, providing several additional  examples of Legendrian knots with the same classical invariants which are  not smoothly isotopic or Lagrangian concordant, but which \emph{are smoothly concordant}. In these examples, Lagrangian concordance is obstructed by Theorem \ref{thm:grid-concordance} while the Legendrian contact homology DGA provides no such obstruction.

Another important application of our work concerns the well-definedness of Honda, Kazez, and Mati{\'c}'s contact invariant. Given a sutured contact manifold $(M,\Gamma,\xi)$ and a  partial open book $\ob$ compatible with $\xi$, Honda, Kazez, and Mati{\'c}  define an element \[c_{\HF}(\ob)\in\SFH(-M,-\Gamma). \] They then prove that the elements associated to any two  open books compatible with $\xi$ agree, and define $c_{\HF}(\xi)$ to be this common element. Their proof that this class is  independent of the choice of   open book relies on Giroux's correspondence---in particular, on its assertion that any two open books compatible with $\xi$ are related by  positive stabilizations and destabilizations.\footnote{Henceforth, any mention of ``Giroux's correspondence" will refer to this assertion; proofs can be found in the literature for the other parts of the correspondence.} 
By contrast, our construction of $c_{\HM}(\xi)$ does not rely on this assertion.  

Moreover, in proving Theorem \ref{thm:main}, what we actually show (see Theorem \ref{thm:mainob}),  again without  using Giroux's correspondence, is that for any  partial open book $\ob$ compatible with $\xi$ there is a $\Lambda$-module isomorphism \begin{equation}\label{eqn:isoob}\SFH(-M,-\Gamma)\otimes \Lambda\to \SHMt(-M,-\Gamma)\end{equation}  sending $c_{\HF}(\ob)\otimes 1$ to $c_{\HM}(\xi)$. Our work thus gives a proof which does not rely on Giroux's correspondence that the elements  associated to any  two partial open books compatible with  $\xi$ are related by an automorphism of $\SFH(-M,-\Gamma)$; in other words, that $c_{\HF}(\xi)$ is well-defined up to isomorphism.
While our well-definedness statement is  weaker than that of Honda, Kazez, and Mati{\'c} (see the remark below), the value of our proof lies in the fact that a complete proof of Giroux's correspondence has yet to appear.

\begin{remark}
We do not claim to have given a Giroux-correspondence-free proof that $c_{\HF}(\xi)$ is well-defined \emph{as an element of} $\SFH(-M,-\Gamma)$, the point being that we do not know whether the isomorphism in \eqref{eqn:isoob} is independent of the choice of partial open book.  However, the question of whether $c_{\HF}(\xi)$ vanishes does not depend on the particular isomorphism, and it is often only this vanishing or non-vanishing that is used in applications.
\end{remark}

\subsection{Proof outline}

We outline our proof of Theorem \ref{thm:main} below following  a  very brief review of sutured monopole Floer homology and  our contact invariant.

A \emph{closure} of a balanced sutured manifold $(M,\Gamma)$, as defined by Kronheimer and Mrowka in \cite{km4}, is a  closed manifold $Y$ together with a distinguished surface $R\subset  Y$, formed from $(M,\Gamma)$ in a certain  manner, and containing $M$ as a submanifold. Let \[\Sc(Y|R) := \{\spc\in\Sc(Y)\mid\langle c_1(\spc),[R]\rangle = 2g(R)-2\}\] denote the set of ``top" $\Sc$ structures on $Y$ with respect to $R$. The sutured monopole Floer homology of $(M,\Gamma)$ is defined as \[\SHMt(M,\Gamma):=\HMtoc(Y| R; \Gamma_\eta):=\bigoplus_{\spc\in\Sc(Y|R)}\HMtoc(Y,\spc;\Gamma_\eta),\] where $\eta$ is a curve in $R$ of a certain form.  Given a sutured contact manifold $(M,\Gamma,\xi)$, we give a procedure in \cite{bsSHM} for extending $\xi$ to a contact structure $\bar\xi$ on  a certain class of closures $(Y,R)$ with respect to which $R$ is convex and such that  \begin{equation}\label{eqn:evalY}\langle c_1(\spc_{\bar\xi}),[R]\rangle = 2-2g(R).\end{equation} For a certain class of  $\eta\subset R$ as above,  we refer to the quadruple $(Y,R,\bar\xi,\eta)$ as a \emph{contact closure} of $(M,\Gamma,\xi)$. The  pairing  in \eqref{eqn:evalY} implies that \[\spc_{\bar\xi}\in\Sc({-}Y|{-}R),\]  which means that  $\HMtoc(-Y,\spc_{\bar\xi};\Gamma_\eta)$ is a direct summand of $\SHMt(-M,-\Gamma)$. We  define \[c_{\HM}(\xi):=c_{\HM}(\bar\xi)\in \HMtoc(-Y,\spc_{\bar\xi};\Gamma_\eta)\subset \SHMt(-M,-\Gamma),\] and prove that this class is independent of the  choices involved in its construction. 
Note that Theorem \ref{thm:closedequiv} provides an isomorphism
\[ \xymatrix@C=12pt@R=-2pt{
\hfp(-Y| {-}R; \Gamma_\eta)\ar[rr]^{}&& \HMtoc(-Y| {-}R; \Gamma_\eta)\,\,\, \ar@{}[r]|{=:}&\SHMt(-M,-\Gamma) \\
\rotatebox{90}{$\in$} && \rotatebox{90}{$\in$} & \rotatebox{90}{$\in$} \\
\,\,\,c_{\HF}(\bar\xi)\,\,\ar@{|->}[rr] &&\,\,\, c_{\HM}(\bar\xi)\,\ar@{}[r]|{=:}&c_{\HM}(\xi).
}\] 
Therefore, in order to prove Theorem \ref{thm:main}, it suffices to prove the following. 

\begin{theorem}
\label{thm:main2}
There is a contact closure $(Y,R,\bar\xi,\eta)$ of $(M,\Gamma,\xi)$ for which there exists an isomorphism of $\Lambda$-modules \[ A:\SFH(-M,-\Gamma)\otimes \Lambda \to \hfp(-Y| {-}R; \Gamma_\eta),\] sending 
$c_{\HF}(\xi)\otimes 1$ to $c_{\HF}(\bar\xi)$.
\end{theorem}

Our strategy for proving Theorem \ref{thm:main2}  makes use of an interesting topological reformulation of the contact invariant of $\bar\xi$ from \cite{bsSHM}. One starts with a partial open book for $(M,\Gamma,\xi)$, which provides a description of this contact manifold as obtained from an $[-1,1]$-invariant contact structure $\xi_S$ on the product sutured manifold \[H(S) = (S\times[-1,1],\partial S\times\{0\})\] by attaching contact 2-handles along certain curves $s_1,\dots,s_n$ in $\partial H(S)$. These curves correspond naturally to Legendrians  in a contact closure $(Y_S,R,\bar\xi_S, \eta)$ of $(H(S),\xi_S)$, and we proved that contact (+1)-surgery on these Legendrian curves results in a contact closure $(Y,R,\bar\xi,\eta)$ of $(M,\Gamma,\xi)$. It then follows from the functoriality of the contact invariant under such surgeries \cite[Theorem 4.2]{osz1} that the map \[B:\hfp(-Y_S|{-}R;\Gamma_\eta)\to\hfp(-Y|{-}R;\Gamma_\eta)\] induced by the natural 2-handle cobordism corresponding to these surgeries  satisfies \[ B(c_{\HF}(\bar\xi_S))=c_{\HF}(\bar\xi).\]  The contact class $c_{\HF}(\bar\xi_S)$ is always nonzero and the domain of $B$,  \[\hfp(-Y_S|{-}R;\Gamma_\eta)\cong \Lambda, \] is 1-dimensional, so this   class may be characterized simply as the generator of this module. We prove that if the initial contact closure $(Y_S,R, \bar\xi_S,\eta)$ is of a certain form, where, in particular, $g(R)$ is sufficiently large, then there is a $\Lambda$-module isomorphism as claimed in the theorem, and
the above characterization of $c_{\HF}(\bar\xi_S)$
enables us
to show that \begin{equation}\label{eqn:A=B}A(c_{\HF}(\xi)\otimes 1)= B(c_{\HF}(\bar\xi_S)).\end{equation} Since the latter equals $c_{\HF}(\bar\xi)$,  this proves Theorem \ref{thm:main2} and therefore Theorem \ref{thm:main}.

It bears mentioning that Lekili has already shown in \cite{lekili2} that the modules in Theorem \ref{thm:main2} are isomorphic. Given a sutured Heegaard diagram for $(-M,-\Gamma)$, Lekili constructs a pointed Heegaard diagram for $-Y$ in a certain  natural way, and, by comparing these diagrams, defines a quasi-isomorphism between the corresponding chain complexes.  Our map $A$ is defined using a similar, but slightly different diagram for ${-}Y$. A novel aspect of our construction is that, for $g(R)$ sufficiently large and for sufficient \emph{winding} of the curves in the Heegaard diagram, we are able to show that $A$ is a chain map and  quasi-isomorphism without resorting to the somewhat involved holomorphic disk analysis that appears in Lekili's proof. A similar principle, applied to counting holomorphic triangles, is used to prove the equality in \eqref{eqn:A=B}.

\subsection{Organization}
In Section \ref{sec:background}, we review the constructions and properties of the contact invariants in Heegaard and monopole Floer homologies and their sutured variants. In Section \ref{sec:proof}, we prove Theorem \ref{thm:main} as outlined above. In Section \ref{sec:concex}, we provide examples which further illustrate the effectiveness of Theorem \ref{thm:grid-concordance} in obstructing Lagrangian concordances.

\subsection{Acknowledgements} We thank Ko Honda for helpful correspondence.  We also thank the referees for several careful readings and many helpful comments, and especially for pointing out a mistake (twice!)\ in earlier drafts of this article; our corrections led to substantial improvements in the proofs of our main results.

\section{Background}
\label{sec:background}
\subsection{Sutured monopole Floer homology and contact invariants} 
\label{ssec:hm} 
Let $\Lambda$ be the Novikov ring over $\zzt$ defined by \[\Lambda=\bigg\{\sum_{\alpha}c_{\alpha}t^{\alpha}\,\,\bigg |\, \,\alpha\in\mathbb{R},\,\,c_{\alpha}\in\zzt,\,\,\#\{\alpha<n|c_{\alpha}\neq 0\}<\infty\textrm{ for all } n\in \zz\bigg\}. \] 
Suppose $Y$ is a closed, oriented 3-manifold and $\eta$ is a smooth 1-cycle in $Y$. Kronheimer and Mrowka  defined in \cite{kmbook,km4} a version of monopole Floer homology with local coefficients which assigns to the pair $(Y,\eta)$ a $\Lambda[U]$-module \[\HMtoc(Y;\Gamma_\eta) = \bigoplus_{\spc\in\Sc(Y)}\HMtoc(Y,\spc;\Gamma_\eta).\] 
Furthermore,  Kronheimer and Mrowka in \cite{km,kmosz} assign to a contact structure $\xi$ on $Y$ a class \[c_{\HM}(\xi)\in\HMtoc(-Y,\spc_\xi;\Gamma_\eta)\subset \HMtoc(-Y;\Gamma_\eta)\footnote{As mentioned in the introduction, this class is denoted by $\psi(\xi)$ in \cite{kmosz}.}\] which depends only on the isotopy class of $\xi$. We note that the construction of $c_{\HM}(\xi)$ does not rely on Giroux's correspondence.

We recall below the definition of sutured monopole Floer homology and our construction of the contact invariant for sutured contact manifolds defined therein.

Suppose $(M,\Gamma)$ is a balanced sutured manifold.  Let $A(\Gamma)$ be a closed tubular neighborhood of $\Gamma$ in $\partial M$, and let $T$ be a compact, connected, oriented surface with $g(T)\geq 2$ and $\pi_0(\partial T)\cong \pi_0(\Gamma)$.  Let \[h:\partial T\times [-1,1]\to A(\Gamma)\] be an orientation-reversing homeomorphism sending $\partial T\times \{\pm 1\}$ to $\partial (R_{\pm}\ssm A(\Gamma))$. Now consider  the \emph{preclosure} \[P = M\cup_h T\times [-1,1]\] formed by gluing $T\times[-1,1]$ according to $h$. The balanced condition ensures that $P$ has two homeomorphic boundary components, $\partial_+P$ and $\partial_-P$, given by\[
\partial_{\pm}P= (R_{\pm}\ssm A(\Gamma)) \cup T\times\{\pm1\}.\] One can then glue $\partial_+P$ to $\partial_-P$ by an orientation-reversing homeomorphism to form a closed, oriented 3-manifold $Y$ containing a distinguished surface \[R:=\partial_+P= -\partial_-P\subset Y.\] In \cite{km4}, Kronheimer and Mrowka define a \emph{closure} of $(M,\Gamma)$ to be any pair $(Y,R)$ obtained in this manner. They refer to $T$ as the \emph{auxiliary surface} used to form this closure.

\begin{remark}
If $(Y,R)$ is a closure of  $(M,\Gamma)$, then $(-Y,-R)$ is a closure of $(-M,-\Gamma)$.
\end{remark}

 \begin{remark}
 \label{rmk:altclosure}
It is sometimes useful to think of $Y$ as obtained by gluing  $R\times[-1,1]$ to $P$, by a map which identifies $R\times\{\pm 1\}$ with $\partial_{\mp}P$, and $R$ as $R\times\{0\}$. In particular, from this perspective, $\partial M$ is a  codimension 1 submanifold of $Y$.
\end{remark}

Suppose $(Y,R)$ is a closure of $(M,\Gamma)$ formed as above, and fix an oriented  curve $\eta\subset R$ which is dual to a homologically essential curve in the auxiliary surface $T$. As in the introduction, we let  \begin{equation}
\label{eqn:topSC}\Sc(Y|R) := \{\spc\in\Sc(Y)\mid \langle c_1(\spc),[R]\rangle = 2g(R)-2\}\end{equation} be the set of ``top" $\Sc$ structures on $Y$ with respect to $R$, and define the sutured monopole Floer homology of $(M,\Gamma)$ to be  the $\Lambda$-module  \[\SHMt(M,\Gamma):=\HMtoc(Y| R; \Gamma_\eta):=\bigoplus_{\spc\in\Sc(Y|R)}\HMtoc(Y,\spc;\Gamma_\eta).\]   Indeed, Kronheimer and Mrowka prove in \cite[Proposition~4.6]{km4} that the isomorphism class of this module is independent of the choice of closure $(Y,R)$ and  curve $\eta$, and is therefore an invariant of $(M,\Gamma)$. We later proved in \cite{bs3} that the modules assigned to different closures are related by \emph{canonical} isomorphisms.

Suppose now that $(M,\Gamma,\xi)$ is a sutured contact manifold. Let $(Y,R)$ be a closure of $(M,\Gamma)$ formed by  gluing on a thickened auxiliary surface $T\times[-1,1]$  to form a preclosure $P$, as  described above, and then gluing $\partial_+P$ to $\partial_-P$ by a map which sends  $c\times\{+1\}$ to $c\times\{-1\}$ for some nonseparating  curve $c\subset T$. In \cite[Section~3]{bsSHM}, we gave a procedure for extending $\xi$ to a contact structure $\bar \xi$ on   $Y$ such that $R$ is convex with respect to $\bar\xi$ with \[\langle c_1(\spc_{\bar\xi}),[R]\rangle = 2-2g(R).\] For any curve $\eta\subset T\subset R$ dual to $c$, we refer to the quadruple $(Y,R,\bar\xi,\eta)$ as a \emph{contact closure} of $(M,\Gamma,\xi)$. The above pairing implies that \[\spc_{\bar\xi}\in\Sc(-Y|{-}R).\footnote{We incorrectly computed this pairing in \cite{bsSHM} to be $2-2g(R)$; this error does not affect the validity of any of the results in that paper.}\]  It therefore makes sense to define the element  \[c_{\HM}(\xi):=c_{\HM}(\bar\xi)\in \HMtoc(-Y,\spc_{\bar\xi};\Gamma_\eta)\subset \HMtoc(-Y| {-}R; \Gamma_\eta)=:\SHMt(-M,-\Gamma).\] We proved in \cite{bsSHM} that this element is independent of the choices involved in its construction. Our proof does not rely on Giroux's correspondence.

\subsection{Partial open books for sutured contact manifolds}
\label{ssec:partialob}
We  assume the reader is familiar with (non-partial) open books for  closed contact manifolds. 
Following \cite[Definition 4.9]{bsSHM},  a  \emph{partial open book} is a quadruple $(S,P,h,\{c_1,\dots,c_n\})$, where: 
\begin{itemize}
\item $S$ is a surface with nonempty boundary, 
\item $P$ is a subsurface of $S$,  
\item $h:P\to S$ is an embedding which restricts to the identity on $\partial P\cap \partial S$, 
\item $\{c_1,\dots,c_n\}$ is a set of disjoint, properly embedded arcs in $P$ whose complement in $S$ deformation retracts onto $S\ssm P$.
\end{itemize}

\begin{remark}
\label{rmk:basisindep}
The collection $\{c_1,\dots,c_n\}$ is called a \emph{basis} for the partial open book. It is often not included as part of the definition since the sutured contact manifold compatible with the partial open book, described below, is independent of the basis.
\end{remark}

Given a partial open book \[\ob=(S,P,h,\{c_1,\dots,c_n\}), \]let $\xi_S$ be the $[-1,1]$-invariant contact structure on $S\times[-1,1]$ for which each $S\times\{t\}$ is convex with  Legendrian boundary and  dividing set consisting of one boundary-parallel arc for each component of $\partial S$, oriented in the direction of $\partial S$. Let $(H(S),\xi_S)$ denote the  sutured contact manifold  obtained from $(S\times[-1,1],\xi_{S})$ by rounding corners, as illustrated in Figure \ref{fig:productsurfacecorners} below. In particular,    the dividing set on $\partial H(S)$ is isotopic to $\partial S\times \{0\}$. 
Let $s_i$ be the curve on $\partial H(S)$ given by\begin{equation}\label{eqn:basishandle}s_i=(c_i\times\{1\})\cup (\partial c_i\times [-1,1])\cup (h(c_i)\times\{-1\}).\end{equation}  (In a slight abuse of notation, we identify $H(S)$ with $(S\times[-1,1],\xi_S)$, ignoring corner rounding.) 
We say that the partial open book $\ob$ is  \emph{compatible} with the sutured contact manifold $(M,\Gamma,\xi)$ if the latter can be obtained  from $(H(S),\xi_S)$ by attaching contact $2$-handles along the curves $s_1,\dots,s_n$.  Honda, Kazez, and Mati{\'c} proved the following in \cite[Theorem~1.3]{hkm4}.

\begin{theorem}[Honda--Kazez--Mati{\'c}]
\label{thm:relativegiroux1}
Every sutured contact manifold admits a compatible partial open book decomposition.
\end{theorem}

\begin{remark}
A (non-partial) open book for a closed contact $3$-manifold can be thought of as a partial open book in which $P$ is the complement of a disk in $S$. The corresponding closed contact $3$-manifold is formed by attaching contact $2$-handles to $(H(S),\xi_S)$ as above and then filling the resulting $S^2$ boundary with a Darboux ball.
\end{remark}

\begin{figure}[ht]
\centering
\includegraphics[width=10cm]{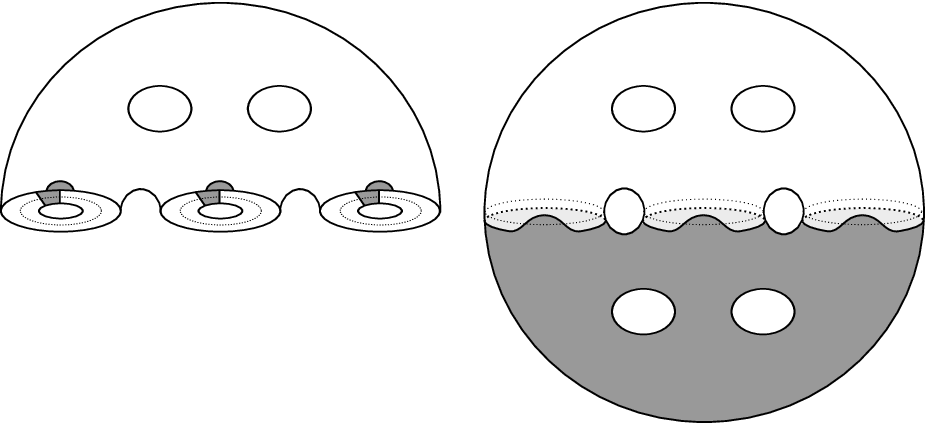}
\caption{Left, $(S\times[-1,1],\xi_{S})$,  with  negative region shaded, for a genus $2$ surface with $3$ boundary components. Right, the convex boundary of the sutured contact manifold $(H(S),\xi_S)$ obtained by rounding corners.}
\label{fig:productsurfacecorners}
\end{figure}

\subsection{Sutured Heegaard Floer homology and contact invariants}
\label{ssec:sfh}
.....................................................................................................................................................................................................................................................................................................................
To define the sutured Heegaard Floer homology of a balanced sutured manifold $(M,\Gamma)$, as introduced by Juh{\'a}sz in \cite{juhasz2}, one starts with an admissible sutured Heegaard diagram \[(\Sigma,\boldsymbol\alpha = \{\alpha_1,\dots,\alpha_n\},\boldsymbol\beta = \{\beta_1\dots,\beta_n\})\] for $(M,\Gamma)$. In particular, 
\begin{itemize}
\item $\Sigma$ is a compact surface with boundary, 
\item $M$ is obtained from $\Sigma\times[-1,1]$ by attaching 3-dimensional 2-handles along the curves $\alpha_i\times\{-1\}$ and $\beta_i\times\{1\}$, for $i= 1,\dots,n$, and 
\item $\Gamma$ is given by $ \partial \Sigma\times\{0\}$. 
\end{itemize} 
The admissibility condition means  that every nontrivial periodic domain has both positive and negative coefficients.

The sutured Heegaard Floer complex $\sfc(\Sigma,\boldsymbol\alpha,\boldsymbol\beta)$ is the $\zzt$ vector space  generated by intersection points \[\mathbf{x}\in\Ta\cap \Tb = (\alpha_1\times\dots\times\alpha_n)\cap (\beta_1\times\dots\times\beta_n)\subset \operatorname{Sym}^n(\Sigma).\]  The differential  is defined by counting holomorphic disks in the usual way; namely, for a generator $\bx$ as above,   \[d\bx =\sum_{\by\in\Ta\cap\Tb}\,\sum_{\substack{\phi\in\pi_2(\bx,\by)\\\mu(\phi)=1}} \#\big(\mathcal{M}(\phi)/\mathbb{R}\big)\cdot \by, \] where $\pi_2(\bx,\by)$ is the set of homotopy classes of Whitney disks from $\bx$ to $\by$;  $\mu(\phi)$ refers to the Maslov index of $\phi$; and $\mathcal{M}(\phi)$ is the moduli space of pseudoholomorphic representatives of $\phi$. The sutured Heegaard Floer homology of $(M,\Gamma)$ is the homology \[\sfh(M,\Gamma):=H_*(\sfc(\Sigma,\boldsymbol\alpha,\boldsymbol\beta))\] of this complex.

Suppose now that $(M,\Gamma,\xi)$ is a sutured contact manifold and that \[\ob=(S,P,h,\{c_1,\dots,c_n\})\] is a  partial open book compatible with $\xi$. Let $\Sigma$ be the surface formed  by  attaching $1$-handles $H_1,\dots,H_n$ to $S$, where the feet of $H_i$ are attached along  the endpoints of $c_i$.
Orient $\Sigma$ so that the induced orientation on $S$ as a subsurface of $\Sigma$ is opposite the given orientation on $S$.
For  $i=1,\dots,n$, let $\alpha_i$ and $\beta_i$ be embedded curves in $\Sigma$ such that: 
\begin{itemize}
\item $\alpha_i$ is the union of $c_i$ with a core of $H_i$, and

\item $\beta_i$ is the union of $h(c_i)$ with a core of $H_i$.
\end{itemize}
We require that these curves intersect  in the region $H_i\subset \Sigma$ in the manner shown in  Figure \ref{fig:handlecurves}.  Then $(\Sigma,\boldsymbol\beta,\boldsymbol\alpha)$ is an admissible sutured Heegaard diagram for $(-M,-\Gamma)$, as observed in \cite[Section~2]{hkm4}. In particular, we note the following for later use.

\begin{remark}
\label{rmk:orientations}
Since \[R_+(-M,-\Gamma)=R_+(M,\Gamma),\] we note that the oriented surface $\Sigma$ agrees with $R_+(\Gamma)$ in the identification of $(-M,-\Gamma)$ with the sutured manifold specified by $(\Sigma,\boldsymbol\beta,\boldsymbol\alpha)$.
\end{remark}

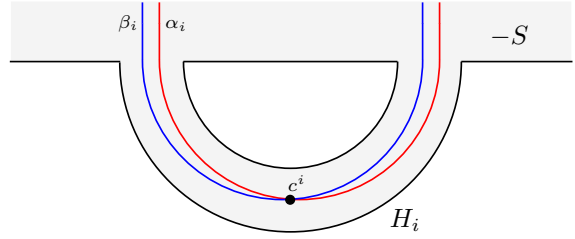
\begin{figure}[ht]
\labellist
\tiny
\pinlabel $\beta_i$ at 57 101
\pinlabel $\alpha_i$ at 80 100
\pinlabel $c^i$ at 138 24
\small
\pinlabel $H_i$ at 190 6
\pinlabel $-S$ at 240 95

\endlabellist
\centering
\includegraphics[width=6.5cm]{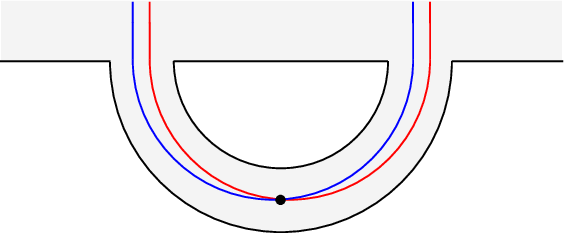}
\caption{The handle $H_i$, the curves $\alpha_i$ and $\beta_i$, and the intersection point $c^i$.}
\label{fig:handlecurves}
\end{figure}

For each $i=1,\dots,n$, let $c^i$ be the  intersection point between $\alpha_i$ and $\beta_i$ in $H_i$,  and define \begin{equation}\label{eqn:defc}\mathbf{c}= \{c^1,\dots,c^n\}\in \Tb\cap \Ta \subset \operatorname{Sym}^n(\Sigma).\end{equation} This generator is a cycle in the  complex $\sfc(\Sigma,\boldsymbol\beta,\boldsymbol\alpha)$, and Honda, Kazez, and Mati{\'c} define \[c_{\HF}(\ob):=[\mathbf{c}]\in \sfh(-M,-\Gamma)\] in \cite{hkm4}. They prove that this class is independent of the basis of the open book. They  moreover prove that if $\ob'$ is obtained from $\ob$ via positive stabilization then \[c_{\HF}(\ob)=c_{\HF}(\ob').\]  Giroux's correspondence then implies that the classes associated to any two partial open books compatible with $\xi$ are equal. Accordingly, Honda, Kazez, and Mati{\'c} define \[c_{\HF}(\xi):=c_{\HF}(\ob)\] for any partial open book compatible with $\xi$.

\subsection{Heegaard Floer homology with local coefficients and contact invariants}
\label{ssec:hflocal}

Suppose $Y$ is a closed, oriented 3-manifold and $\eta$ is a smooth 1-cycle in $Y$. To define the Heegaard Floer homology of $Y$  in a $\Sc$ structure $\spc\in \Sc(Y)$ with local coefficient system associated with $\eta$, one starts with an weakly $\spc$-admissible pointed Heegaard diagram  \[(\Sigma,\boldsymbol\alpha = \{\alpha_1,\dots,\alpha_n\},\boldsymbol\beta=\{\beta_1,\dots,\beta_n\},z)\] for $Y$. This admissibility condition means that every nontrivial periodic domain $P$ with \[\langle c_1(\spc), [{P}]\rangle=0\] has a negative multiplicity, where $[{P}]$ refers to the class in $H_2(Y)$ represented by $P$. We may view $\eta$ as a (possibly non-embedded) curve on $\Sigma$. The chain complex \[\cfp(\Sigma,\boldsymbol\alpha,\boldsymbol\beta,z,\spc;\Gamma_\eta)\] is the $\Lambda[U]$-module \[\bigoplus_{\substack{\bx \in T_\alpha \cap T_\beta\\
\spc_z(\bx)=\spc}} \bigg(\frac{\Lambda[U,U^{-1}]}{U\cdot \Lambda[U]}\bigg)\langle\bx\rangle,\] 
  generated by intersection points $\bx\in\Ta\cap \Tb$ whose associated  $\Sc$ structure $\spc_z(\bx)$ equals $\spc$. For ease of notation, we will adopt the convention in \cite{osz8} and use  $[\bx,i]$ to denote $U^{-i}\bx$, which then vanishes for $i<0$. The differential is defined on such a pair by \[ \partial([\bx,i]) =\sum_{\substack{\phi\in\pi_2(\bx,\by)\\\mu(\phi)=1}} \#\big(\mathcal{M}(\phi)/\mathbb{R}\big)\cdot [\by,i-n_z(\phi)]\cdot t^{\partial_\alpha(\phi)\cdot \eta}, \]  and extended linearly with respect to multiplication in $\Lambda$. Here,  $\partial_\alpha(\phi)\cdot \eta$ refers to the oriented intersection in $\Sigma$ of the $\boldsymbol\alpha$ portion of the boundary of the domain of $\phi$ in $\Sigma$ with the curve $\eta$. The  Heegaard Floer homology of $Y$ is the homology \[\hfp(Y,\spc;\Gamma_\eta):=H_*(\cfp(\Sigma,\boldsymbol\alpha,\boldsymbol\beta,z,\spc;\Gamma_\eta),\partial)\] of this complex. This module is an invariant of the  class $[\eta]\in H_1(Y)$. Given an embedded surface $R\subset Y$ in addition to $\eta$, we define \[\hfp(Y|R;\Gamma_\eta):=\bigoplus_{\spc\in\Sc(Y|R)}\hfp(Y,\spc;\Gamma_\eta),\] just as for monopole Floer homology. Heegaard Floer homology obeys the following adjunction inequality \cite[Corollary 7.2]{osz14}. 
    
\begin{theorem}[Ozsv{\'a}th--Szab{\'o}]
\label{thm:adj3}
If $Z$ is a connected, embedded surface in $Y$ and $\hfp(Y,\spc;\Gamma_\eta)$ is nonzero  then \begin{equation*}
 \label{eqn:adj3}|\langle c_1(\spc),[Z] \rangle| \leq \max\{0,-\chi(Z)\}.\end{equation*} 
\end{theorem}

Suppose now that $(Y,\xi)$ is a closed contact 3-manifold and $\eta$ is a smooth 1-cycle in $Y$. To define the associated Ozsv{\'a}th-Szab{\'o} contact invariant,  one chooses a (non-partial) open book $\ob$ compatible with $(Y,\xi)$ and assigns to this open book a class \[c_{\HF}(\ob)\in\hfp(-Y,\spc_\xi;\Gamma_\eta),\] following \cite{osz1,hkm4}. As in the partial open book case, this class is independent of the basis and is invariant under positive stabilization. Giroux's correspondence then implies that this class is an invariant of the contact structure $\xi$. Accordingly, Ozsv{\'a}th and Szab{\'o} define \[c_{\HF}(\xi):=c_{\HF}(\ob)\] for any open book $\ob$ compatible with $\xi$. 
This  invariant is functorial with respect to contact $(+1)$-surgery, as follows from \cite[Theorem~4.2]{osz1}.\footnote{This theorem was originally stated with coefficients in $\zzt$, but the proof works in the setting of local coefficients just as well.}

\begin{theorem}[Ozsv{\'a}th--Szab{\'o}]
\label{thm:contactsurgeryhf}
Suppose $(Y',\xi')$ is the result of contact $(+1)$-surgery on a Legendrian link $L\subset (Y,\xi)$ disjoint from $\eta$. Let $W$ be the corresponding 2-handle cobordism, obtained from $Y\times[0,1]$ by attaching  contact $(+1)$-framed 2-handles along $L\times\{1\}$, and let   $\nu\subset W$ be the  cylinder $\nu= \eta\times [0,1]$. Then there are open books $\ob$ and $\ob'$ compatible with $\xi$ and $\xi'$, respectively, such that the induced map \[\hfp(-W;\Gamma_{\nu}):\hfp(-Y;\Gamma_{\eta})\to\hfp(-Y';\Gamma_\eta)\] sends $c_{\HF}(\ob)$ to $c_{\HF}(\ob')$.
\end{theorem}

The following version of Theorem \ref{thm:closedequiv} relates the Heegaard Floer and monopole Floer contact invariants. This version of the theorem does not rely on Giroux's correspondence.

\begin{theorem}[Taubes, Colin--Ghiggini--Honda]
\label{thm:closedequiv2}
For every $\spc\in\Sc({-}Y)$ and every open book $\ob$ compatible with $\xi$, there is an isomorphism of $\Lambda[U]$-modules
\[\Phi_\spc:\hfp({-}Y,\spc;\Gamma_{\eta})\to\HMtoc(-Y,\spc;\Gamma_{\eta}),\]
such that $\Phi_{\spc_\xi}(c_{\HF}(\ob))=c_{\HM}(\xi)$.
\end{theorem}

\subsection{Reformulating the invariant of a contact closure}
\label{ssec:reform}
We explain below a reformulation of the invariant of a contact closure which  will be critical in our proof of Theorem \ref{thm:mainob2} (a strong version of Theorem \ref{thm:main2}) in the next section.

Suppose $(M,\Gamma,\xi)$ is a sutured contact manifold with compatible partial open book \[(S,P,h,\{c_1,\dots,c_n\}).\] Suppose $(Y_S,R,\bar\xi_S,\eta)$ is a contact closure of the  sutured contact manifold $(H(S),\xi_S)$ defined in Subsection \ref{ssec:partialob}. Adopting the perspective of Remark \ref{rmk:altclosure}, we may view the curves $s_1,\dots,s_n$ defined in \eqref{eqn:basishandle} as embedded curves in $Y_S$  disjoint from $R$. After  small perturbation, we may assume that these $s_i$ are Legendrian with respect to $\bar\xi_S$, via the Legendrian Realization Principle \cite{kanda, honda2}. Each $s_i$ intersects the dividing set  of $\partial H(S)$ in two places, which implies that the $\partial H(S)$-framing on $s_i$ is one more than its contact framing.  In \cite[Section~4.2.3]{bsSHM}, we proved that the result of contact $(+1)$-surgeries on the resulting Legendrian link \[\mathbb{L}=s_1\cup \dots\cup s_n\] is a contact closure $(Y,R,\bar\xi,\eta)$ of $(M,\Gamma,\xi)$. We further showed that 
\begin{equation*}\label{eqn:productlambda}\HMtoc(-Y_S|{-}R;\Gamma_\eta)\cong \Lambda\end{equation*} is generated by the contact class $c_{\HM}(\bar\xi_S)$.  Theorem \ref{thm:closedequiv2} then implies that the same is true in Heegaard Floer homology.  Specifically,  \begin{equation*}\label{eqn:productlambda2}\hfp(-Y_S|{-}R;\Gamma_\eta)\cong \Lambda\end{equation*} is generated by the contact class of any open book compatible with $\bar\xi_S$.

\begin{remark}
One does not need Theorem \ref{thm:closedequiv2} to prove that this Heegaard Floer module is $1$-dimensional or that the contact class is nonzero for some open book for $\bar\xi_S$; one does need the theorem, however, to see that the contact class is nonzero for \emph{any} open book without appealing to Giroux's correspondence.
\end{remark}

Let $W$ be the 2-handle cobordism from $Y_S$ to $Y$ corresponding to the  surgery on $\mathbb{L}$. Note that $R\subset Y_S$ and $R\subset Y$ are homologous (in fact, isotopic) in $W$. Therefore, letting \begin{equation}\label{eqn:topSCW}\Sc(-W|{-}R):=\{\spt\in\Sc(-W)\mid\langle c_1(\spt),[-R]\rangle = 2g(R)-2\}\end{equation} denote the set of ``top" $\Sc$ structures on $-W$ with respect to $-R$, we have that the  cobordism map in Theorem \ref{thm:contactsurgeryhf} restricts to a map \begin{equation}\label{eqn:cobmapbkgnd}\hfp(-W|{-}R;\Gamma_{\nu}):\hfp(-Y_S|{-}R;\Gamma_{\eta})\to\hfp(-Y|{-}R;\Gamma_\eta),\end{equation} given by   \[\hfp(-W|{-}R;\Gamma_{\nu}):=\sum_{\spt\in\Sc(-W|{-}R)} \hfp(-W,\spt;\Gamma_{\nu}).\]  By Theorem \ref{thm:contactsurgeryhf} and the discussion above, we have the following reformulation of the contact invariant of $\bar\xi$. 
\begin{corollary}
\label{cor:functgen}
There exists an open book $\bar\ob$ compatible with $\bar\xi$ such that \[c_{\HF}(\bar\ob)=\hfp(-W|{-}R;\Gamma_{\nu})(\mathbf{1}),\] where $\mathbf{1}$ refers to a generator of $\hfp(-Y_S|{-}R;\Gamma_{\eta}) \cong \Lambda$.
\end{corollary}

\subsection{Cobordism maps in Heegaard Floer homology}
\label{ssec:cob}
We recall below  the construction of the map on Heegaard Floer homology induced by a 2-handle cobordism of the sort in Theorem \ref{thm:contactsurgeryhf}, as  will need it in the next section.

Suppose $\mathbb{L}$ is a framed link in $Y$ disjoint from an embedded curve $\eta\subset Y$. Let $W$ be the cobordism obtained from $Y\times[0,1]$ by attaching $2$-handles along $\mathbb{L}\times\{1\}$ and  let $\nu\subset W$ be the cylinder $\nu = \eta\times[0,1].$ To define the map on Heegaard Floer homology induced by the cobordism \[({-}W,\nu):({-}Y,\eta)\to ({-}Y',\eta)\]  in the $\Sc$ structure $\spt\in\Sc({-}W)$, one starts with a weakly $\spt$-admissible pointed Heegaard triple diagram  \[(\Sigma,\boldsymbol\gamma,\boldsymbol\beta,\boldsymbol\alpha,z)\]  for  $-W$ which is left-subordinate to the framed link $\mathbb{L}\subset -Y$, as  in \cite[Section 5.2]{osz5}. This admissibility condition means that   every nontrivial triply-periodic domain $P$ which is a sum  of doubly-periodic domains and satisfies \[\langle c_1(\spt),[P]\rangle =0\] has a negative multiplicity. Note that if the pointed triple diagram is weakly $\spt$-admissible then the induced diagrams for $Y_{\gamma\beta}$, $Y_{\beta\alpha}$, and $Y_{\gamma\alpha}$ are weakly admissible for the restrictions of $\spt$ to these $3$-manifolds.

For  this triple diagram, we have that $Y_{\gamma\beta}$ is a connected sum of copies of $S^1\times S^2$,
$Y_{\beta\alpha} = -Y$, and $Y_{\gamma\alpha} = -Y'$, and  there is an intersection point $\Theta\in \Tc\cap\Tb$ such that $[\Theta,0]$ is the unique generator of \[\cfp(\Sigma,\boldsymbol\gamma,\boldsymbol\beta,z;\Gamma_\eta)\] in the top Maslov grading among generators killed by $U$.
The map 
\[\hfp(-W,\spt;\Gamma_\nu): \hfp(-Y,\spt|_{{-}Y};\Gamma_\eta)\to \hfp(-Y',\spt|_{{-}Y'};\Gamma_\eta)\]  is induced by the chain map \begin{equation*}\label{eqn:f+}f^+_{\gamma\beta\alpha,\spt;\Gamma_\nu}:\cfp(\Sigma,\boldsymbol\beta,\boldsymbol\alpha,z,\spt|_{Y_{\beta\alpha}};\Gamma_\eta)\to\cfp(\Sigma,\boldsymbol\gamma,\boldsymbol\alpha,z,\spt|_{Y_{\gamma\alpha}};\Gamma_\eta),\end{equation*} defined on $[\bx,i]$ by \[ f^+_{\gamma\beta\alpha,\spt;\Gamma_\nu}([\bx,i]) =\sum_{\by\in\Tc\cap\Ta}\,\sum_{\substack{\phi\in\pi_2(\Theta,\bx,\by)\\\spc_z(\phi)=\spt\\\mu(\phi)=0}} \#\mathcal{M}(\phi)\cdot [\by,i-n_z(\phi)]\cdot t^{\partial_\alpha(\phi)\cdot \eta}, \] where $\pi_2(\Theta,\bx,\by)$ is the set of homotopy classes of Whitney triangles with vertices at $\Theta,\bx,\by$,  $\mathcal{M}(\phi)$ is the moduli space of holomorphic representatives of $\phi$, and $\spc_z(\phi)$ is the $\Sc$ structure on $-W$ represented by $\phi$. This map is an invariant of the  class $[\nu]\in H_2(-W,-\partial W).$ Ozsv{\'a}th and Szab{\'o} prove in \cite[Theorem 3.3]{osz5} that for each element $x\in\hfp(-Y;\Gamma_\eta)$, \[\hfp(-W,\spt;\Gamma_\nu)(x)=0\] for all but  finitely $\spt\in\Sc(-W)$. Furthermore, they prove the following adjunction inequality \cite[Proof of Theorem 1.5]{osz5}. 

\begin{theorem}[Ozsv{\'a}th--Szab{\'o}]
\label{thm:adj4}
If $Z$ is a connected, embedded surface in $W$ with nonnegative self-intersection and the map $\hfp(-W,\spt;\Gamma_\nu)$ is nonzero  then \begin{equation*}\label{eqn:adj4}|\langle c_1(\spt),[Z] \rangle| +Z\cdot Z \leq \max\{0,-\chi(Z)\}.\end{equation*}\end{theorem}

\section{Proof of Main Theorem}
\label{sec:proof}
Let $(M,\Gamma,\xi)$ be a sutured contact manifold. 

The goal of this section is to prove the following version of our main theorem,  Theorem \ref{thm:main}, stated in terms of the  classes associated to  partial open books compatible with $\xi$ rather than  $c_{\HF}(\xi)$, so as not to rely on  Giroux's correspondence.

\begin{theorem} 
\label{thm:mainob} For any partial open book $\ob$ compatible with $\xi$, there exists an isomorphism of $\Lambda$-modules \[
 \SFH(-M,-\Gamma)\otimes \Lambda \to \SHMt(-M,-\Gamma),\] sending 
 $c_{\HF}(\ob)\otimes 1$ to $c_{\HM}(\xi)$.

\end{theorem}

This  follows, as outlined in the introduction, from the   version of Theorem \ref{thm:main2} below.

\begin{theorem} 
\label{thm:mainob2}For any partial open book $\ob$ compatible with $\xi$, there is
\begin{itemize}
\item a contact closure $(Y,R,\bar\xi,\eta)$  of $(M,\Gamma,\xi)$ and
\item an open book $\bar\ob$ compatible with $\bar \xi$
\end{itemize}
for which  there exists an isomorphism of $\Lambda$-modules \[
 A:\SFH(-M,-\Gamma)\otimes \Lambda \to \hfp(-Y| {-}R; \Gamma_\eta),\] sending 
 $c_{\HF}(\ob)\otimes 1$ to $c_{\HF}(\bar\ob)$.
\end{theorem}

Let us explain  in more detail how Theorem \ref{thm:mainob} follows from Theorem \ref{thm:mainob2}. For this, suppose $\ob$ is a partial open book compatible with $\xi$. Let $(Y,R,\bar\xi,\eta)$, $\bar\ob$, and $A$ be as in the conclusion of Theorem \ref{thm:mainob2}. Theorem \ref{thm:closedequiv2} provides an isomorphism of $\Lambda$-modules \[\Phi_R:\hfp({-}Y|{-}R;\Gamma_{\eta})\to\HMtoc(-Y|{-}R;\Gamma_{\eta})=:\SHMt(-M,-\Gamma)\] sending $c_{\HF}(\bar\ob)$ to \[c_{\HM}(\bar\xi)=:c_{\HM}(\xi),\] where $\Phi_R$ is the sum over ``top" $\Sc$ structures with respect to $-R$, \[\Phi_R:=\sum_{\spc\in\Sc(-Y|{-}R)}\Phi_s.\] The composition \[\Phi_R\circ A: \SFH(-M,-\Gamma)\otimes \Lambda \to \SHMt(-M,-\Gamma)\] is therefore an isomorphism of $\Lambda$-modules sending $c_{\HF}(\ob)\otimes 1$ to $c_{\HM}(\xi)$, as desired.

It just remains to prove Theorem \ref{thm:mainob2}. We  do so as outlined in the introduction (in particular, we show in Subsection \ref{ssec:facthm} that Theorem \ref{thm:mainob2} follows from Theorems  \ref{thm:iso} and \ref{thm:map}), except that we again take care to talk about the  classes associated to open books rather than to contact structures, in order to make clear that our proof is independent of Giroux's correspondence.

\subsection{Heegaard diagrams for closures and cobordisms}
\label{ssec:hdclosure}
Fix a partial open book \[\ob=(S,P,h,\{c_1,\dots,c_n\})\] compatible with  $(M,\Gamma,\xi)$. Let \[(H(S)=S\times[-1,1],\xi_S)\] be the associated product sutured contact manifold defined in Section \ref{ssec:partialob}, and let \[\mathbb{L}=s_1\cup \dots \cup s_n\] denote the $\partial H(S)$-framed link on the boundary of $H(S)$, where \begin{equation}\label{eqn:si}s_i=(c_i\times\{1\})\cup (\partial c_i\times [-1,1])\cup (h(c_i)\times\{-1\})\end{equation} 
An important ingredient in the proof of Theorem \ref{thm:mainob2} involves understanding the map (denoted in the introduction by $B$) induced by $-W$, where $W$ is the  2-handle cobordism  from a closure of $H(S)$ to a closure   of $(M,\Gamma)$,  corresponding to  surgery on the link $\mathbb{L}$ in the first closure. In order to  understand this map, we first describe pointed  Heegaard diagrams for these closures and a pointed Heegaard triple diagram for this cobordism.

Let $\Sigma$ be the surface formed  by  attaching $1$-handles $H_1,\dots,H_n$ and $H_1',\dots,H_n'$ to $S$, where:
\begin{itemize}
\item the feet of $H_i$ are attached along  the endpoints of $c_i$, \item  the feet of $H_i'$ are attached along  the endpoints of a cocore of $H_i$.
\end{itemize}
We orient $\Sigma$ so that the induced orientation on $S$ as a subsurface of $\Sigma$ is opposite the given orientation on $S$.
For each $i=1,\dots,n$, let $\alpha_i,\beta_i,\gamma_i$ be embedded curves in $\Sigma$ such that: 
\begin{itemize}
\item $\alpha_i$ is the union of $c_i$ with a core of $H_i$,
\item  $\beta_i$ is the union of a cocore of $H_i$ with a core of $H_i'$,
\item $\gamma_i$ is the union of $h(c_i)$ with a core of $H_i$.
\end{itemize}
We require that these curves intersect  in the region $H_i\subset \Sigma$ in the manner shown on the right in Figure \ref{fig:HH2}. 

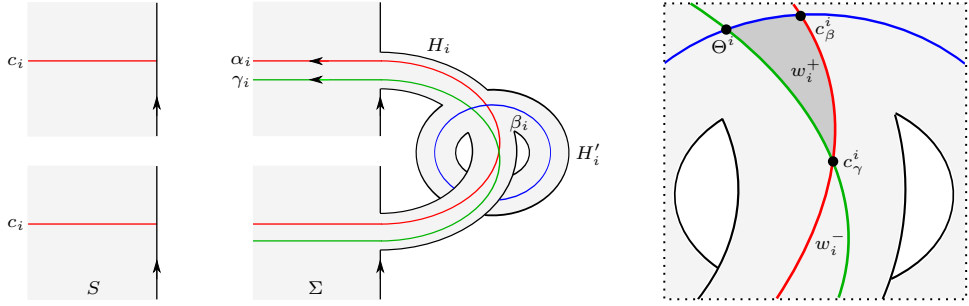
\begin{figure}[ht]
\labellist
\tiny
\pinlabel $S$ at 40 8
\pinlabel $\Sigma$ at 175 8
\pinlabel $c_i$ at -7 46
\pinlabel $c_i$ at -7 145
\pinlabel $\alpha_i$ at 130 146
\pinlabel $\beta_i$ at 299 108
\pinlabel $\gamma_i$ at 130 133

\pinlabel $H_i$ at 250 155
\pinlabel $H_i'$ at 341 89

\pinlabel $\Theta^i$ at 423 156
\pinlabel $c_{\beta\alpha}^i$ at 485 169
\pinlabel $c_{\gamma\alpha}^i$ at 504 84
\pinlabel $w_i^+$ at 473 140
\pinlabel $w_i^-$ at 488 35

\endlabellist
\centering
\includegraphics[width=12.5cm]{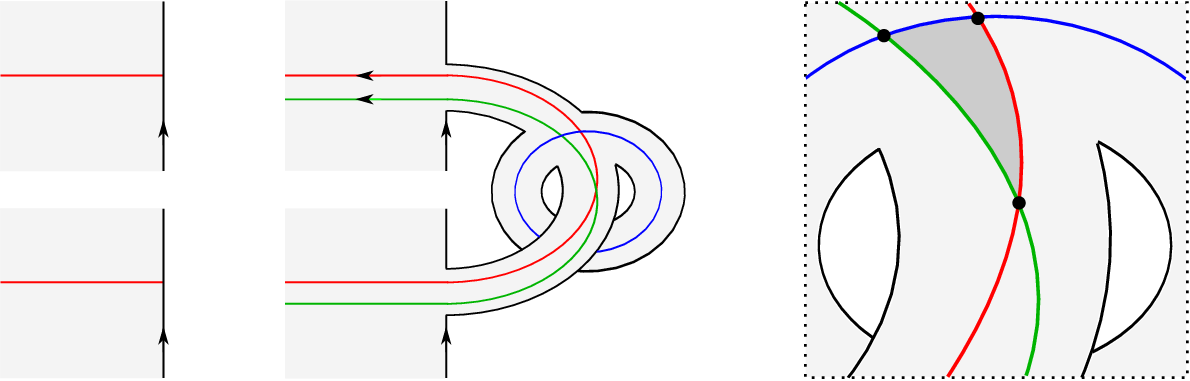}
\caption{On the left, a portion of $S$ near $c_{\gamma\alpha}^i$. In the middle, the corresponding portion of $\Sigma$ with the curves $\alpha_i,\beta_i,\gamma_i$. On the right, a closeup of these curves in $H_i$. We have labeled the intersection points $\Theta^i, c_{\beta\alpha}^i,c_{\gamma\alpha}^i$ and the points $w_i^{\pm}$, and have shaded the triangle $\Delta^i$.}
\label{fig:HH2}
\end{figure}

\begin{remark}
\label{rmk:three}Note that the sutured Heegaard diagram \begin{equation}\label{eqn:diagstab}(\Sigma,\{\beta_1,\dots,\beta_n\},\{\alpha_1,\dots,\alpha_n\})\end{equation} is an $n$-fold stabilization of the standard diagram for $-H(S)$. Meanwhile, the sutured  diagram  \begin{equation}\label{eqn:diagob}(\Sigma, \{\gamma_1,\dots,\gamma_n\},\{\alpha_1,\dots,\alpha_n\})\end{equation} is obtained from the standard Heegaard diagram for $(-M,-\Gamma)$ associated with  the partial open book $\ob$, as described in Subsection \ref{ssec:sfh}, by attaching the handles $H_1',\dots,H_n'$. In particular, it is a sutured Heegaard diagram for the sutured manifold obtained from $(-M,-\Gamma)$ by attaching $n$ contact $1$-handles. We will  ignore this difference, however, and think of the Heegaard diagram in \eqref{eqn:diagob} as encoding $(-M,-\Gamma)$ since (1) there is a canonical isomorphism 
\begin{align*}
H_*(SFC(\Sigma,\{\gamma_1,\dots,\gamma_n\},\{\alpha_1,\dots,\alpha_n\}))&\cong H_*(SFC(\Sigma\ssm (\cup_i H_i'),\{\gamma_1,\dots,\gamma_n\},\{\alpha_1,\dots,\alpha_n\}))\\
&=: SFH(-M,-\Gamma),
\end{align*} and (2) a contact closure of a sutured manifold obtained from $(-M,-\Gamma)$ via contact $1$-handle attachments is also a contact closure of $(-M,-\Gamma)$, as explained in \cite[Section~4.2.2]{bsSHM}.
\end{remark}

We now describe a Heegaard triple diagram which  encodes closures of the sutured manifolds specified by the Heegaard diagrams in \eqref{eqn:diagstab} and \eqref{eqn:diagob} as well as the cobordism $-W$. Let $T$ be a compact, oriented, connected surface with boundary and $g(T)\geq 2$, such that \[\pi_0(\partial T) \cong \pi_0(\partial \Sigma) \cong \pi_0(\partial S).\] Let $R_S$ be the closed, oriented surface formed by gluing $T$ to $S$ by a diffeomorphism of their boundaries. Let $R_\Sigma$ be the  surface formed by gluing $T$ to $\Sigma$ in a similar manner.  Let $D_a$ and $D_b$ be two disjoint disks in $T$. Let $\ul{R}_S$ and $\ul{R}_\Sigma$ denote the complements of these  disks in $R_S$ and $R_\Sigma$, 
\begin{align*}
\ul{R}_S&=R_S\ssm D_a \ssm D_b,\\
\ul{R}_\Sigma&=R_\Sigma\ssm D_a \ssm D_b,
\end{align*} and let \[\ul\Sigma =  \ul{R}_S\cup \ul{R}_\Sigma\] be the closed surface formed by gluing these complements together by the identity maps on $\partial D_a$ and $\partial D_b$. In other words, $\ul\Sigma$ is obtained by connecting $R_S$ and $R_\Sigma$ via two tubes. We will think of the $\alpha_i,\beta_i,\gamma_i$  curves above as lying in $\ul{R}_\Sigma\subset \ul\Sigma.$ See the middle diagram in Figure \ref{fig:annuluseg} for an illustration of $\ul\Sigma = \ul R_S\cup \ul R_\Sigma$ in the case that $S$ is an annulus, $h$ is a right-handed Dehn twist around the core, $\{c_1,\dots,c_n\}$ consists of just the cocore $c_1$, and $T$ is a genus 2 surface with 2 boundary components.
\begin{figure}[ht]
\labellist
\pinlabel $\ul R_\Sigma$ at 220 360
\pinlabel $\ul R_S$ at 640 360
\pinlabel $\longrightarrow$ at 220 900
\pinlabel $=$ at 574 900

\tiny
\pinlabel $c_1$ at 85 957
\pinlabel $h(c_1)$ at -25 905
\pinlabel $S$ at 118 790
\pinlabel $\Sigma$ at 390 790
\pinlabel $\Sigma$ at 718 790

\pinlabel $\Sigma$ at 115 430
\pinlabel $S$ at 730 430
\pinlabel $T\ssm D_a\ssm D_b$ at 565 430
\pinlabel $T\ssm D_a\ssm D_b$ at 285 430
\pinlabel $\partial D_a$ at 390 465
\pinlabel $\partial D_b$ at 390 625

\endlabellist
\centering
\includegraphics[width=11.5cm]{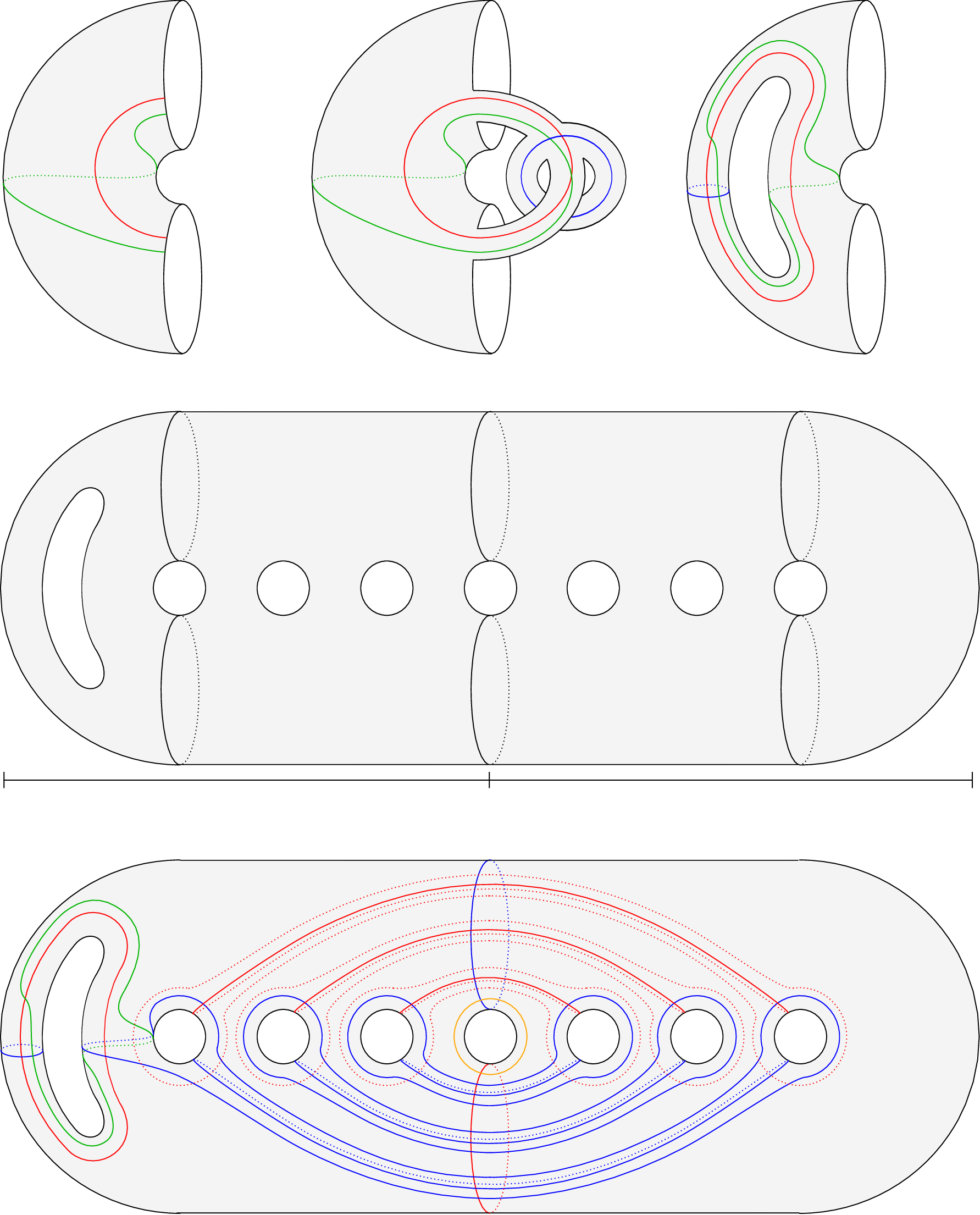}
\caption{Top left, the annulus $S$ with basis arc $c_1$ and its image under the right-handed Dehn twist around the core (it looks like a left-handed Dehn twist because we reverse the orientation of $S$ in forming $\Sigma$). Top middle and right, $\Sigma$ and the curves $\alpha_1,\beta_1,\gamma_1$.  Middle, the corresponding surface $\ul\Sigma =\ul R_\Sigma\cup\ul R_S$. Bottom, the triple diagram $(\ul\Sigma,\gamma,\beta,\alpha)$ without the curves $\gamma_{2},\dots,\gamma_{8}$,  which are just small Hamiltonian translates of $\beta_2,\dots,\beta_8$, and the curve $\delta$ in orange.}
\label{fig:annuluseg}
\end{figure}
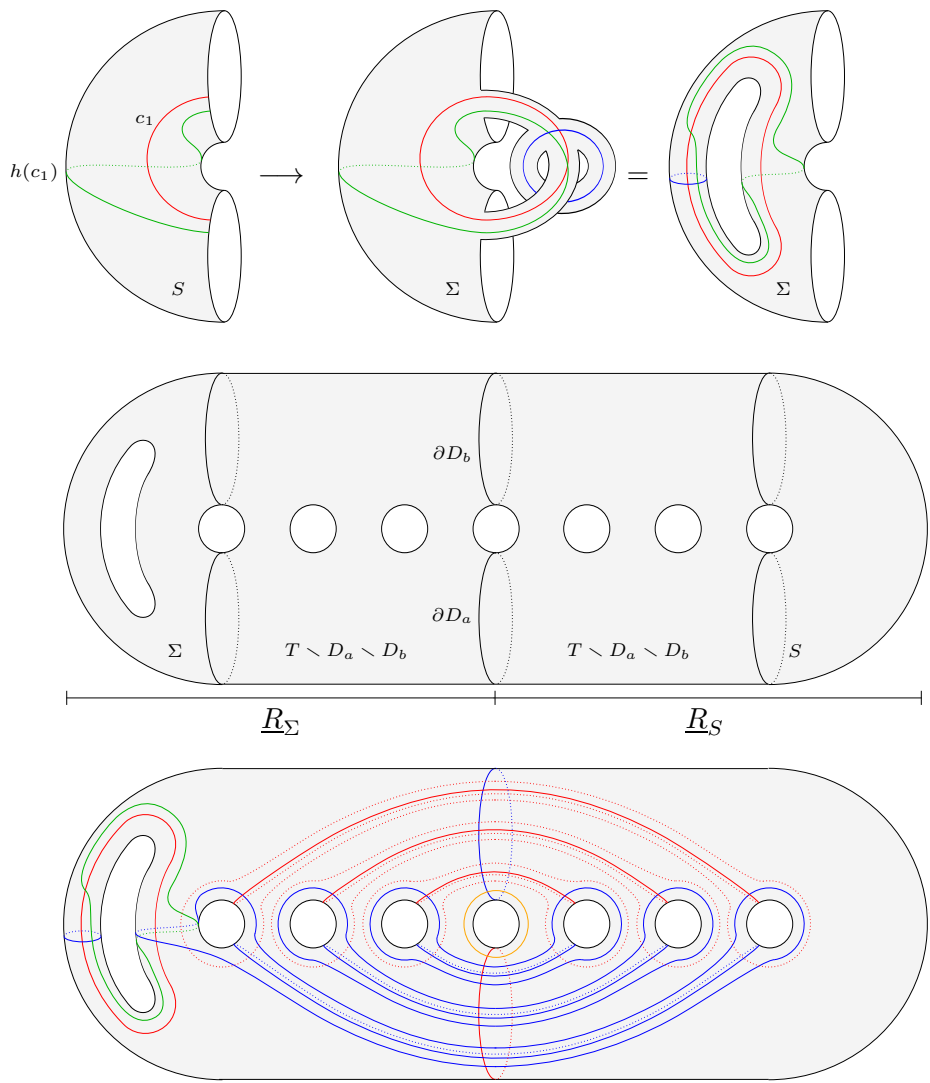

\begin{figure}[ht]
\labellist

\pinlabel $A$ at 3 214
\pinlabel $\ul R_S$ at 400 35
\tiny
\pinlabel $180^\circ$ at 659 234
\pinlabel $a_{n+2g}$ at -5 277
\pinlabel $a_{n+2g-1}$ at -13 292
\pinlabel $a_{n+4}$ at -2 338
\pinlabel $a_{n+3}$ at -2 353
\pinlabel $a_{n+2}$ at -2 382
\pinlabel $a_{n+1}$ at -2 398

\pinlabel $b_{n+2g}$ at -5 155
\pinlabel $b_{n+2g{-}1}$ at -13 140
\pinlabel $b_{n+4}$ at -2 99
\pinlabel $b_{n+3}$ at -2 82
\pinlabel $b_{n+2}$ at -2 50
\pinlabel $b_{n+1}$ at -2 34

\pinlabel $d$ at 103 195

\endlabellist
\centering
\includegraphics[width=11cm]{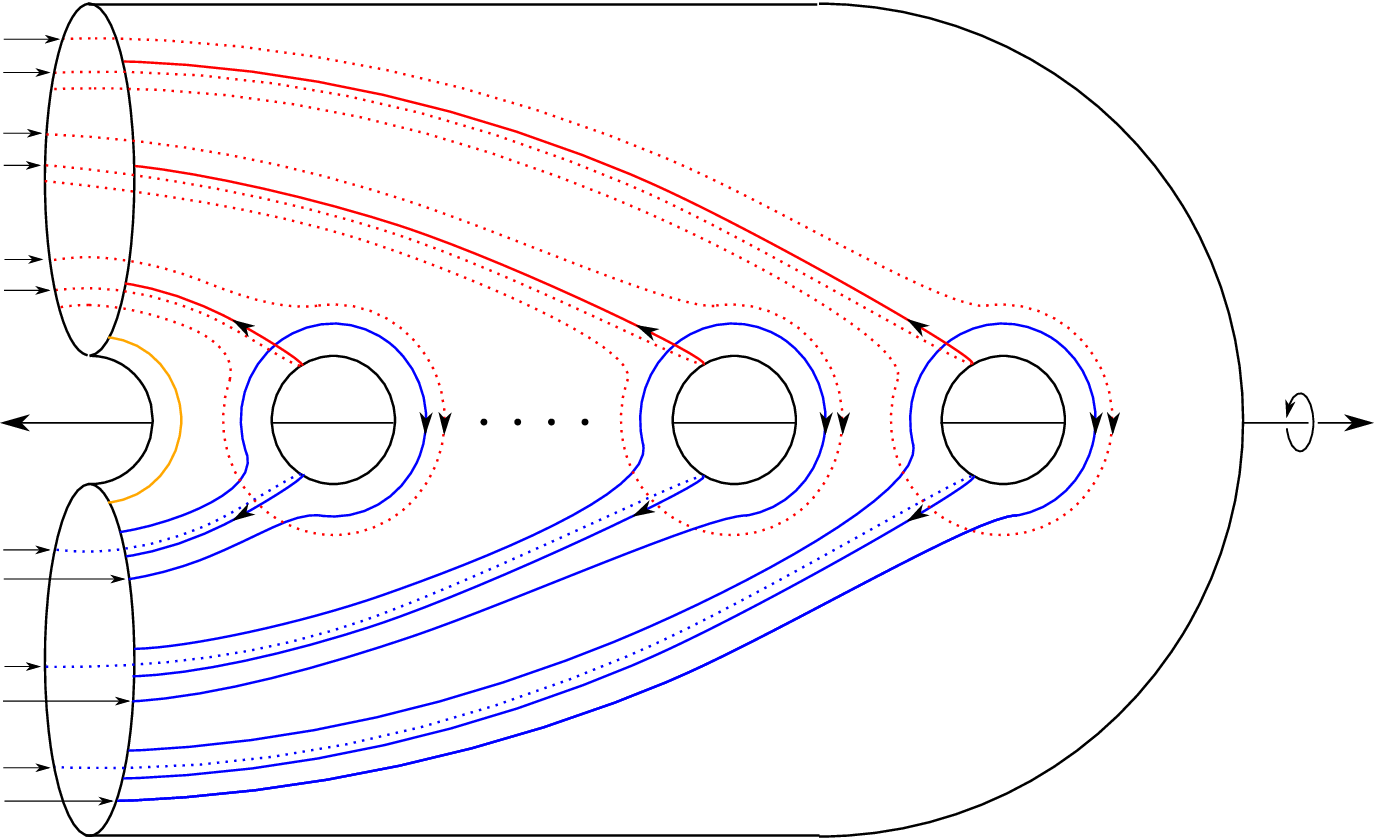}
\caption{The arcs $a_i, b_j$ on $\ul R_S$. The arc $b_i$ is the image of $a_i$ under rotation of $180^\circ$ about the axis $A$. The arc $d$ connects the boundary components of $\ul R_S$.}
\label{fig:TSarcs}
\end{figure}

Suppose $R_S$ has genus $g$. Then $R_\Sigma$ and  $\ul\Sigma$ have genera $n{+}g$ and $n{+}2g{+}1$, respectively. Let \begin{equation}\label{eqn:abarcs}\{a_{n{+}1},\dots,a_{n+2g}\}\textrm{ and }\{b_{n+1},\dots,b_{n+2g}\}\end{equation} be the sets of pairwise disjoint, properly embedded arcs in $\ul R_S$ shown in Figure \ref{fig:TSarcs}. In particular, each $b_i$ is the image of $a_i$ under the  $180^\circ$ rotation of the surface around the axis shown in the figure.  Furthermore, each $b_i$ intersects  $a_{\sigma(i)}$ in exactly one point and is disjoint from $a_j$ for $j\neq \sigma(i)$, where $\sigma$ is the permutation of $\{n{+}1,\dots,n{+}2g\}$ given by 
\begin{equation}
\label{eqn:sigma}\sigma(n{+}i) = 
\begin{cases}
n{+}i{+}1, & i \textrm{  odd}\\
n{+}i{-}1, & i \textrm{  even}\\
\end{cases}
\end{equation}
Note that the $a_i$ have endpoints on $\partial D_b$ and cut $\ul R_S$ into an annulus with  $\partial D_a$ as a boundary component. Likewise, the $b_j$ have endpoints on $\partial D_a$ and cut $\ul R_S$ into an annulus with  $\partial D_b$ as a boundary component. 

Observe that the regular neighborhoods
\begin{align*}
U_i&=N(\alpha_i\cup\beta_i)\subset \Sigma,\\
V_i&=N(\beta_i\cup \gamma_i)\subset \Sigma
\end{align*}
are once-punctured tori. In particular, there exist disks $D_1,\dots,D_n\subset S$ such that \begin{equation}\label{eqn:homo}\ul R_S\ssm (\cup_{i=1}^n D_i)\, \cong\, \ul R_\Sigma \ssm (\cup_{i=1}^n U_i)\, \cong\, \ul R_\Sigma \ssm (\cup_{i=1}^n V_i).\end{equation} We may assume that these $D_i$, as disks in $\ul R_S$, are disjoint from the arcs in \eqref{eqn:abarcs}.
Let 
\begin{align}
\label{eqn:vara}\varphi_\alpha&:  \ul R_S \ssm (\cup_{i=1}^n D_i) \to \ul R_\Sigma\ssm (\cup_{i=1}^n U_i),\\
\label{eqn:varc}\varphi_\gamma&:  \ul R_S \ssm (\cup_{i=1}^n D_i) \to \ul R_\Sigma\ssm (\cup_{i=1}^n V_i)
\end{align}
 be diffeomorphisms which restrict to the identity map on $T\ssm D_a \ssm D_b$. For $i=n{+}1,\dots, n{+}2g$, let $\alpha_i,\beta_i$ be the curves  in $\ul \Sigma= \underline{R}_S \cup \underline{R}_\Sigma$ given by
\begin{align*}
\alpha_i &= a_i \cup \varphi_\alpha(a_i),\\
\beta_i &= b_i\cup \varphi_\gamma(b_i).
\end{align*}
 Let 
 \begin{align*}
 \alpha_{n{+}2g{+}1} &= \partial D_a\subset \ul\Sigma,\\ 
  \beta_{n{+}2g{+}1} &= \partial D_b\subset \ul\Sigma.
  \end{align*} For  $n{+}1,\dots,n{+}2g{+}1$, let $\gamma_i$ be a small Hamiltonian translate  of $\beta_i$  in $\ul \Sigma$ such that $\beta_i$ and $\gamma_i$ intersect in exactly two points, both contained in $\ul R_S\subset \ul \Sigma$. See Figure \ref{fig:TSarcs2} for a closeup  near the intersections of the curves $\alpha_i,\beta_i, \gamma_i$, for $i=n{+}1,\dots,n{+}2g$, on $\ul R_S\subset \ul\Sigma$. Let
\begin{align*}
\ba &= \{\alpha_1,\dots,\alpha_{n+2g+1}\},\\
\bb &=\{\beta_1,\dots,\beta_{n+2g+1}\},\\
\bc &= \{\gamma_1,\dots,\gamma_{n+2g+1}\}.
\end{align*}

We claim the following:
\begin{itemize}
\item $(\ul\Sigma, \boldsymbol\beta, \boldsymbol\alpha)$ is a  Heegaard diagram for a closure $Y_{\beta\alpha} = -Y_S$ of $-H(S)$,
\item $(\ul\Sigma, \boldsymbol\gamma, \boldsymbol\alpha)$ is a  Heegaard diagram for a closure $Y_{\gamma\alpha} = -Y$ of $(-M,-\Gamma)$,\footnote{Really, $-Y_S$ and $-Y$ are the 3-manifolds underlying closures of $-H(S)$ and $(-M,-\Gamma)$; a closure, as defined in Subsection \ref{ssec:hm}, comes with the additional information of a distinguished surface.}
\item $(\ul\Sigma,\boldsymbol\gamma,\boldsymbol\beta,\boldsymbol\alpha)$ is a  Heegaard triple diagram for the cobordism $W_{\gamma\beta\alpha}= -W:-Y_S\to -Y$, where $W:Y_S\to Y$ is the $2$-handle cobordism corresponding to surgery on the framed link $\mathbb{L}\subset \partial H(S)\subset Y_S$.
\end{itemize}
These claims will be unsurprising to the expert. These sorts of Heegaard diagrams for closures are used in Lekili's work \cite{lekili2} and are  very similar to those in Ozsv{\'a}th-Szab{\'o}'s work \cite[Section 3]{osz1}. Nevertheless, we provide an explanation below. We will explain the first of these Heegaard diagrams in depth; the claim regarding the second admits a similar explanation. Along the way, we identify the distinguished surfaces 
\begin{align*}
-R_{\beta\alpha}=-R&\textrm{ in } Y_{\beta\alpha} = -Y_S, \textrm{ and}\\
-R_{\gamma\alpha}=-R&\textrm{ in }Y_{\gamma\alpha} = -Y.
\end{align*}
We will then address the third claim, regarding the Heegaard triple diagram for the cobordism. 

\begin{remark}
As indicated above, we will often use $-R$ to refer to the distinguished surfaces in $-Y_S$ and $-Y$. When more specificity is desired, we will use $-R_{\beta\alpha}$ and $-R_{\gamma\alpha}$, respectively.
\end{remark}

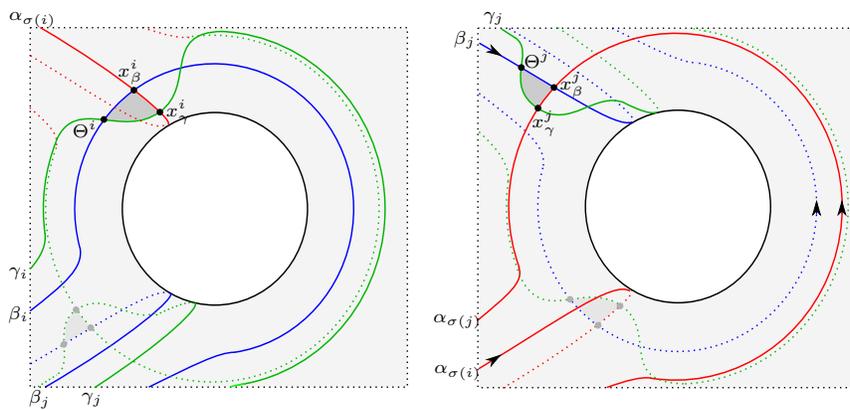
\begin{figure}[ht]
\labellist
\tiny
\pinlabel $\beta_i$ at -5 68
\pinlabel $\alpha_{\sigma(i)}$ at 7 321
\pinlabel $\gamma_i$ at -5 103

\pinlabel $\beta_j$ at 13 -5
\pinlabel $\gamma_j$ at 58 -5

\pinlabel $\alpha_{{\sigma(i)}}$ at 372 18

\pinlabel $\beta_j$ at 378 304
\pinlabel $\alpha_{\sigma(j)}$ at 372 62
\pinlabel $\gamma_j$ at 403 322

\pinlabel $x_{\beta\alpha}^i$ at 96 276
\pinlabel $x_{\gamma\alpha}^i$ at 136 242
\pinlabel $x_{\beta\alpha}^j$ at 476 265
\pinlabel $x_{\gamma\alpha}^j$ at 450 234
\pinlabel $\Theta^i$ at 52 227
\pinlabel $\Theta^j$ at 440 287

\endlabellist
\centering
\includegraphics[width=10cm]{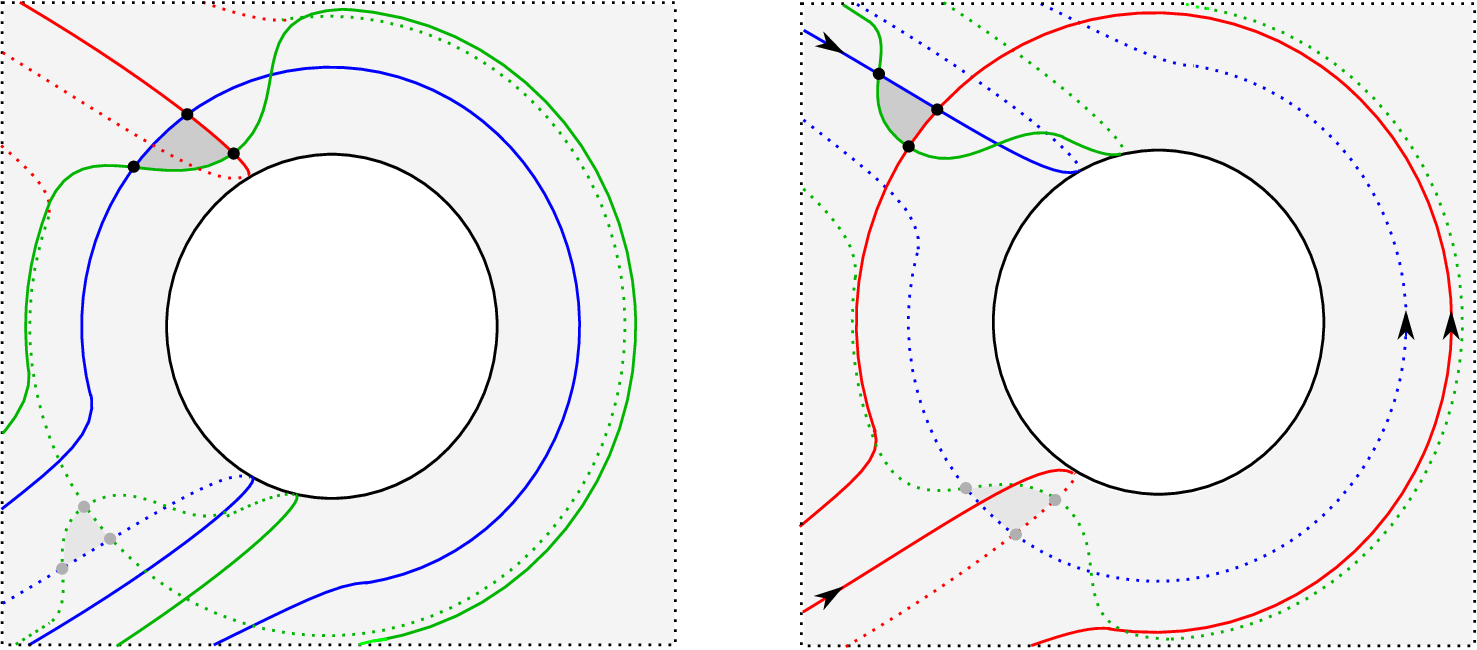}
\caption{Intersections of the $\alpha_i,\beta_i,\gamma_i$ curves for $i=n{+}1,\dots,n{+}2g$. The figures on the left and right show the same curves, after rotating this portion of $\ul R_S$ by $180^\circ$ about the axis $A$ of Figure \ref{fig:TSarcs}. We have labeled the intersection points $\Theta^i, x_{\beta\alpha}^i,x_{\gamma\alpha}^i$ and shaded the triangle $\Delta^i$; likewise for $j=\sigma(i)$. }
\label{fig:TSarcs2}
\end{figure}

Recall from Remark \ref{rmk:three} that \[(\Sigma,\{\beta_1,\dots,\beta_n\},\{\alpha_1,\dots,\alpha_n\})\] is a Heegaard diagram for the sutured manifold $-H(S)$. It follows that \[(R_\Sigma,\{\beta_1,\dots,\beta_n\},\{\alpha_1,\dots,\alpha_n\})\] specifies the preclosure $P$ obtained from $-H(S)$ by attaching $T\times[-1,1]$ in the usual way. That is, $P$ is obtained from $R_\Sigma\times[-1,1]$ by attaching thickened disks to $R_\Sigma\times\{-1\}$ along  $\beta_1,\dots,\beta_n$ and to $R_\Sigma\times\{1\}$ along  $\alpha_1,\dots,\alpha_n$. Let us   denote these unions of thickened disks by $C_{\beta}$ and $C_\alpha$, respectively, so that \[P = R_\Sigma\times[-1,1]\cup C_{\beta} \cup C_\alpha.\] The boundary of $P$ consists of two homeomorphic components, $\partial P = \partial_+P\sqcup -\partial_-P.$ Let \[R_{\beta\alpha} := \partial_+P.\] Per Remark \ref{rmk:altclosure}, 
one may  form a closure $Y_{\beta\alpha}=-Y_S$ of $-H(S)$ by gluing $R_{\beta\alpha}\times[-1,1]$ to $P$, and one may do so in such a way that
\begin{itemize}
\item $D_a\times\{\mp 1\}\subset R_{\beta\alpha}\times\{\mp 1\}$ is identified with $D_a\times\{\pm 1\}\subset \partial_{\pm}P$, and
\item $D_b\times\{\mp 1\}\subset R_{\beta\alpha}\times\{\mp 1\}$ is identified with $D_b\times\{\pm 1\}\subset \partial_{\pm}P$.
\end{itemize} It follows from Remark \ref{rmk:orientations} that the distinguished surface $-R\subset -Y_S$ may    be identified with $-R_{\beta\alpha}$. The diagram on the left in Figure \ref{fig:schematichd} is a schematic illustration of this closure. Note that the disjoint union  $R_\Sigma\times\{0\}\sqcup R_{\beta\alpha}\times\{0\}$ separates $Y_{\beta\alpha}$ into two pieces $V_{\beta}$ and $V_\alpha$, containing $C_{\beta}$ and $C_\alpha$, respectively.  Let $T_\beta$ and $T_\alpha$ be the tubes in $V_\alpha$ and $V_{\beta}$, respectively, defined by
\begin{align*}
T_\beta &= (D_b\times[0,1]\subset R_\Sigma\times[0,1])\cup (D_b\times[-1,0]\subset R_{\beta\alpha}\times[-1,0]),\\
T_\alpha &= (D_a\times[-1,0]\subset R_\Sigma\times[-1,0])\cup (D_a\times[0,1]\subset R_{\beta\alpha}\times[0,1]).
\end{align*}
Then the handlebodies 
\begin{align*}
H_\beta &= \overline{V_\beta\ssm T_\alpha} \cup T_\beta,\\
H_\alpha &= \overline{V_\alpha\ssm T_\beta} \cup T_\alpha
\end{align*} provide a Heegaard splitting of $Y_{\beta\alpha}$, as indicated in the diagram on the right in Figure \ref{fig:schematichd}. The Heegaard surface in this splitting is therefore obtained by connecting $R_\Sigma\times\{0\}$ and  $R_{\beta\alpha}\times\{0\}$ via two tubes. In particular, since $R_{\beta\alpha}\cong R_S$, this Heegaard surface may be identified with $\ul\Sigma$. Under this identification,   one sees that the Heegaard diagram $(\ul \Sigma,\boldsymbol\beta,\boldsymbol\alpha)$ specifies a splitting of precisely this form, where the periodic domains $\ul R_\Sigma$ and $-\ul R_S$ in $\ul \Sigma$ represent $R_{\beta\alpha}.$ Similar reasoning shows that $(\ul\Sigma,\boldsymbol\gamma,\boldsymbol\alpha)$ determines an analogous Heegaard splitting of a closure $Y_{\gamma\alpha}=-Y$ of $(-M,-\Gamma)$, with distinguished surface $-R_{\gamma\alpha}=-R.$

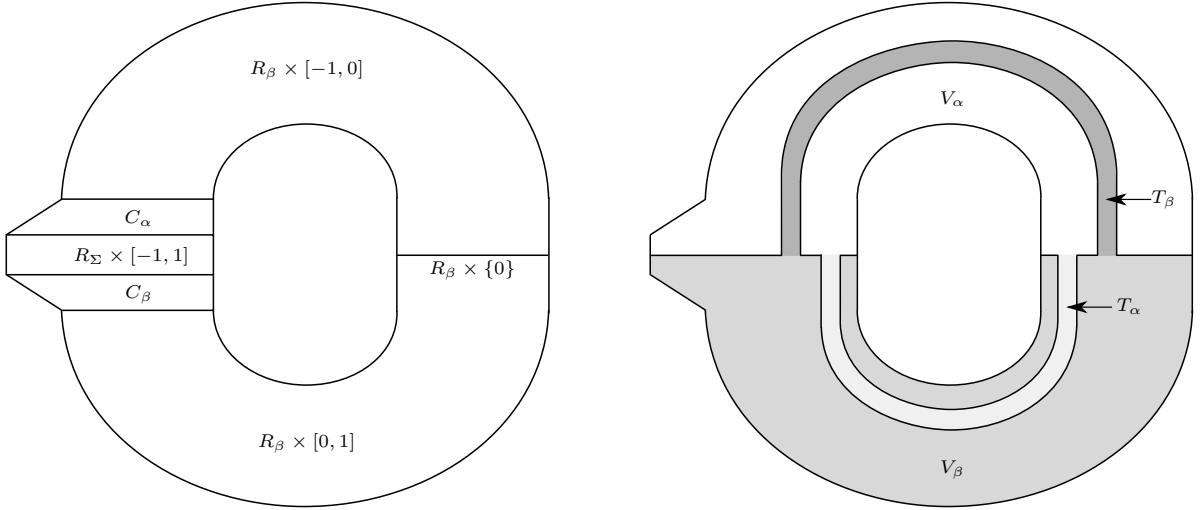
\begin{figure}[htbp]
\labellist
\tiny
\pinlabel $R_\Sigma\times[-1,1]$ at 66 132
\pinlabel $C_\beta$ at 70 112
\pinlabel $C_\alpha$ at 70 152
\pinlabel $R_{\beta\alpha}\times[-1,0]$ at 158 230
\pinlabel $R_{\beta\alpha}\times[0,1]$ at 158 34
\pinlabel $R_{\beta\alpha}\times\{0\}$ at 245 125

\pinlabel $V_\alpha$ at 497 215
\pinlabel $V_\beta$ at 497 20

\pinlabel $T_\beta$ at 608 162

\pinlabel $T_\alpha$ at 590 106

\endlabellist
\centering
\includegraphics[width=13cm]{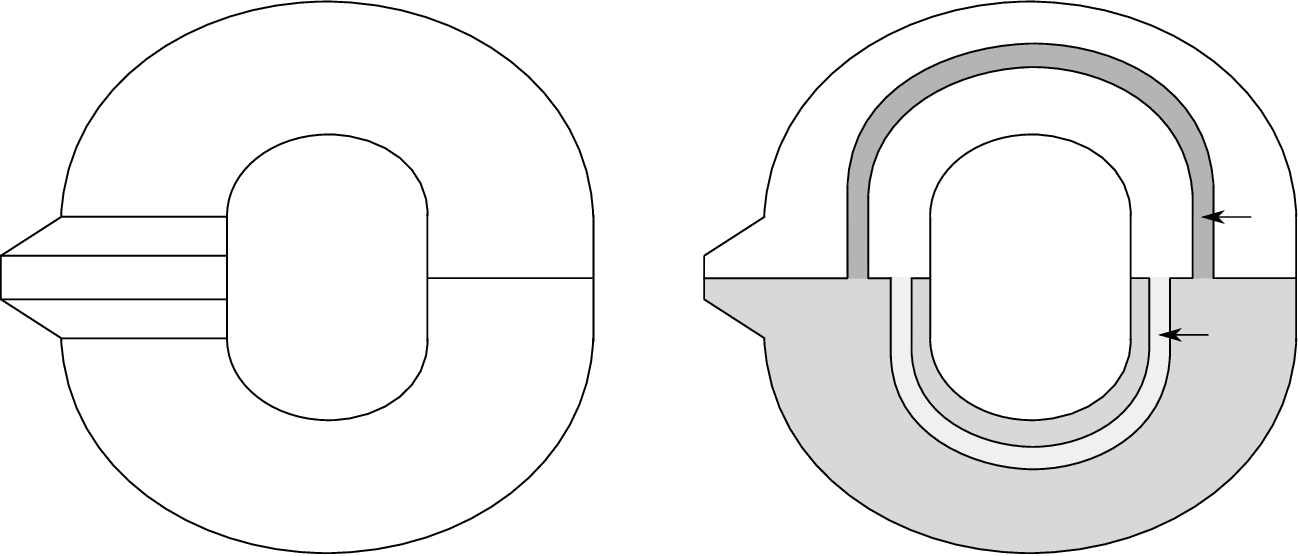}
\caption{Left, a schematic of the closure $Y_{\beta\alpha}$. Right, a schematic of the standard Heegaard splitting of the closure: the handlebody $H_\alpha$ is the union of $V_\alpha$ and $T_\alpha$, shown in white and light gray, respectively; the handlebody $H_\beta$ is the union of $V_{\beta}$ and $T_\beta$, shown in medium and dark gray, respectively.}
\label{fig:schematichd}
\end{figure}

We turn now to the claim about the triple diagram. Viewing  $-H(S)$ as the sutured manifold determined by the sutured Heegaard diagram \[(\Sigma,\{\beta_1,\dots,\beta_n\},\{\alpha_1,\dots,\alpha_n\}),\] we first observe that $\gamma_i\subset \Sigma$ is isotopic in $-H(S)$ to the curve $s_i$ in \eqref{eqn:si}, for $i=1,\dots,n$, and that the $\Sigma$-framing of $\gamma_i$ coincides with the $\partial H(S)$-framing of $s_i$. Indeed, this observation only requires thinking about how $\Sigma$ is embedded in $-H(S)$. In particular,  the subsurface \[H_1 \cup\dots \cup H_n\cup S\subset \Sigma\] on which the $\gamma_i$ lie  is isotopic in $-H(S)$ to the union \[(A_1\cup\dots\cup A_n)\times\{1\}\,\cup\,\partial S\times[-1,1]\,\cup\,S\times\{-1\},\] where each $A_j$ is a rectangular neighborhood of $c_j\subset S$ (note that  $A_j$ is disjoint from $\boldsymbol\alpha$ after isotoping the latter, which is why $A_j$ can be pushed into $S\times\{1\}$). Furthermore, this isotopy carries $\gamma_i$  to the union \[(c_i\times\{1\} )\,\cup\,(\partial c_i\times[-1,1])\,\cup\, (h(c_i)\times\{-1\}),\] which is precisely the curve $s_i$. Moreover, it is also not hard to see that $\beta_i$ bounds a meridional disk for $\gamma_i$. Alternatively,  the above observation is made clear in Figure \ref{fig:surgerytriple}.

This observation, combined with the facts that 
\begin{itemize}\item  $\gamma_i$ intersects $\beta_i$ in exactly one point and is disjoint from the other curves in $\boldsymbol\beta$, for $i=1,\dots,n$, and 
\item $\gamma_i$ is isotopic to $\beta_i$ for $i=n+1,\dots,n+2g+1$,
\end{itemize} is equivalent to the statement that \[(\ul\Sigma,\boldsymbol\gamma,\boldsymbol\beta,\boldsymbol\alpha)\] is a left-subordinate Heegaard triple diagram for the cobordism $W_{\gamma\beta\alpha}=-W$ associated to surgery on the  framed link \[\mathbb{L}=s_1\cup\dots\cup s_n\subset Y_{\beta\alpha},\] as  in \cite{osz5}. In particular, \[(\ul\Sigma,\{\beta_{n+1},\dots,\beta_{n+2g+1}\},\alpha)\] specifies the complement of a bouquet for this framed link.

\begin{figure}[ht]
\labellist

\tiny
\pinlabel $=$ at 540 80

\endlabellist
\centering
\includegraphics[width=13cm]{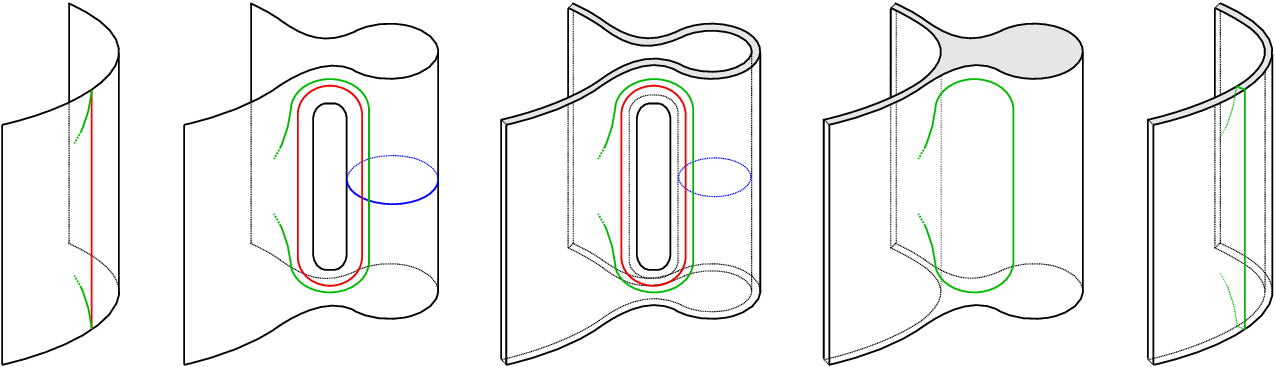}
\caption{From left to right: the first picture shows a neighborhood in $S$ of the arc $c_i$. The arc $c_i$ is shown in red,  its image $h(c_i)$ is shown in green. The second picture shows the result of attaching the $1$-handles $H_i$ and $H_i'$ to $S$ to form $\Sigma$. The curve $\alpha_i$ is shown in red, $\gamma_i$ in green, $\beta_i$ in blue. The third picture shows the corresponding curves on $\Sigma\times[-1,1]$. The fourth picture shows the result of attaching thickened disks to $\Sigma\times[-1,1]$ along $\alpha_i$ and $\beta_i$. The resulting manifold is  $S\times[-1,1]$, as is evident by comparing the fourth and fifth pictures, and  $\gamma_i$ is isotopic to $s_i$, where $s_i$ is shown in green in the last picture. It is also clear from this diagram that the $\Sigma$-framing of $\gamma_i$ is taken to the $\partial H(S)$-framing of $s_i$ under this isotopy. }
\label{fig:surgerytriple}
\end{figure}

\begin{remark}We will   use  $-R$ to refer to the isotopic  surfaces  $-R_{\beta\alpha}$ and $-R_{\gamma\alpha}$ in $-W$.
\end{remark}

To define Heegaard Floer homology groups and the map between them from these diagrams,  we must specify the location of the basepoint $z$. For this, let us   wind $\alpha_{n{+}2g{+}1}$ halfway along the curve $\delta$, as in  Figure \ref{fig:uvuv2}, so that it meets each of $\beta_{n{+}2g{+}1}$ and $\gamma_{n{+}2g{+}1}$ in  two points \[u_{\beta\alpha},v_{\beta\alpha}\textrm{ and } u_{\gamma\alpha},v_{\gamma\alpha},\] respectively, as indicated in the figure. We  place the basepoint $z$ in the small bigon  created by this isotopy, as shown  in the figure.

\begin{figure}[ht]
\labellist
\tiny
\pinlabel $z$ at 438 267

\pinlabel $\delta$ at 61 106
\pinlabel $u_{\beta\alpha}$ at 714 274
\pinlabel $v_{\beta\alpha}$ at 714 140
\pinlabel $u_{\gamma\alpha}$ at 642 269
\pinlabel $v_{\gamma\alpha}$ at 650 122
\pinlabel $\Theta^{n{+}2g{+}1}$ at 744 327
\pinlabel $\alpha_{n{+}2g{+}1}$ at 162 66
\pinlabel $\beta_{n{+}2g{+}1}$ at 162 292
\pinlabel $F$ at 330 60
\pinlabel $E$ at 520 60


\small
\pinlabel $z$ at 740 200

\endlabellist
\centering
\includegraphics[width=11cm]{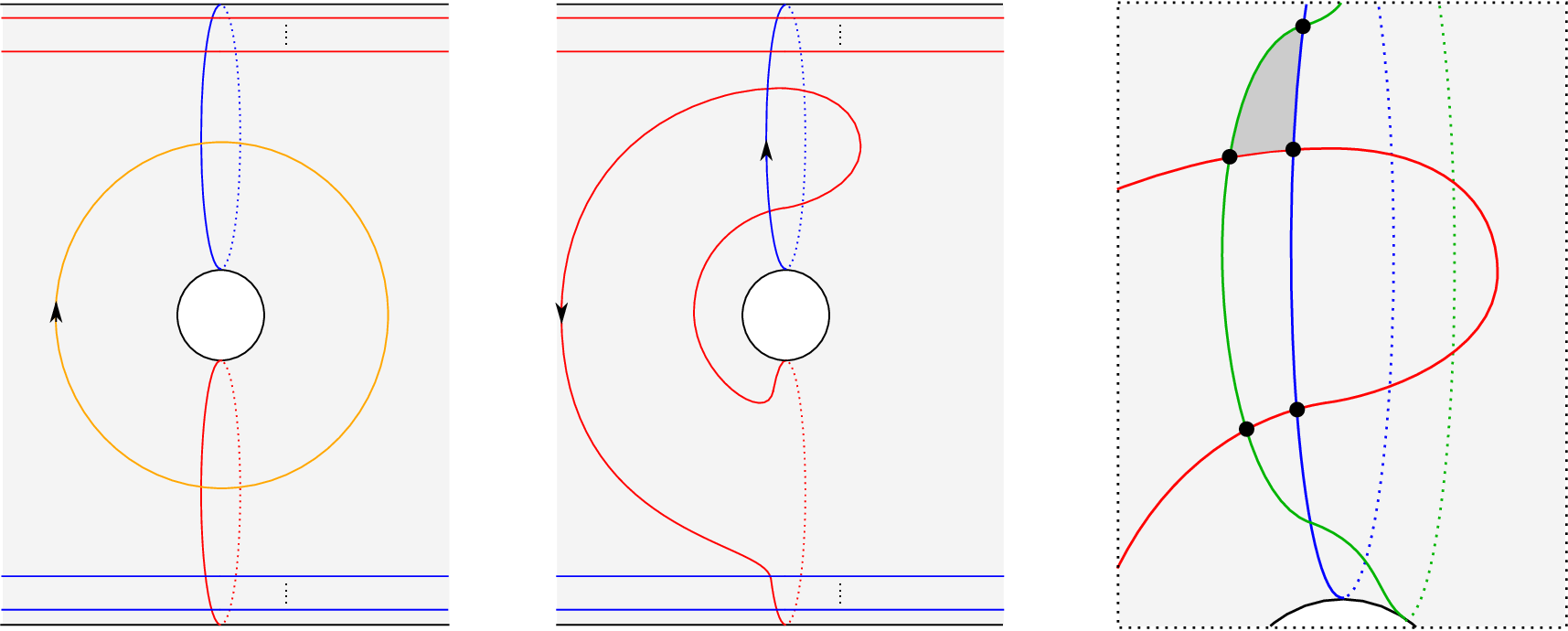}
\caption{Left, before winding. Middle, after winding, and the basepoint $z$. Right, a closeup  near the intersection points $u_{\beta\alpha},v_{\beta\alpha},u_{\gamma\alpha},v_{\gamma\alpha},$ and $\Theta^{n{+}2g{+}1}$. The triangle $\Delta^{n{+}2g{+}1}$ is shaded.}
\label{fig:uvuv2}
\end{figure}

Finally, let us fix an oriented, embedded curve \begin{equation}\label{eqn:etatwisting}\eta\subset T\ssm D_a\ssm D_b\subset \ul R_{S}\subset \ul\Sigma\end{equation} which is dual to a nonseparating curve $c$ in $T$.  This  defines curves in $Y_{\beta\alpha}=-Y_S$ and $Y_{\gamma\alpha}=-Y$ which we will also denote by $\eta$.  Let \[\nu = \eta\times I\subset -W\] be the cylindrical cobordism from $\eta\subset -Y_S$ to $\eta\subset -Y$. These $\eta$ and   $\nu$ will be used to define both the Heegaard Floer homology groups of $-Y_S$ and $-Y$, with local coefficients in $\Gamma_\eta$, and the map induced by $-W$ between these groups, as in Subsections \ref{ssec:hflocal} and \ref{ssec:cob}.

\begin{remark}
The condition on $\eta$ above along with  the condition that the maps $\varphi_\alpha$ and $\varphi_\gamma$ in \eqref{eqn:vara} and \eqref{eqn:varc} restrict to the identity on $T\ssm D_a \ssm D_b$ ensures that the triple $(Y_S,R,\eta)$ meets the topological requirements for a contact closure; namely, that  it  is formed from a preclosure $P$ by gluing $\partial_+P$ to $\partial_-P$ by a map which sends  $c\times\{+1\}$ to $c\times\{-1\}$ for some nonseparating  curve $c$ in the auxiliary surface, with $\eta$ dual to $c$, as described in Subsection \ref{ssec:hm}.
\end{remark}

\subsection{Periodic domains and winding} 
\label{ssec:periodic}
In this subsection, we catalogue certain important periodic domains in the pointed Heegaard triple diagram $(\ul\Sigma,\boldsymbol\gamma,\boldsymbol\beta,\boldsymbol\alpha)$ and  introduce a procedure called \emph{winding}. This setup will be used crucially for results in later subsections.

Let us henceforth orient $\gamma_i,$ $\beta_i$, and $\alpha_i$ as in
\begin{itemize}
\item Figure \ref{fig:HH2} for $i=1,\dots,n$ (the $\beta_i$ here are not oriented; see below)
\item Figures \ref{fig:TSarcs} and \ref{fig:TSarcs2} for $i=n{+}1,\dots,n{+}2g$, and 
\item Figure \ref{fig:uvuv2} for $i=n{+}2g{+}1$.
\end{itemize} We orient each $\gamma_i$ in the same direction as the parallel  $\beta_i$ for $i=n{+}1,\dots,n{+}2g{+}1$. We will not need to orient $\beta_1,\dots,\beta_n$, which is why we have not specified an orientation on the $\beta_i$ in Figure \ref{fig:HH2}.

Let $\P_{\beta\alpha}$ and $\P_{\gamma\alpha}$ be the $(\boldsymbol\beta,\boldsymbol\alpha)$- and $(\boldsymbol\gamma,\boldsymbol\alpha)$-periodic domains in $(\ul\Sigma,\boldsymbol\gamma,\boldsymbol\beta,\boldsymbol\alpha)$ with multiplicities $2$ and $1$ in the regions $\Sigma$ and $S$ of $\ul\Sigma$, respectively, and with 
\begin{align*}
\partial \P_{\beta\alpha} &= \beta_{n{+}2g{+}1} - \alpha_{n{+}2g{+}1} ,\\
\partial \P_{\gamma\alpha} &= \gamma_{n{+}2g{+}1}- \alpha_{n{+}2g{+}1},
\end{align*}
as  in Figure \ref{fig:periodic}. Observe that $\P_{\beta\alpha}$ corresponds, in the diagram $(\ul \Sigma, \boldsymbol\beta,\boldsymbol\alpha)$  \emph{before} winding  $\alpha_{n{+}2g{+}1}$ as in Figure \ref{fig:uvuv2}, to the periodic domain  $2\ul R_\Sigma + \ul R_S$. Since the periodic  domains $\ul R_\Sigma$ and $-\ul R_S$ in that diagram both represent $R_{\beta\alpha}\subset Y_{\beta\alpha}$, as explained in the previous subsection, it follows that  the periodic domain $\P_{\beta\alpha}$ represents the homology class \[2[R_{\beta\alpha}]-[R_{\beta\alpha}]=[R_{\beta\alpha}]\in H_2(Y_{\beta\alpha};\mathbb{Z}).\] By the same argument, $\P_{\gamma\alpha}$ represents the homology class \[[R_{\gamma\alpha}]\in H_2(Y_{\gamma\alpha};\mathbb{Z}).\] 

\begin{figure}[ht]
\labellist
\tiny
\pinlabel $E$ at 260 35
\pinlabel $F$ at 37 35
\pinlabel $1$ at 190 150
\pinlabel $1$ at 260 65
\pinlabel $0$ at 166 241
\pinlabel $2$ at 119 150
\pinlabel $2$ at 37 65

\endlabellist
\centering
\includegraphics[width=4.3cm]{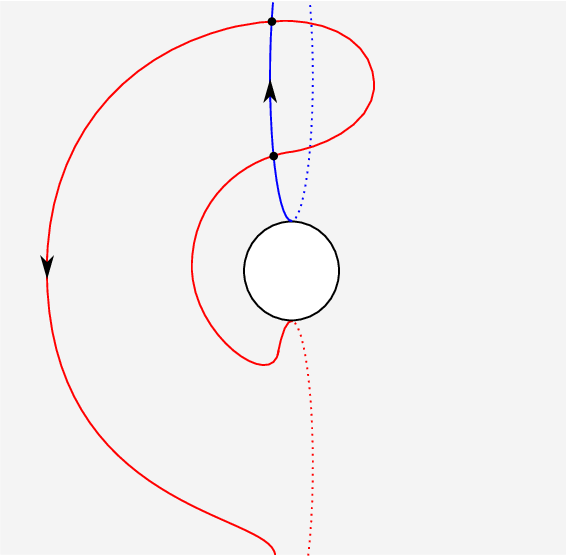}
\caption{The multiplicities of $\P_{\beta\alpha}$ near $\beta_{n+2g+1}$ and $\alpha_{n+2g+1}$. }
\label{fig:periodic}
\end{figure}

For  $i=n{+}1,\dots,n{+}2g{+}1$, let $\P_i$ be denote the small $(\boldsymbol\gamma,\boldsymbol\beta)$-periodic domain in $(\ul\Sigma,\boldsymbol\gamma,\boldsymbol\beta,\boldsymbol\alpha)$ with \[\partial \P_i = \gamma_i-\beta_i.\] This domain has multiplicities $\pm 1$ in two thin bigon regions. We note that \[\{\P_{n+1},\dots,\P_{n+2g+1}\}\] is a basis for the $\Q$-vector space $\ul\Pi_{\gamma\beta}$ of rational $(\boldsymbol\gamma,\boldsymbol\beta)$-periodic domains in  the pointed Heegaard triple diagram $(\ul\Sigma,\boldsymbol\gamma,\boldsymbol\beta,\boldsymbol\alpha,z)$, though we will not really make use of this fact.

Note that there is no triply-periodic domain in the sutured Heegaard triple diagram \[(\Sigma,\{\gamma_1,\dots,\gamma_n\},\{\beta_1,\dots,\beta_n\}, \{\alpha_1,\dots,\alpha_n\})\] whose boundary contains a nonzero multiple of $\beta_i$ for $i=1,\dots,n$. This is because  in the complement of the attaching curves, there are regions on either side of $\beta_i$ which intersect $\partial\Sigma$. Letting $\Pi_{\gamma\alpha}$ and $\Pi_{\gamma\beta\alpha}$  be the $\Q$-vector spaces of rational $(\boldsymbol\gamma,\boldsymbol\alpha)$- and triply-periodic domains, respectively, in this sutured triple diagram, we have just argued that $\Pi_{\gamma\alpha} = \Pi_{\gamma\beta\alpha}$. Let us fix a basis \begin{equation}\label{eqn:basisD}\{\D_1,\dots,\D_p\}\end{equation} for this vector space. Note that $\beta_i$ intersects each of $\gamma_i$ and $\alpha_i$ positively in  one point,  and is disjoint from all other $\boldsymbol\gamma$ and $\boldsymbol\alpha$ curves, for $i=1,\dots,n$. It follows that $\gamma_i$ and $\alpha_i$  occur with  opposite multiplicities in the boundary of every $\D_j$. By  a change of basis,  we may therefore assume that \begin{equation}\label{eqn:basis2}\partial \D_j =\alpha_{i_j}-\gamma_{i_j} +\sum_{k=1}^{i_j-1} a_{j,k}(\alpha_k - \gamma_k),\end{equation} for some integers \[1\leq i_1<\dots<i_p\leq n,\] where 
$a_{j,i_r} =0$ for all $r=1,\dots,p$. 

We will extend this basis in two ways. First, let $\ul\Pi_{\gamma\beta\alpha}^2$ denote the $\Q$-vector space of rational triply-periodic domains in the pointed Heegaard triple diagram $(\ul\Sigma,\boldsymbol\gamma,\boldsymbol\beta,\boldsymbol\alpha,z)$ which are sums of rational doubly-periodic domains. The collection \[\{\P_{\gamma\alpha},\D_1,\dots,\D_p,\P_{n+1},\dots,\P_{n+2g+1}\}\] is linearly independent in $\ul\Pi_{\gamma\beta\alpha}^2$ and therefore extends for some $m$ to a basis \begin{equation}\label{eqn:basisTextended}\{\P_{\gamma\alpha},\D_1,\dots,\D_p,\T_{p+1}\dots,\T_m,\P_{n+1},\dots,\P_{n+2g+1}\}\end{equation} for this vector space. By adding multiples of $\P_i$, we may assume that no $\T_j$ contains a nonzero multiple of $\beta_{i}$ in its boundary, for  $i=n{+}1,\dots,n{+}2g{+}1$. Moreover, for  each such $i$, there is an oriented curve in $\ul\Sigma$ which intersects $\gamma_i$ and $\alpha_i$ positively in exactly one point each, intersects all other $\boldsymbol\gamma$ and $\boldsymbol\alpha$ curves zero times algebraically, and is disjoint from $\beta_1,\dots,\beta_n$ (this is evident from Figure \ref{fig:TSarcs}).  It follows that $\gamma_i$ and $\alpha_i$  occur with  opposite multiplicities in the boundary of each $\T_j$. In particular, by adding multiples of $\P_{\gamma\alpha}$, we may assume that no $\T_j$ contains a nonzero multiple of $\gamma_{n+2g+1}$ or $\alpha_{n+2g+1}$ in its boundary. By further change of basis, we may therefore assume that  \begin{equation}\label{eqn:basisTstd}\partial \T_j =\alpha_{i_j}-\gamma_{i_j} +\sum_{k=1}^{i_j-1} a_{j,k}(\alpha_k - \gamma_k) +\sum_{k=1}^{n} b_{j,k}\beta_k,\end{equation} for some integers
\[
n+1\leq i_{p+1}<\dots<i_m\leq n+2g ,\] where $a_{j,i_r} =0$ for all $r=1,\dots,m$.  

Next, let $\ul\Pi_{\gamma\alpha}$ denote the $\Q$-vector space of rational $(\boldsymbol\gamma,\boldsymbol\alpha)$-periodic domains in the pointed Heegaard  diagram $(\ul\Sigma,\boldsymbol\gamma,\boldsymbol\alpha,z)$. We extend the basis  \eqref{eqn:basisD} to a basis \begin{equation}\label{eqn:basisDextended}\{\P_{\gamma\alpha},\D_1,\dots,\D_p,\D_{p+1},\dots,\D_m\}\end{equation} for this  space (we are reusing the notation $m$; see Remark \ref{rmk:reuse}). By the same reasoning as above, we may assume  for $j=p{+}1,\dots,m$ that \begin{equation}\label{eqn:basisDstd}\partial \D_j =\alpha_{i_j}-\gamma_{i_j} +\sum_{k=1}^{i_j-1} a_{j,k}(\alpha_k - \gamma_k),\end{equation} for some integers
\[
n+1\leq i_{p+1}<\dots<i_m\leq n+2g ,\] where $a_{j,i_r} =0$ for all $r=1,\dots,m$. 
\begin{remark}\label{rmk:reuse}For expository purposes, we are  abusing notation slightly above; namely, we are   repurposing  the notation $m$, $i_{p+1},\dots,i_m$, and $a_{j,k}$ from \eqref{eqn:basisTextended} and \eqref{eqn:basisTstd}. This will not lead to confusion as we will never use the bases \eqref{eqn:basisTextended} and \eqref{eqn:basisDextended} at the same time.
\end{remark}

For much of what follows in this paper, we will need to isotope the $\boldsymbol\alpha$ curves by a procedure called \emph{winding}, described below. 
 
Note that for each $i=n{+}1,\dots,n{+}2g$, there is a homologically essential curve $\nu_i\subset \ul R_S$ such that $\nu_i$ intersects $\alpha_i$ exactly once and is disjoint from all other $\alpha_j$. We may assume that  $\nu_i$ is disjoint from the disks $D_1,\dots,D_n$ in \eqref{eqn:homo}. Let \begin{equation}\label{eqn:etai}\eta_i = \varphi_\alpha(\nu_i)\subset \ul R_\Sigma\subset \ul \Sigma,\end{equation} where $\varphi_\alpha$ is the map in \eqref{eqn:vara}. Then $\eta_i$ also intersects $\alpha_i$ exactly once and is disjoint from the other $\boldsymbol\alpha$ curves. Let $\eta_i^\pm$ and $\eta_i^-$ be parallel copies of $\eta_i$, oriented as on the left in Figure \ref{fig:etawinding}. One may then \emph{wind} $\alpha_i$ along these curves in the directions given by their orientations. The diagram on the right in Figure \ref{fig:etawinding} shows the result of winding $\alpha_i$ once along each of $\eta_i^\pm$. 

\begin{figure}[ht]
\labellist
\tiny
\pinlabel $\alpha_i$ at -7 61
\pinlabel $\eta_i^+$ at 63 129
\pinlabel $\eta_i^-$ at 144 129
\pinlabel $w_i^+$ at 71 31
\pinlabel $w_i^-$ at 152 95
\pinlabel $w_i^+$ at 290 31
\pinlabel $w_i^-$ at 372 95

\endlabellist
\centering
\includegraphics[width=10cm]{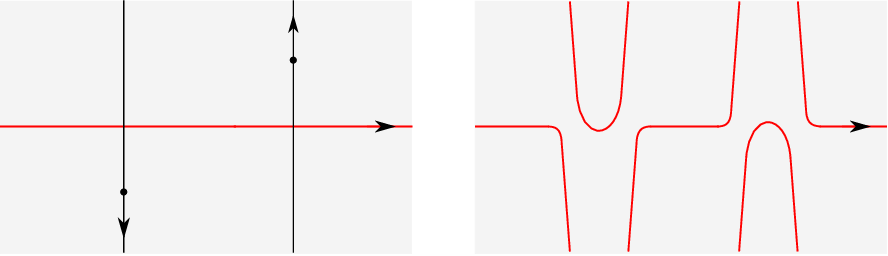}
\caption{Left, the intersection of $\eta_i^+$ and $\eta_i^-$ with $\alpha_i$, and the points $w_i^+$ and $w_i^-$. Right, the curve $\alpha_i$ after winding once along $\eta_i^+$ and $\eta_i^-$.}
\label{fig:etawinding}
\end{figure}

We will need to keep track of the effect of   such winding on the coefficients of various domains of these Heegaard diagrams. In order to do so, we introduce  new basepoints \begin{equation*}\label{eqn:basepts}w_1^+,\dots,w_{n+2g}^+\, \textrm{ and } \,w_1^-,\dots,w_{n+2g}^-\subset \ul R_{\Sigma}\subset \ul\Sigma,\end{equation*} where 
\begin{itemize}
\item $w_i^\pm$  are the points in 
$H_i\subset\Sigma\subset\ul\Sigma$ shown in Figure \ref{fig:HH2}, for   $i=1,\dots,n$,   and 
\item  $w_i^\pm$  are the points   on $\eta_i^\pm$ shown  in Figure \ref{fig:etawinding}, for   $i=n{+}1,\dots,n{+}2g$. 
\end{itemize}
Note that if $D$ is a domain of the diagram $(\ul\Sigma,\boldsymbol\gamma,\boldsymbol\beta,\boldsymbol\alpha,z)$ with \[\partial_{\alpha_i}(D)\cdot \eta_i^{\pm}=\pm a\] for some $i=n{+}1,\dots,n{+}2g$, then the quantity \[n_{w_i^{\pm}}(D)\] changes by $\pm a$ after winding $\alpha_i$ once along $\eta_i^{\pm}$, and is unaffected by all other winding. Furthermore, the quantity \[n_{w_j^{\pm}}(D)\] is unaffected by  winding, for $j=1,\dots,n$. For any such  domain $D$, and any $i=1,\dots,n{+}2g$, we will use the shorthand  \begin{equation}\label{eqn:min}\ul n_{w_i}(D) := \min\{n_{w_i^+}(D),n_{w_i^-}(D)\}.\end{equation}

We  describe below how the rational periodic  domains  in the bases \eqref{eqn:basisTextended} and \eqref{eqn:basisDextended}  behave  with respect to winding. Note first that 
\begin{align}
\label{eqn:pga}n_{w_{i}^\pm}(\P_{\gamma\alpha})&=2 \textrm{ for  all $i$},\\
\label{eqn:pi}n_{w_{i}^\pm}(\P_j)&=0 \textrm{ for  all $i$, and $j=n+1,\dots,n+2g+1$}.
\end{align}
Moreover, for $j=1,\dots,p$, 
\begin{align}
\label{eqn:ndj}n_{w_{i_j}^\pm}(\D_j) &= \pm 1,\\
\label{eqn:ndk}n_{w_{i_k}^\pm}(\D_j) &= 0 \textrm{ for  all $k\neq j$}.
\end{align}
The quantities in \eqref{eqn:pga}-\eqref{eqn:ndk} are not affected by winding. On the other hand, for  $j=p{+}1,\dots,m$, \[n_{w_{i_j}^\pm}(\T_j)\textrm{ and }n_{w_{i_j}^\pm}(\D_j)\] change by $\pm 1$ after winding $\alpha_{i_j}$ once along $\eta_{i_j}^\pm$  and are unaffected by all other winding (note that  $i_j$ above does  not necessarily refer to  the same number in  $n_{w_{i_j}^\pm}(\T_j)$ as in $n_{w_{i_j}^\pm}(\D_j)$; recall that we are simply using the same notation for each basis, per Remark \ref{rmk:reuse}). 

Note also that \begin{equation}\label{eqn:EFzero}n_E(\D_j) = n_F(\D_j) = 0=2n_E(\P_{\gamma\alpha})-n_F(\P_{\gamma\alpha})\end{equation} for all $j=1,\dots,m$, where $E$ and $F$ are the points shown in Figure \ref{fig:uvuv2}.  In particular, the fact that  \[n_E(\D_j)=n_F(\D_j)=0\]  follows from the facts that $n_z(\D_j)=0$ and  the boundary of $\D_j$ is disjoint from $\gamma_{n+2g+1}$ and $\alpha_{n+2g+1}$.

\subsection{Top $\Sc$ structures} 
We will identify below the generators of the vector spaces \[\cfp(\ul\Sigma,\boldsymbol\beta,\boldsymbol\alpha,z)\textrm{ and }\cfp(\ul\Sigma,\boldsymbol\gamma,\boldsymbol\alpha,z)\] in the ``top" $\Sc$ structures  with respect to the genus $g$ distinguished surfaces \begin{align*}
-R_{\beta\alpha}=-R&\textrm{ in } Y_{\beta\alpha} = -Y_S, \textrm{ and}\\
-R_{\gamma\alpha}=-R&\textrm{ in }Y_{\gamma\alpha} = -Y.
\end{align*}
Recall from the previous subsection that the $(\boldsymbol\beta,\boldsymbol\alpha)$- and $(\boldsymbol\gamma,\boldsymbol\alpha)$-periodic domains $\P_{\beta\alpha}$ and $\P_{\gamma\alpha}$ represent the homology classes of $R_{\beta\alpha}$ and $R_{\gamma\alpha}$, respectively. Suppose $[\bx,i]$ is a generator of \[\cfp(\ul\Sigma,\boldsymbol\beta,\boldsymbol\alpha,z)\textrm{ or }\cfp(\ul\Sigma,\boldsymbol\gamma,\boldsymbol\alpha,z).\] We therefore have that
\begin{align}
\label{eqn:1}\langle c_1(\spc_z(\bx)),[R_{\beta\alpha}]\rangle&=\langle c_1(\spc_z(\bx)),[\P_{\beta\alpha}]\rangle=e(\P_{\beta\alpha})+ 2n_{\bx}(\P_{\beta\alpha}), \textrm{ or}\\
\label{eqn:2}\langle c_1(\spc_z(\bx)),[R_{\gamma\alpha}]\rangle&=\langle c_1(\spc_z(\bx)),[\P_{\gamma\alpha}]\rangle=e(\P_{\gamma\alpha}) + 2n_{\bx}(\P_{\gamma\alpha}),
\end{align}
respectively, where the rightmost equalities follow from \cite[Proposition~7.5]{osz14}, and where $e(D)$ refers to the \emph{Euler measure} of the domain $D$. For each $i=n{+}1,\dots,n{+}2g$, let $x_{\beta\alpha}^i$ and $x_{\gamma\alpha}^i$ be the unique intersection points 
\begin{align*}
x_{\beta\alpha}^i&=\beta_i\cap\alpha_{\sigma(i)}\cap \ul R_S,\\
x_{\gamma\alpha}^i&=\gamma_i\cap\alpha_{\sigma(i)}\cap \ul R_S,
\end{align*}
shown in Figure \ref{fig:TSarcs2}, and 
define
\begin{align*}
\bx_{\beta\alpha}&=\{x_{\beta\alpha}^{n+1},\dots, x_{\beta\alpha}^{n+2g}\}\in \operatorname{Sym}^{2g}(\ul R_S),\\
\bx_{\gamma\alpha}&=\{x_{\gamma\alpha}^{n+1},\dots, x_{\gamma\alpha}^{n+2g}\}\in \operatorname{Sym}^{2g}(\ul R_S).
\end{align*}
Recall that $\alpha_{n+2g+1}$ intersects  each of $\beta_{n+2g+1}$ and $\gamma_{n+2g+1}$ in two points \[u_{\beta\alpha},v_{\beta\alpha}\textrm{ and } u_{\gamma\alpha},v_{\gamma\alpha},\] respectively, as indicated in Figure \ref{fig:uvuv2}.
Our main result is the following.

\begin{lemma}
\label{lem:bot}
A generator $[\bx,i]\in\cfp(\ul\Sigma,\boldsymbol\gamma,\boldsymbol\alpha,z)$ satisfies \[\spc_z(\bx)\in\Sc(Y_{\gamma\alpha}|{-}R_{\gamma\alpha})\]  iff $\bx$ is of the form
\[\bx = \by\cup\{u_{\gamma\alpha}\}\cup \bx_{\gamma\alpha}\,\textrm{ or } \,\bx=\by\cup\{v_{\gamma\alpha}\}\cup \bx_{\gamma\alpha}, \]  where
\[
 \by\in (\gamma_1\times\dots\times\gamma_n)\cap (\alpha_1\times\dots\times\alpha_n)\subset \operatorname{Sym}^n(\Sigma).\] The analogous statement holds when replacing $\boldsymbol\gamma$ with $\boldsymbol\beta$ everywhere.
\end{lemma}

\begin{proof}
First, let us suppose that $\bx$ is of the form described in the lemma. Note that 
\begin{align*}
e(\P_{\gamma\alpha}) &= 2\chi(\ul R_\Sigma) + \chi(\ul R_S) = -6g-4n,\\
n_{\by}(\P_{\gamma\alpha})&=2n,\\
n_{\bx_{\gamma\alpha}}(\P_{\gamma\alpha})&=2g,\\
n_{u_{\gamma\alpha}}(\P_{\gamma\alpha}) = n_{v_{\gamma\alpha}}(\P_{\gamma\alpha}) &= 1 .
\end{align*}
The formula  \eqref{eqn:1} then implies that \[\langle c_1(\spc_z(\bx)),[-R_{\gamma\alpha}]\rangle=2g-2,\] as desired. For the converse, it is easy to see that if $\bx$ is not of this form, then $n_{\bx}(\P_{\gamma\alpha})>2n+2g+1$ (changing $\bx$ from a generator of this form to another generator  moves intersection points from the portion of $\ul\Sigma$ where $\P_{\gamma\alpha}$ has multiplicity $1$ to the portion where it has multiplicity $2$) which implies that \[\langle c_1(\spc_z(\bx)),[-R_{\gamma\alpha}]\rangle<2g-2.\] See \cite[Lemma~11]{lekili2} for what is virtually the same argument.
\end{proof}

\begin{remark} Note that the generators in the top $\Sc$ structures do not change after winding by the procedure described in the previous subsection.
\end{remark}

\subsection{Technical lemmas}
In this subsection, we use the set up of the previous two subsections to prove some technical lemmas. These lemmas are needed for  our admissibility results in the next subsection and for many of the results in Subsections \ref{ssec:proofiso} and \ref{ssec:proofmap}.

In preparation for the results below, we introduce some new terminology. This terminology makes reference to the pointed Heegaard triple diagram $(\ul\Sigma,\boldsymbol\gamma,\boldsymbol\beta,\boldsymbol\alpha,z)$ for $-W$. In particular, we say that $\spt\in\Sc(-W)$ \emph{satisfies} the adjunction inequality if \[|\langle c_1(\spt),[Z] \rangle| +Z\cdot Z \leq \max\{0,-\chi(Z)\}\] for every connected, embedded surface $Z\subset -W$ with nonnegative self-intersection. Define \[\SA(-W|{-}R)\subset \Sc(-W|{-}R)\] to be the set of $\spc\in\Sc(-W|{-}R)$ which either
\begin{enumerate}
\item satisfy the adjunction inequality, or
\item restrict to a $\Sc$ structure on each of $Y_{\gamma\beta}$, $Y_{\beta\alpha}$, $Y_{\gamma\alpha}$ represented by a generator of the corresponding Heegaard Floer complex.
\end{enumerate} Likewise, we say that $\spc\in\Sc(-Y)$ \emph{satisfies} the  adjunction inequality if \[|\langle c_1(\spt),[Z] \rangle| \leq \max\{0,-\chi(Z)\}\] for every connected, embedded surface $Z\subset -Y$, and we define  \[\SA(-Y|{-}R)\subset \Sc(-Y|{-}R)\] to be the set of $\spc\in\Sc(-Y|{-}R)$ which either
\begin{enumerate}
\item satisfy the adjunction inequality, or
\item are represented by a generator of the corresponding Heegaard Floer complex,
\end{enumerate} and similarly for $-Y_S$. 

\begin{remark}
We are abusing notation slightly above since $\SA(-W|{-}R)$ and $\SA(-Y|{-}R)$ technically depend on the Heegaard (triple) diagram; the  diagram we are using will generally be implicit, however, whenever this notation is used.
\end{remark}

\begin{remark}
\label{rmk:SAadj}
The Floer homology groups $\hfp(-Y_S|{-}R;\Gamma_\eta)$ and $\hfp(-Y|{-}R;\Gamma_\eta)$ and the map \[\hfp(-W|{-}R;\Gamma_\nu):\hfp(-Y_S|{-}R;\Gamma_\eta)\to \hfp(-Y|{-}R;\Gamma_\eta)\] can be nontrivial \emph{only} for $\Sc$ structures in $\SA(-Y_S|{-}R)$, $\SA(-Y|{-}R)$, and $\SA(-W|{-}R)$, by Theorems \ref{thm:adj3} and \ref{thm:adj4}. We can and will therefore restrict our attention to these $\Sc$ structures.
\end{remark}

We next define a quantity $L>0$ which depends only on the sutured Heegaard  diagram  \[(\Sigma,\{\gamma_1,\dots,\gamma_n\},\{\alpha_1,\dots,\alpha_n\})\] for $(-M,-\Gamma)$. We will refer to this quantity below and  in later subsections as well. The definition of $L$ makes reference to the basis  \eqref{eqn:basisD} \[\{\D_1,\dots,\D_p\}\] for the vector space  $\Pi_{\gamma\alpha}$ of rational periodic domains in this sutured diagram. For  $j=1,\dots, p$, choose a connected, embedded surface $Z_j\subset (-M,-\Gamma)$ with $[Z_j]=[\D_j]$ in $H_2(-M).$ Let \[B_j=\max\{0,-\chi(Z_j)\},\] and let \[B=\max\{B_1,\dots,B_p\}.\] Let  \[C_j=\max\{| e(\D_j) + 2n_{\by}(\D_j)|\},\] where the maximum is  over all  \begin{equation}\label{eqn:by}\by\in (\gamma_1\times\dots\times\gamma_n)\cap (\alpha_1\times\dots\times\alpha_n)\subset \operatorname{Sym}^n(\Sigma),\end{equation} and let \[C=\max\{C_1,\dots,C_p\}.\] We define \begin{equation}\label{eqn:L}L:=\max\{B,C,1\}.\end{equation}

We may now prove the following \emph{a priori} bound on $\langle c_1(\spc),[\D_j]\rangle$ for $\spc\in \SA(-Y|{-}R)$.

\begin{lemma}
\label{lem:SAs}
If $\spc\in \SA(-Y|{-}R)$ then $|\langle c_1(\spc),[\D_j]\rangle|\leq L$ for  $j=1,\dots,p$.
\end{lemma}

\begin{proof}
Suppose $\spc\in\SA(-Y|{-}R)$. If $\spc$ satisfies the adjunction inequality then we have that \[|\langle c_1(\spc),[\D_j]\rangle|\leq B\leq L. \]   If $\spc$ is represented by a generator $\bx$ of $\cfp(\ul\Sigma,\boldsymbol\gamma,\boldsymbol\alpha,z)$ then  \[\bx = \by\cup\{u_{\gamma\alpha}\}\cup \bx_{\gamma\alpha}\, \,\textrm{ or }\,\, \bx=\by\cup\{v_{\gamma\alpha}\}\cup \bx_{\gamma\alpha} \] for some $\by$ as in \eqref{eqn:by}, by Lemma \ref{lem:bot}. Therefore, \[|\langle c_1(\spc),[\D_j]\rangle| = |\langle c_1(\spc_z(\bx)),[\D_j]\rangle|=| e(\D_j) + 2n_{\bx}(\D_j)|=| e(\D_j) + 2n_{\by}(\D_j)|\leq C\leq L,\] where the second equality is from \cite[Proposition~7.5]{osz14}, and the third   from the fact that \[n_{u_{\gamma\alpha}}(\D_j)=n_{v_{\gamma\alpha}}(\D_j)=n_{\bx_{\gamma\alpha}}(\D_j)=0. \qedhere\] 
\end{proof}

A similar bound holds for the cobordism $-W$, per the following.

\begin{lemma}
\label{lem:SAt}
If $\spt\in \SA(-W|{-}R)$ then $|\langle c_1(\spt),[\D_j]\rangle|\leq L$ for  $j=1,\dots,p$.
\end{lemma}

\begin{proof}If  $\spt\in \SA(-W|{-}R)$ then  $\spt|_{-Y}\in \SA(-Y|{-}R)$, and 
 \[|\langle c_1(\spt),[\D_j]\rangle|=|\langle c_1(\spt|_{-Y}),[\D_j]\rangle|\leq L\] for  $j=1,\dots,p$ by Lemma \ref{lem:SAs}.\end{proof}

We further have the following \emph{a priori} bound on $\langle c_1(\spt),[\P_j]\rangle$ for $\spt\in \SA(-W|{-}R)$.

\begin{lemma}
\label{lem:SAtP}
If $\spt\in\SA(-W|{-}R)$ then $|\langle c_1(\spt),[\P_j]\rangle|=0$ for  $j=n{+1},\dots,n{+}2g{+}1$.
\end{lemma}

\begin{proof}
Note first that \[|\langle c_1(\spt),[\P_j]\rangle|=|\langle c_1(\spt|_{Y_{\gamma\beta}}),[\P_j]\rangle|,\] where  \[Y_{\gamma\beta}\cong\#^{2g+1}(S^1\times S^2).\] Moreover, the diagram $(\ul\Sigma,\boldsymbol\gamma,\boldsymbol\beta,z)$ is  an $n$-fold stabilization of the standard genus $2g{+}1$ Heegaard diagram for this manifold. In particular, the homology classes $[\P_{n+1}],\dots,[\P_{n+2g+1}]$ are represented by $\{\textrm{pt}\}\times S^2$s and the generators  of $\cfp(\ul\Sigma,\boldsymbol\gamma,\boldsymbol\beta,z)$ represent the torsion $\Sc$ structure. 
If $\spt$ satisfies the adjunction inequality then  $|\langle c_1(\spt),[\P_j]\rangle|=0$  since the $[\P_i]$  are represented by spheres in $Y_{\gamma\beta}$; if $\spt$ restricts to a $\Sc$ structure on $Y_{\gamma\beta}$ represented by a generator then $|\langle c_1(\spt|_{Y_{\gamma\beta}}),[\P_j]\rangle|=0$ since every such  $\Sc$ structure is torsion. 
\end{proof}

Finally, we  have the following bound on $\langle c_1(\spt),[\T_j]\rangle$ for $\spt\in \SA(-W|{-}R)$, where the number $Q$ below may depend on the closure.

\begin{lemma}
\label{lem:SAtT}
There exists  some $Q>0$ such that \[|\langle c_1(\spt),[\T_j]\rangle| \leq Q\] for  $j=p{+}1,\dots,m$ and any $\spt\in\SA(-W|{-}R)$.
\end{lemma}

\begin{proof}
For  $j=p+1,\dots,m$, choose  a connected, embedded surface $Z_j\subset -W$ with $[Z_j]=[\T_j]$ in $H_2(-W).$ Let
 \[G_j=\max\{0,-\chi(Z_j)\},\] and let \[G=\max\{G_{p+1},\dots,G_m\}.\] For   $j=p{+}1,\dots,m$, let us write $\T_j$ as a sum \[\T_j = \T_j^{\gamma\beta} + \T_j^{\beta\alpha} +\T_j^{\gamma\alpha}\] of rational $(\boldsymbol\gamma,\boldsymbol\beta)$-, $(\boldsymbol\beta,\boldsymbol\alpha)$-, and $(\boldsymbol\gamma,\boldsymbol\alpha)$-periodic domains (recall that the $\T_j$ are sums of doubly-periodic domains, by definition).  Let 
\begin{align*}
H_j^{\gamma\beta} &= \max\{|\langle c_1(\spc_z(\bx)),[\T_j^{\gamma\beta}]\rangle|\mid \bx\in \T_\gamma\cap \T_\beta\},\\
H_j^{\beta\alpha} &= \max\{|\langle c_1(\spc_z(\bx)),[\T_j^{\beta\alpha}]\rangle|\mid \bx\in \T_\beta\cap \T_\alpha\},\\
H_j^{\gamma\alpha} &= \max\{|\langle c_1(\spc_z(\bx)),[\T_j^{\gamma\alpha}]\rangle|\mid \bx\in \T_\gamma\cap \T_\alpha\}.
\end{align*}
Let \[H_j = H_j^{\gamma\beta}+H_j^{\beta\alpha}+H_j^{\gamma\alpha},\] and let \[H = \max\{H_{p+1},\dots,H_m\}.\] Define \[Q:=\max\{G,H,1\}.\] Now suppose $\spt\in\SA(-W|{-}R)$. Since $\T_j$ is a sum of rational doubly-periodic domains, we have that $[\T_j]\cdot[\T_j]=0$. If $\spt$ satisfies the adjunction inequality then  \[|\langle c_1(\spt),[\T_j]\rangle|\leq G\leq Q. \] If $\spt$ restricts to a $\Sc$ structure on each of $Y_{\gamma\beta}$, $Y_{\beta\alpha}$, $Y_{\gamma\alpha}$ represented by a generator, then we have that \[|\langle c_1(\spt),[\T_j]\rangle|\leq H\leq Q, \] by construction. This proves the claim.
\end{proof}

We now prove Lemmas \ref{lem:tech1} and \ref{lem:tech2} and the accompanying Lemmas \ref{lem:tech13} and \ref{lem:tech23}, which are the main results of this subsection. Roughly, these lemmas state that by making the genus $g$ of the distinguished surface  large enough and winding sufficiently, we can ensure that  any periodic domain of a certain  form will have either a very negative Maslov index contribution (as measured by  evaluations of Chern classes of $\Sc$ structures on its homology class) or  a very negative multiplicity somewhere. These lemmas will be essential for proving admissibility results in the next subsection, as well as for the results which go into the proofs of Theorems \ref{thm:iso} and \ref{thm:map} in Subsections \ref{ssec:proofiso} and \ref{ssec:proofmap}. The latter are the key theorems  used to prove Theorem \ref{thm:mainob2} in Subsection \ref{ssec:facthm}. Recall below that $p$ is the size of the basis \eqref{eqn:basisD}.

\begin{lemma}
\label{lem:tech1}
Fix some $N>0$. Suppose \[g\geq 2pL+2.\] After sufficient winding, the following is true:
for any rational linear combination \[P=a_{\gamma\alpha} \P_{\gamma\alpha} + \sum_{i=1}^p b_i \D_i +\sum_{i=p+1}^m c_i\T_i + \sum_{i=n+1}^{n+2g+1} d_i\P_i \] where some $|c_j|\geq 1$, and any $\spt\in \SA(-W|{-}R)$, we have that either:
\begin{itemize}
\item $\langle c_1(\spt),[P]\rangle \leq -N$, or
\item $\ul n_{w_{i_j}}(P)\leq -N$ for some $j\in\{1,\dots,m\}$.
\end{itemize}
\end{lemma}

\begin{proof}Fix some $N>0$ and fix some $g\geq 2pL+2.$ Let $Q$ be as in Lemma \ref{lem:SAtT}.
Let \[K= \max\{|n_{w_{i}^\pm}(\T_j)|\mid i=1,\dots,n+2g \text{ and } j=p{+}1,\dots,m\}.\] Let us  wind  $\alpha_i$ along each of the curves $\eta_i^+$ and $\eta_i^-$ a total of  \begin{equation}\label{eqn:J}J\geq 2Km+2N+mQ\end{equation} times in the directions given by their orientations, for each $i=n{+}1,\dots,n{+}2g$. Now suppose $P$ is as in the hypothesis of the lemma. Suppose $r$ and $s$ satisfy
\[|b_r|= \max\{|b_i|\mid i=1,\dots,p\}\textrm{ and }
|c_{s}|= \max\{|c_i|\mid i=p{+}1,\dots,m\}.\] Before winding, we have that 
\begin{align}
\label{eqn:wir}\ul n_{w_{i_r}}(P)&\leq K(m-p)|c_s| -|b_r| + 2a_{\gamma\alpha}\leq Km|c_s| -|b_r| + 2a_{\gamma\alpha},\\
\ul n_{w_{i_s}}(P)&\leq K(m-p)|c_s| + 2a_{\gamma\alpha}\leq Km|c_s| + 2a_{\gamma\alpha},
\end{align}
as follows from \eqref{eqn:pga}, \eqref{eqn:pi}, \eqref{eqn:ndj}, and \eqref{eqn:ndk}. (Recall that the  $i_j$ subscripts refer to curves in the boundaries of $\D_j$ and $\T_j $ as in \eqref{eqn:basis2} and \eqref{eqn:basisTstd}.) After winding, $\ul n_{w_{i_r}}(P)$ is unchanged, but  \begin{equation}\label{eqn:wis}
\ul n_{w_{i_s}}(P)\leq Km|c_s| + 2a_{\gamma\alpha} -J|c_s|.\end{equation} Suppose    $\ul n_{w_{i_r}}(P)$ and $\ul n_{w_{i_s}}(P)$ are both greater than $-N$ after winding. Then 
\begin{align}
\label{eqn:br}|b_r|&<2a_{\gamma\alpha}+Km|c_s|+N\leq 2a_{\gamma\alpha}+Km|c_s|+N|c_s|,\\
\label{eqn:2a}2a_{\gamma\alpha}&>J|c_s|-Km|c_s|-N\geq (Km+2N+mQ)|c_s|-N\geq(Km+N+mQ)|c_s|. \end{align} The top inequalities follow from \eqref{eqn:wir} and the fact that $|c_s|\geq 1$; the bottom inequalities follow from \eqref{eqn:wis}, \eqref{eqn:J}, and $|c_s|\geq 1$. 
To prove the lemma, it suffices to show that \[\langle c_1(\spt),[P]\rangle \leq -N.\] Note that  \begin{align*}
\langle c_1(\spt),[P]\rangle &\leq a_{\gamma\alpha}(2-2g)+pL|b_r| + mQ|c_s|\\
&\leq 2a_{\gamma\alpha}(-1-2pL)+pL(2a_{\gamma\alpha}+Km|c_s|+N|c_s|) + mQ|c_s|\\
&=-2a_{\gamma\alpha} -pL2a_{\gamma\alpha} + pL(Km+N+mQ)|c_s| + mQ|c_s|-pLmQ|c_s|\\
&\leq -2a_{\gamma\alpha}-pL(Km+N+mQ)|c_s|+pL(Km+N+mQ)|c_s| + mQ|c_s|-pLmQ|c_s|\\
&=-2a_{\gamma\alpha}+mQ|c_s|-pLmQ|c_s|\\
&<-(Km+N+mQ)|c_s|+mQ|c_s|-pLmQ|c_s|\\
&=(-Km-N)|c_s|-pLmQ|c_s|\\
&\leq -N,
\end{align*}
as desired. The first line follows from  Lemmas \ref{lem:SAtP} and \ref{lem:SAtT}. The second line follows from \eqref{eqn:br}, the fact that \[g\geq2pL+2 \iff 1-g \leq -1-2pL,\] and the fact that $a_{\gamma\alpha}>0$. The latter fact follows from the inequality \eqref{eqn:2a} since $K,m,N,Q$ are nonnegative by definition. The fourth and sixth lines in the inequalities above also follow from \eqref{eqn:2a}.
\end{proof}

\begin{lemma}
\label{lem:tech2}
Fix some $N> 0$. Suppose \[g\geq pLN+pL+N+1.\]
The following is true: for any rational linear combination \[P=a_{\gamma\alpha} \P_{\gamma\alpha} + \sum_{i=1}^p b_i \D_i + \sum_{i=n+1}^{n+2g+1} d_i\P_i \] where $a_{\gamma\alpha}\geq 1$, and any $\spt\in\SA(-W|{-}R)$,  we have that either:
\begin{itemize}
\item $\langle c_1(\spt),[P]\rangle \leq -N$, or
\item $\ul n_{w_{i_j}}(P)\leq -N$ for some $j\in\{1,\dots,p\}$.
\end{itemize}\end{lemma}

\begin{proof} Suppose $P$ is as in the hypothesis of the lemma.
Suppose \[|b_r|= \max\{|b_i|\mid i=1,\dots,p\}.\] Then
\[\ul n_{w_{i_r}}(P)\leq -|b_r| +2a_{\gamma\alpha}.\] Suppose $\ul n_{w_{i_r}}(P)$ is greater than $-N$. Then \begin{equation}\label{eqn:br2}|b_r|<2a_{\gamma\alpha}+N.\end{equation} To prove the lemma, it suffices to show that \[\langle c_1(\spt),[P]\rangle \leq -N.\]  Note that \begin{align*}
\langle c_1(\spt),[P]\rangle &\leq a_{\gamma\alpha}(2-2g)+pL|b_r|\\
&\leq 2a_{\gamma\alpha}(-pLN-pL-N) + pL|b_r|\\
&\leq 2a_{\gamma\alpha}(-pLN-pL-N) + pL(2a_{\gamma\alpha}+N)\\
&=-pLN(2a_{\gamma\alpha}-1)-N(2a_{\gamma\alpha})\\
&\leq -N,
\end{align*}
as desired. The second line follows from the fact that  \[g\geq pLN+pL+N+1 \iff 1-g \leq -pLN-pL-N;\] the third from \eqref{eqn:br2}; the last from the fact that $a_{\gamma\alpha}\geq 1$.
\end{proof}

Lemmas \ref{lem:tech13} and \ref{lem:tech23} below are analogues of those above, but for the basis \eqref{eqn:basisDextended}  rather than \eqref{eqn:basisTextended}. Their proofs are  identical to those of Lemmas \ref{lem:tech1} and \ref{lem:tech2}, so we omit them (the  properties of the domains $\mathbb{T}_j$ used in these proofs---namely, Lemma \ref{lem:SAtT} and the behaviors of the coefficients $n_{w_{i}^\pm}(\T_j)$ under winding---have direct analogues which hold after replacing the $\mathbb{T}_j$ with the $\mathbb{D}_j$, for $j=p+1,\dots,m$).

\begin{lemma}
\label{lem:tech13}
Fix some $N>0$. Suppose \[g\geq 2pL+2.\] After sufficient winding, the following is true:
for any rational linear combination \[P=a_{\gamma\alpha} \P_{\gamma\alpha} + \sum_{i=1}^p b_i \D_i +\sum_{i=p+1}^m c_{i}\D_{i}  \] where some $|c_j|\geq 1$, and any $\spc\in \SA(-Y|{-}R)$, we have that either:
\begin{itemize}
\item $\langle c_1(\spc),[P]\rangle \leq -N$, or
\item $\ul n_{w_{i_j}}(P)\leq -N$ for some $j\in\{1,\dots,m\}$.
\end{itemize}

\end{lemma}

\begin{lemma}
\label{lem:tech23}
Fix some $N> 0$. Suppose \[g\geq pLN+pL+N+1.\]
The following is true: for any rational linear combination \[P=a_{\gamma\alpha} \P_{\gamma\alpha} + \sum_{i=1}^p b_i \D_i  \] where $a_{\gamma\alpha}\geq 1$, and any $\spc\in\SA(-Y|{-}R)$,  we have that either:
\begin{itemize}
\item $\langle c_1(\spc),[P]\rangle \leq -N$, or
\item $\ul n_{w_{i_j}}(P)\leq -N$ for some $j\in\{1,\dots,p\}$.\qed
\end{itemize}
\end{lemma}

\begin{remark}
One can just as easily prove analogues of Lemmas \ref{lem:tech13} and \ref{lem:tech23} with respect to a basis of periodic domains for the diagram $(\ul\Sigma,\boldsymbol\beta,\boldsymbol\alpha,z)$ for $-Y_S$.
\end{remark}

\subsection{Admissibility} In order to use the diagrams  $(\ul\Sigma,\boldsymbol\beta,\boldsymbol\alpha,z)$,  $(\ul\Sigma,\boldsymbol\gamma,\boldsymbol\alpha,z)$, and $(\ul\Sigma,\boldsymbol\gamma,\boldsymbol\beta,\boldsymbol\alpha,z)$  to define  chain complexes and a chain map which compute the homology groups 
 $\hfp(-Y_S|{-}R;\Gamma_\eta)$ and $\hfp(-Y|{-}R;\Gamma_\eta)$ and the map \[\hfp(-W|{-}R;\Gamma_\nu):\hfp(-Y_S|{-}R;\Gamma_\eta)\to \hfp(-Y|{-}R;\Gamma_\eta),\]  these diagrams must satisfy certain admissibility conditions, as described in Subsections \ref{ssec:hflocal} and \ref{ssec:cob}. We prove below that we can achieve the required admissibility for $g$ sufficiently large and after sufficient winding.

\begin{proposition}
\label{prop:admissible} For sufficiently large $g$ and sufficient winding, the diagram $(\ul\Sigma,\boldsymbol\gamma,\boldsymbol\beta,\boldsymbol\alpha,z)$ is weakly $\spt$-admissible for all $\spt\in\SA(-W|{-}R)$.
\end{proposition}

\begin{proof}

Fix some $N>0$. Fix some \[g\geq \max\{2pL+2, pLN+pL+N+1\}\] for $L$ as in \eqref{eqn:L}.  Wind sufficiently so that the conclusion of Lemma \ref{lem:tech1} holds for this $N$. With this  $g$  and this winding, we may  prove the admissibility  claimed in the proposition. Suppose  $\spt\in\SA(-W|{-}R)$. Suppose $P$ is a nontrivial  triply-periodic domain in $(\ul\Sigma,\boldsymbol\gamma,\boldsymbol\beta,\boldsymbol\alpha,z)$ which is a sum of doubly-periodic domains. Then $[P]\cdot [P]=0$. Suppose \[\langle c_1(\spt),[P]\rangle=0.\] We must show that $P$ has a negative multiplicity.  If the coefficient of $\T_j$ in $P$ is nonzero for some $j\in\{p{+}1,\dots,m\}$ then $P$ has a negative multiplicity by Lemma \ref{lem:tech1}. Otherwise, if the coefficient of $\T_j$ in $P$ is zero for  $j=p{+}1,\dots,m$ and the coefficient of $\P_{\gamma\alpha}$ in $P$ is nonzero then $P$ has a negative multiplicity by Lemma \ref{lem:tech2} (either the coefficient of $\P_{\gamma\alpha}$ is negative, in which case $P$ has a negative multiplicity, or the coefficient is positive, in which case Lemma \ref{lem:tech2} applies). We may  thus assume that the coefficients in $P$ of $\P_{\gamma\alpha}$ and $\T_j$ are zero for  $j=p{+}1,\dots,m$. Therefore, $P$ is  a linear combination \[P=\sum_{i=1}^pb_i\D_i+\sum_{i=n+1}^{n+2g+1}d_{i}\P_{i}.\] Since $P$ is nontrivial, either some $b_j$ is nonzero or some $d_i$ is nonzero. In the first case,  \[n_{w_{i_j}^{\pm}}(P)= \pm b_j\] by \eqref{eqn:pi}, \eqref{eqn:ndj}, \eqref{eqn:ndk}. In the second case, $P$ has  multiplicities $\pm d_i$ in some thin regions. In either case, $P$ has a negative multiplicity, completing the proof.
\end{proof}

The result below is proven is exactly the same way as Proposition \ref{prop:admissible}, except that we use Lemmas \ref{lem:tech13} and \ref{lem:tech23} instead of Lemmas \ref{lem:tech1} and \ref{lem:tech2}; we therefore omit its proof. 

\begin{proposition}
\label{prop:admissible3} For sufficiently large $g$ and sufficient winding, the diagram $(\ul\Sigma,\boldsymbol\gamma,\boldsymbol\alpha,z)$ is weakly $\spc$-admissible for  all $\spc\in\SA(-Y|{-}R)$.\qed\end{proposition} 

\begin{remark}
\label{rmk:admissibleba}
One can just as easily prove the analogous admissibility result for the diagram $(\ul\Sigma,\boldsymbol\beta,\boldsymbol\alpha,z)$ and $\spc\in\SA(-Y_S|{-}R)$.
\end{remark}

\subsection{Theorems  \ref{thm:iso} and \ref{thm:map}  imply Theorem \ref{thm:mainob2}}
\label{ssec:facthm}
In this subsection, we state  Theorems \ref{thm:iso} and \ref{thm:map}, and explain how they imply Theorem \ref{thm:mainob2}, which then implies our main theorem, Theorem \ref{thm:mainob}, as explained at the beginning of this section. 

In preparation for the statements of these theorems, let  us define \[\Theta= \{\Theta^1,\dots,\Theta^{n{+}2g{+}1}\}\in \T_\gamma\cap \T_\beta,\]  where each $\Theta^i$ is one of the two  intersection points between $\gamma_i$ and $\beta_i$, as shown in 
\begin{itemize}
\item Figure \ref{fig:HH2} for $i=1,\dots,n$,
\item Figure \ref{fig:TSarcs2} for $i=n{+}1,\dots,n{+}2g$, and 
\item Figure \ref{fig:uvuv2} for $i=n{+}2g{+}1$,
\end{itemize}
 so that $[\Theta,0]$ is the unique generator in the top Maslov grading of \[\cfp(\ul\Sigma,\boldsymbol\gamma,\boldsymbol\beta,z)\] among generators killed by $U$. For each $\spt\in\Sc(-W)$, the map 
\[\hfp(W_{\gamma\beta\alpha},\spt;\Gamma_\nu): \hfp(Y_{\beta\alpha},\spt|_{Y_{\beta\alpha}};\Gamma_\eta)\to \hfp(Y_{\gamma\alpha},\spt|_{Y_{\gamma\alpha}};\Gamma_\eta)\] is then defined as in Subsection \ref{ssec:cob}, in terms of  a chain map \begin{equation*}\label{eqn:f+2}f^+_{\gamma\beta\alpha,\spt;\Gamma_\nu}:\cfp(\ul\Sigma,\boldsymbol\beta,\boldsymbol\alpha,z,\spt|_{Y_{\beta\alpha}};\Gamma_\eta)\to\cfp(\ul\Sigma,\boldsymbol\gamma,\boldsymbol\alpha,z,\spt|_{Y_{\gamma\alpha}};\Gamma_\eta)\end{equation*} defined on a generator $[\bx,i]$ by \[ f^+_{\gamma\beta\alpha,\spt;\Gamma_\nu}([\bx,i]) =\sum_{\substack{\phi\in\pi_2(\Theta,\bx,\by)\\\spc_z(\phi)=\spt\\\mu(\phi)=0}} \#\mathcal{M}(\phi)\cdot [\by,i-n_z(\phi)]\cdot t^{\partial_\alpha(\phi)\cdot \eta}, \]   
assuming the diagram $(\ul\Sigma,\boldsymbol\gamma,\boldsymbol\beta,\alpha,z)$ is weakly $\spt$-admissible.

\begin{remark}We hereafter assume that $g$ is sufficiently large and  we have wound sufficiently so that 
\begin{itemize}
\item $(\ul\Sigma,\boldsymbol\gamma,\boldsymbol\beta,\boldsymbol\alpha,z)$ is weakly $\spt$-admissible for  $\spt\in\SA(-W|{-}R)$, and
\item $(\ul\Sigma,\boldsymbol\beta,\boldsymbol\alpha,z)$ and  $(\ul\Sigma,\boldsymbol\gamma,\boldsymbol\alpha,z)$ are weakly $\spc$-admissible for  $\spc$ in $\SA(-Y_S|{-}R)$ and $\SA(-Y|{-}R)$,
\end{itemize} per Propositions \ref{prop:admissible} and \ref{prop:admissible3} and Remark \ref{rmk:admissibleba}.
\end{remark}

Let us denote by \[\cfp(\ul\Sigma,\boldsymbol\beta,\boldsymbol\alpha,z|{-}R_{\beta\alpha};\Gamma_\eta)\textrm{ and }\cfp(\ul\Sigma,\boldsymbol\gamma,\boldsymbol\alpha,z|{-}R_{\gamma\alpha};\Gamma_\eta)\] the  direct sums of  the chain complexes \[\cfp(\ul\Sigma,\boldsymbol\beta,\boldsymbol\alpha,z,\spc;\Gamma_\eta)\textrm{ and }\cfp(\ul\Sigma,\boldsymbol\gamma,\boldsymbol\alpha,z,\spc;\Gamma_\eta)\] over $\Sc$ structures $\spc$ in $\SA(-Y_S|{-}R)$ and $\SA(-Y|{-}R)$, respectively.
Let us define \[f^+_{\gamma\beta\alpha;\Gamma_\nu}:\cfp(\ul\Sigma,\boldsymbol\beta,\boldsymbol\alpha,z|{-}R_{\beta\alpha};\Gamma_\eta)\to\cfp(\ul\Sigma,\boldsymbol\gamma,\boldsymbol\alpha,z|{-}R_{\gamma\alpha};\Gamma_\eta)\] to be the sum of the maps $f^+_{\gamma\beta\alpha,\spt;\Gamma_\nu}$ over $\spt\in\SA(-W|{-}R)$.  This  induces the map \begin{equation}\label{eqn:Gmap}B:=\hfp({-}W|{-}R;\Gamma_\nu): \hfp({-}Y_S|{-}R;\Gamma_\eta)\to \hfp({-}Y|{-}R;\Gamma_\eta)\end{equation} on homology, by Theorems \ref{thm:adj3} and  \ref{thm:adj4} (see Remark \ref{rmk:SAadj}).

For each $i=1,\dots,n$, let $c_{\beta\alpha}^i$ and $c_{\gamma\alpha}^i$ be the unique   intersection points 
\begin{align*}
c_{\beta\alpha}^i&=\beta_i\cap\alpha_i\cap H_i,\\
c_{\gamma\alpha}^i&=\gamma_i\cap\alpha_i\cap H_i
\end{align*}
shown in Figure \ref{fig:HH2}, 
and define
\begin{align*}
\bcc_{\beta\alpha}&=\{c_{\beta\alpha}^{1},\dots, c_{\beta\alpha}^{n}\}\in \sfc(\Sigma,\{\beta_1,\dots,\beta_n\},\{\alpha_1,\dots,\alpha_n\}),\\
\bcc_{\gamma\alpha}&=\{c_{\gamma\alpha}^{1},\dots, c_{\gamma\alpha}^{n}\}\in \sfc(\Sigma,\{\gamma_1,\dots,\gamma_n\},\{\alpha_1,\dots,\alpha_n\}).
\end{align*}
\begin{remark}
\label{rmk:reps}These  $\bcc_{\beta\alpha}$ and $\bcc_{\gamma\alpha}$ are representatives of the contact elements $c_{\HF}(\ob_S)$ and $c_{\HF}(\ob)$ associated to the partial open books $\ob_S=(S,\emptyset,\emptyset,\emptyset)$ and $\ob=(S,P,h,\{c_1,\dots,c_n\})$, by the discussion in Subsection \ref{ssec:partialob}.
\end{remark}

We will prove the two theorems below in the next two subsections.

\begin{theorem}
\label{thm:iso}
For sufficiently large $g$  and sufficient winding, there are quasi-isomorphisms 
\begin{align}
\label{eqn:qi1}f_{\beta\alpha}&:\sfc(\Sigma,\{\beta_1,\dots,\beta_n\},\{\alpha_1,\dots,\alpha_n\})\otimes\Lambda\to\cfp(\ul\Sigma,\boldsymbol\beta,\boldsymbol\alpha,z|{-}R_{\beta\alpha};\Gamma_{\eta})\\
\label{eqn:qi2}
f_{\gamma\alpha}&:\sfc(\Sigma,\{\gamma_1,\dots,\gamma_n\},\{\alpha_1,\dots,\alpha_n\})\otimes\Lambda\to\cfp(\ul\Sigma,\boldsymbol\gamma,\boldsymbol\alpha,z|{-}R_{\gamma\alpha};\Gamma_{\eta})\end{align} sending 
$\bcc_{\beta\alpha}\otimes 1$ to $[\bcc_{\beta\alpha}\cup \{u_{\beta\alpha}\}\cup \bx_{\beta\alpha},0]$ and $\bcc_{\gamma\alpha}\otimes 1$ to $ [\bcc_{\gamma\alpha}\cup \{u_{\gamma\alpha}\}\cup \bx_{\gamma\alpha},0]$, respectively.
\end{theorem}

\begin{remark}
Theorem \ref{thm:iso} also holds without tensoring with the Novikov ring on the left side and without twisted coefficients on the right. 
\end{remark}

\begin{remark}
The fact that $\bcc_{\beta\alpha}$ and $\bcc_{\gamma\alpha}$ are cycles implies that \[ [\bcc_{\beta\alpha}\cup \{u_{\beta\alpha}\}\cup \bx_{\beta\alpha},0]\textrm{ and } [\bcc_{\gamma\alpha}\cup \{u_{\gamma\alpha}\}\cup \bx_{\gamma\alpha},0]\] are cycles for sufficiently large $g$ and sufficient winding,  by Theorem \ref{thm:iso}.
\end{remark}

\begin{theorem}
\label{thm:map}
For sufficiently large $g$ and sufficient winding, the  map  \[B=(f^{+}_{\gamma\beta\alpha;\Gamma_\nu})_*:H_*(\cfp(\ul\Sigma,\boldsymbol\beta,\boldsymbol\alpha,z|{-}R_{\beta\alpha};\Gamma_\eta))\to H_*(\cfp(\ul\Sigma,\boldsymbol\gamma,\boldsymbol\alpha,z|{-}R_{\gamma\alpha};\Gamma_\eta))\] sends \[[[\bcc_{\beta\alpha}\cup\{u_{\beta\alpha}\}\cup \bx_{\beta\alpha},0]]\,\textrm{ to }\,[[\bcc_{\gamma\alpha}\cup\{u_{\gamma\alpha}\}\cup \bx_{\gamma\alpha},0]].\]
\end{theorem}

We explain below how Theorems \ref{thm:iso} and \ref{thm:map} imply Theorem \ref{thm:mainob2}.

\begin{proof}[Proof of Theorem \ref{thm:mainob2}]
Suppose $g$ is sufficiently large and  we have wound sufficiently to guarantee the conclusions of Theorems  \ref{thm:iso} and \ref{thm:map}. Fix   a contact structure $\bar\xi_S$ on $Y_S$ such that  $(Y_S,R,\bar\xi_S,\eta)$ is a contact closure of $(H(S),\xi_S)$. Let $\bar\xi$ be the contact structure on $Y$ obtained from $\bar\xi_S$ via contact $(+1)$-surgery on a Legendrian realization of the link \[\mathbb{L}=s_1\cup\dots \cup s_n\subset Y_S.\] As  discussed in Subsection \ref{ssec:reform}, $(Y,R,\bar\xi,\eta)$ is a contact closure of $(M,\Gamma,\xi)$.
Note  that \[\sfc(\Sigma,\{\beta_1,\dots,\beta_n\},\{\alpha_1,\dots,\alpha_n\})\otimes\Lambda\] is generated by $\bcc_{\beta\alpha}\otimes 1$. It therefore follows from Theorem \ref{thm:iso}  that $[\bcc_{\beta\alpha}\cup \{u_{\beta\alpha}\}\cup \bx_{\beta\alpha},0]$ is a cycle which represents a generator  \[ \mathbf{1}\in\hfp(Y_{\beta\alpha}|{-}R_{\beta\alpha};\Gamma_\eta)=  \hfp(-Y_S|{-}R;\Gamma_\eta)\cong \Lambda.\]  By Theorem \ref{thm:map}, the map \[B=\hfp({-}W|{-}R;\Gamma_\nu): \hfp({-}Y_S|{-}R;\Gamma_\eta)\to \hfp({-}Y|{-}R;\Gamma_\eta)\] satisfies \begin{equation}\label{eqn:B1}B(\mathbf{1})=[[\bcc_{\gamma\alpha}\cup \{u_{\gamma\alpha}\}\cup \bx_{\gamma\alpha},0]].\end{equation} Meanwhile,  Corollary \ref{cor:functgen} says that there is an open book $\bar\ob$ compatible with $\bar\xi$ such that \begin{equation}\label{eqn:B12}B(\mathbf{1})=c_{\HF}(\bar\ob).\end{equation}
 It follows from \eqref{eqn:B1} and \eqref{eqn:B12} that \begin{equation}\label{eqn:chf2}[[\bcc_{\gamma\alpha}\cup \{u_{\gamma\alpha}\}\cup \bx_{\gamma\alpha},0]]=c_{\HF}(\bar\ob).\end{equation}  As the class $c_{\HF}(\ob)\otimes 1$ is given by \[c_{\HF}(\ob)\otimes 1=[\bcc_{\gamma\alpha}\otimes 1]\in H_*(\sfc(\Sigma,\{\gamma_1,\dots,\gamma_n\},\{\alpha_1,\dots,\alpha_n\})\otimes \Lambda)=\sfh(-M,-\Gamma)\otimes\Lambda,\]  per Remark \ref{rmk:reps}, Theorem \ref{thm:iso} combined with \eqref{eqn:chf2}  implies that the isomorphism \[A:=(f_{\gamma\alpha})_*:\sfh(-M,-\Gamma)\otimes \Lambda\to  \hfp(-Y|{-}R;\Gamma_\eta)\] sends $c_{\HF}(\ob)\otimes 1$ to $c_{\HF}(\bar\ob)$, proving Theorem \ref{thm:mainob2}.
\end{proof}

\subsection{Proof of Theorem  \ref{thm:iso}} 
\label{ssec:proofiso}

The maps $f_{\beta\alpha}$ and $f_{\gamma\alpha}$ we have in mind in \eqref{eqn:qi1} and \eqref{eqn:qi2} of Theorem \ref{thm:iso} are the $\Lambda$-linear maps 
\begin{align*}
f_{\beta\alpha}&:\sfc(\Sigma,\{\beta_1,\dots,\beta_n\},\{\alpha_1,\dots,\alpha_n\})\otimes\Lambda\to\cfp(\ul\Sigma,\boldsymbol\beta,\boldsymbol\alpha,z|{-}R_{\beta\alpha};\Gamma_{\eta})\\
f_{\gamma\alpha}&:\sfc(\Sigma,\{\gamma_1,\dots,\gamma_n\},\{\alpha_1,\dots,\alpha_n\})\otimes\Lambda\to\cfp(\ul\Sigma,\boldsymbol\gamma,\boldsymbol\alpha,z|{-}R_{\gamma\alpha};\Gamma_{\eta})
\end{align*} which send a generator $\bx\otimes 1$ to $ [\bx\cup \{u_{\beta\alpha}\}\cup \bx_{\beta\alpha},0]$ and $ [\bx\cup \{u_{\gamma\alpha}\}\cup \bx_{\gamma\alpha},0]$, respectively. We will focus exclusively on the case of $f_{\gamma\alpha}$; the proof of Theorem \ref{thm:iso} for $f_{\beta\alpha}$ proceeds identically.

\begin{lemma}
\label{lem:fchainmap}Sufficiently large $g$ and sufficient winding guarantee that $f_{\gamma\alpha}$ is a chain map. 
\end{lemma}

 As we shall see, this  follows easily from Lemma \ref{lem:stddisk} below. Roughly, this lemma states (in the case  $k=1$) that for  large $g$   and sufficient winding, the Whitney disks with $n_z=0$ between generators of  the Heegaard Floer complex  \[\cfp(\ul\Sigma,\boldsymbol\gamma,\boldsymbol\alpha,z|{-}R_{\gamma\alpha};\Gamma_{\eta})\] with any hope of having holomorphic representatives (i.e. with no negative multiplicities) are supported in the $\Sigma$ portion of $\ul\Sigma$, meaning that they have holomorphic representatives iff the corresponding Whitney disks in the sutured diagram \[(\Sigma,\{\gamma_1,\dots,\gamma_n\},\{\alpha_1,\dots,\alpha_n\})\] do.

\begin{lemma}
\label{lem:stddisk} Fix a pair of intersection points,  \begin{equation}\label{eqn:xy} \bx,\by\in (\gamma_1\times\dots\times\gamma_n)\cap (\alpha_1\times\dots\times\alpha_n)\in \operatorname{Sym}^n(\Sigma)\end{equation} and an integer $k$. For sufficiently large $g$ and sufficient winding, the following is true: for any Whitney disk  \begin{equation*}\label{eqn:whitdisk}\phi\in \pi_2(\bx\cup\{u_{\gamma\alpha}\}\cup \bx_{\gamma\alpha}, \by\cup\{s\}\cup \bx_{\gamma\alpha})\end{equation*} with $s\in \{u_{\gamma\alpha},v_{\gamma\alpha}\}$, where 
\begin{enumerate}
\item $\mu(\phi)=k$, 
\item $n_z(\phi)=0$, and
\item $D(\phi)$ has no negative multiplicities,
\end{enumerate}
we have that
\begin{itemize}
\item  $s=u_{\gamma\alpha}$, and
\item the domain  $D(\phi)$ is supported in $\Sigma\subset \ul\Sigma$. 
\end{itemize}
\end{lemma}

Our proof of this lemma involves a careful balancing of multiplicities against Maslov index along the lines of, and using crucially, the technical Lemmas \ref{lem:tech13} and \ref{lem:tech23}.

\begin{proof}[Proof of Lemma \ref{lem:stddisk}]
Fix  $\bx$ and $\by$ as in \eqref{eqn:xy}, and an integer $k$.
  We will break the proof into two cases.

\noindent \emph{\underline{Case 1: $s=u_{\gamma\alpha}$}}.  Suppose \[g\geq 2pL+2\] as in the hypothesis of Lemma \ref{lem:tech13}, for $L$ as in \eqref{eqn:L}. Fix a Whitney disk in \begin{equation*}\label{eqn:pi2}\pi_2( \bx\cup\{u_{\gamma\alpha}\}\cup \bx_{\gamma\alpha}, \by\cup\{u_{\gamma\alpha}\}\cup \bx_{\gamma\alpha})\end{equation*} with domain $D$ satisfying   $n_z(D)=0$ (if no such disk exists then the lemma holds vacuously). The boundary of $D$ consists of  
\begin{itemize}
\item integer multiples of complete $\gamma_i$ and   $\alpha_i$ curves for $i=n{+}1,\dots,n{+}2g{+}1$, and
\item  integer multiples of arcs of the $\gamma_i$ and  $\alpha_i$ curves for $i=1,\dots,n$.
\end{itemize} Recall from Subsection \ref{ssec:periodic} that for each $i=1,\dots,n{+}2g{+}1$ there is a curve which intersects $\gamma_i$ and $\alpha_i$ positively in exactly one point each and  all other $\boldsymbol\gamma$ and $\boldsymbol\alpha$ curves zero times algebraically. It follows that $\gamma_i$ and $\alpha_i$ appear with opposite multiplicities in the boundary of $D$. We may therefore assume, after adding some integer  linear combination of the elements in the basis \eqref{eqn:basisDextended}, \[\{\P_{\gamma\alpha},\D_1,\dots,\D_m\},\]  that the boundary of $D$ is  disjoint from \begin{itemize}
\item  $\gamma_{i_j}$ and $\alpha_{i_j}$ for $j=p{+}1,\dots,m$, and
\item  $\gamma_{n+2g+1}$ and $\alpha_{n+2g+1}$.
\end{itemize} 
This $D$ will serve as a reference domain for the rest of the proof.

Fix some integer $M$ such that \begin{equation}\label{eqn:choiceN3}M-|k|>|\mu(D)| + \max\{|\ul n_{w_{i_j}}(D)| \mid j=1,\dots,m\}.\end{equation} In particular, this implies that \begin{align}
\label{eqn:muDN}\mu(D) &<M+k, \textrm{ and }\\
\label{eqn:nwDN}\ul n_{w_{i_j}}(D)&<M, \textrm { for } j=1,\dots,m.
\end{align} Note that winding cannot  change the second inequality. Indeed,  winding can only  \emph{potentially} affect \[\ul n_{w_{i_j}}(D)\] for $j=p{+}1,\dots,m$, as discussed in Subsection \ref{ssec:periodic}, and it does not do so in this case since  the boundary of $D$ is  disjoint from $\alpha_{i_j}$ for such $j$, by construction. As mentioned above,  the curves $\gamma_i$ and $\alpha_i$  appear with opposite  multiplicities in the boundary of $D$ for all $i=n{+}1,\dots,n{+}2g$. If this multiplicity is nonzero for some such $i$, then we can ensure that \begin{equation}\label{eqn:windinglots}\ul n_{w_{i}}(D)\leq -2M\end{equation} by winding sufficiently around $\eta_i$, as described in Subsection \ref{ssec:periodic}. After such winding, we may therefore assume  that \eqref{eqn:windinglots} holds for every $i=n{+}1,\dots,n{+}2g$ for which  $\gamma_i$ and $\alpha_i$ appear with nonzero multiplicities in the boundary of $D$, and  that the conclusion of Lemma \ref{lem:tech13} holds for  $N=M$. 

With the above established, let us now suppose $\phi$ is any  Whitney disk, \[\phi\in \pi_2( \bx\cup\{u_{\gamma\alpha}\}\cup \bx_{\gamma\alpha}, \by\cup\{u_{\gamma\alpha}\}\cup \bx_{\gamma\alpha}),\]  where $\mu(\phi)=k$, $n_z(\phi)=0$, and $D(\phi)$ has no negative multiplicities, as in the hypothesis of the lemma. Then we can write \begin{equation}\label{eqn:Dphi}D(\phi)=D+P\end{equation} for some integer linear combination $P$ of the elements of the basis in \eqref{eqn:basisDextended}.

Let $\spc\in \SA(-Y|{-}R)$ denote the $\Sc$ structure represented by the generators \[ \bx\cup\{u_{\gamma\alpha}\}\cup \bx_{\gamma\alpha}\textrm{ and } \by\cup\{u_{\gamma\alpha}\}\cup \bx_{\gamma\alpha}.\] 
We have that \[\mu(\phi) = \mu(D) +\langle c_1(\spc),[P]\rangle,\] by \cite[Theorem 4.9]{osz8}. The fact that  $\mu(\phi)=k$, together with \eqref{eqn:muDN}, then forces  \begin{equation}\label{eqn:c1N3}\langle c_1(\spc),[P]\rangle>-M.\end{equation} Meanwhile, the fact that $D(\phi)$ has no negative multiplicities, together with \eqref{eqn:nwDN},  forces \begin{equation}\label{eqn:nwijN3}\ul n_{w_{i_j}}(P)>-M\end{equation} for  $j=1,\dots,m$.  Now   \eqref{eqn:c1N3} and \eqref{eqn:nwijN3}, together with Lemma \ref{lem:tech13}, imply that the coefficient of $\D_j$ in $P$ is zero for all $j=p{+}1,\dots,m.$ That is,  $P$ is an integer linear combination \[P=a_{\gamma\alpha} \P_{\gamma\alpha} + \sum_{i=1}^p b_i \D_i.\]Therefore,  we can rewrite \eqref{eqn:Dphi} as \begin{equation}\label{eqn:domainD13}D(\phi)=\Big(D+ \sum_{i=1}^p b_i \D_i \Big) + a_{\gamma\alpha} \P_{\gamma\alpha}.\end{equation} We will show next  that the domain in parentheses  is contained in $\Sigma\subset \ul\Sigma$, \begin{equation}\label{eqn:domainsigma}D+ \sum_{i=1}^p b_i \D_i\subset \Sigma.\end{equation} Since this domain is disjoint from the basepoint $z$ and $\D_1,\dots,\D_p$ are contained in $\Sigma$, the only way \eqref{eqn:domainsigma} can fail  is if $\gamma_i$ and $\alpha_i$ appear with (opposite) nonzero multiplicities in the boundary of $D$ for some $i=n{+}1,\dots,n{+}2g$. Suppose, for a contradiction, that this is the case. Then  the fact that $D(\phi)$ has no negative multiplicities, combined with \eqref{eqn:windinglots} and the fact that 
\[n_{w_i^{\pm}}(\P_{\gamma\alpha})= 2,\]   forces \[a_{\gamma\alpha}\geq M.\] We set $P' = P/M$ and  apply the $N=1$ case of Lemma \ref{lem:tech23}\footnote{We can do this because the  hypothesis of Lemma \ref{lem:tech23}  in the $N=1$ case is that $g\geq 2pL+2$, which we have assumed, and that the coefficient of $\P_{\gamma\alpha}$ is $\geq 1$, which is true for $P'$ since $a_{\gamma\alpha}/M\geq 1$.}   to $P'$ to see that either
\begin{itemize}
\item $\langle c_1(\spc),[P']\rangle \leq -1$, or
\item $\underline{n}_{w_{i_j}}(P') \leq -1$ for some $j \in \{1,\dots,p\}$,
\end{itemize}
which is equivalent to either
\begin{itemize}
\label{eqn:cN1}\item $\langle c_1(\spc),[P]\rangle \leq -M$, or
\label{eqn:cN2}\item $\ul n_{w_{i_j}}(P)\leq -M$ for some $j\in\{1,\dots,p\}$.\footnote{The reader may wonder why we cannot conclude these inequalities directly from Lemma \ref{lem:tech23}; that is, why dividing by $M$ above is a necessary step. The answer is that to apply Lemma \ref{lem:tech23} directly we would need that $g\geq pLM+pL+M+1$ while we have only assumed that $g\geq 2pL+2$. We do not want to assume $g\geq pLM+pL+M+1$ at the outset, as this $M$ may depend on the specific diagram for the closure.}
\end{itemize} But these contradict either \eqref{eqn:c1N3} or \eqref{eqn:nwijN3}. We may conclude that \eqref{eqn:domainsigma} holds. Furthermore, note that \eqref{eqn:domainD13} and \eqref{eqn:domainsigma} imply that \[D + \sum_{i=1}^p b_i \D_i\] is the domain of a Whitney disk in $\pi_2(\bx,\by).$ In summary, we have shown that for $g\geq 2pL+2$ and sufficient winding, any Whitney disk $\phi$ as in the statement of the lemma has domain \begin{equation}\label{eqn:domainD23}D(\phi) = D(\phi') + a_{\gamma\alpha}\P_{\gamma\alpha},\end{equation} for some nonnegative integer $a_{\gamma\alpha}$ (if it is not nonnegative then $D(\phi)$  has a negative multiplicity), where  \begin{equation}\label{eqn:phi'3}\phi'\in\pi_2(\bx,\by)\end{equation} is a Whitney disk in the sutured Heegaard  diagram \begin{equation}\label{eqn:shd}(\Sigma,\{\gamma_1,\dots,\gamma_n\},\{\alpha_1,\dots,\alpha_n\}).\end{equation}  

To complete the proof of the lemma in this case, we show next that for sufficiently large $g$ and sufficient winding it must also be true that $a_{\gamma\alpha}=0$ in \eqref{eqn:domainD23}. For this, fix some  \[\phi''\in\pi_2(\bx,\by)\] for reference. Fix some  $N$ with\begin{equation}\label{eqn:choiceN23}N-|k|>|\mu(\phi'')| + \max\{|\ul n_{w_i}(\phi'')| \mid i=1,\dots,n\}.\end{equation} Note that the range of values for $N$ which satisfy this inequality depends only on $k$ and the sutured diagram. As before, this implies that \begin{align}
\label{eqn:muDN2}\mu(\phi'') &<N+k, \textrm{ and }\\
\label{eqn:nwDN2}\ul n_{w_i}(\phi'')&<N, \textrm { for } i=1,\dots,n.
\end{align}Suppose  \begin{equation}\label{eqn:genuslem3}g\geq \max\{2pL+2,pLN+pL+N+1\},\end{equation} and   wind sufficiently  that any  Whitney disk $\phi$ as in the statement of the lemma can be written in the form \eqref{eqn:domainD23}.  
Suppose $\phi$ is such a  disk, and write $D(\phi)$ in this form, \[D(\phi) = D(\phi') +a_{\gamma\alpha} \P_{\gamma\alpha}.\] The domains $D(\phi')$ and $D(\phi'')$ differ by a doubly-periodic domain in the sutured Heegaard  diagram \eqref{eqn:shd}. We can therefore write  \[D(\phi') = D(\phi'') + \sum_{i=1}^p b_i \D_i\] for some integers $b_1,\dots,b_p$. We  thus have \[D(\phi) = D(\phi'')  + P,\] where  \[P=a_{\gamma\alpha} \P_{\gamma\alpha} +\sum_{i=1}^p b_i \D_i.\]   Therefore, \[\mu(\phi)=\mu(\phi'')+\langle c_1(\spc),[P]\rangle.\] The fact that  $\mu(\phi)=k$, together with \eqref{eqn:muDN2}, then forces  \begin{equation}\label{eqn:c1P3}\langle c_1(\spt),[P]\rangle>-N.\end{equation} Suppose for a contradiction that $a_{\gamma\alpha}\neq 0$. Then $a_{\gamma\alpha}\geq 1$ since it is an integer. But then Lemma \ref{lem:tech23}, together with \eqref{eqn:c1P3}, implies that  \[\ul n_{w_{i_j}}(P)\leq -N\] for some $j=1,\dots,p$. But this, together with \eqref{eqn:nwDN2}, implies that \[\ul n_{w_{i_j}}(D(\phi))<0,\]  a contradiction. This shows that $a_{\gamma\alpha}=0$, proving Lemma \ref{lem:stddisk} in this case.

\noindent \emph{\underline{Case 2: $s=v_{\gamma\alpha}$}}. This case follows quickly from the previous case. Let $B$ be the bigon  shown in Figure \ref{fig:bigon}, with vertices at $u_{\gamma\alpha}$ and $v_{\gamma\alpha}$ and $n_z(B)=1$. 

\begin{figure}[ht]
\labellist
\tiny

\pinlabel $u_{\gamma\alpha}$ at 150 266
\pinlabel $v_{\gamma\alpha}$ at 150 187
\pinlabel $z$ at 155 228

\endlabellist
\centering
\includegraphics[width=4cm]{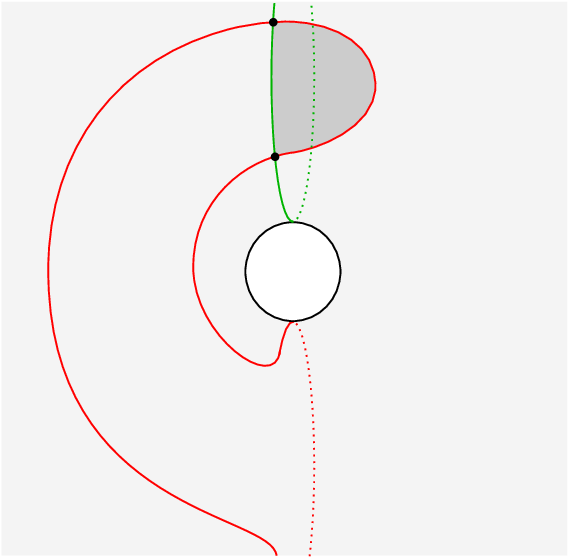}
\caption{The bigon $B$, shaded.}
\label{fig:bigon}
\end{figure}

\noindent Observe that for any \[\phi\in\pi_2( \bx\cup\{u_{\gamma\alpha}\}\cup \bx_{\gamma\alpha}, \by\cup\{v_{\gamma\alpha}\}\cup \bx_{\gamma\alpha})\] where $\mu(\phi)=k$, $n_z(\phi)=0$, and $D(\phi)$ has no negative multiplicities, as in the hypothesis of the lemma, we can write \begin{equation}\label{eqn:dphi-3}D(\phi) = D(\phi') + B-\ul\Sigma\end{equation} for some Whitney disk \[\phi'\in\pi_2( \bx\cup\{u_{\gamma\alpha}\}\cup \bx_{\gamma\alpha}, \by\cup\{u_{\gamma\alpha}\}\cup \bx_{\gamma\alpha})\] where $\mu(\phi')=k+1$, $n_z(\phi')=0$, and $D(\phi')$ has no negative multiplicities ($B-\ul\Sigma$ contributes $1-2=-1$ to the Maslov index). We proved in the previous case that for sufficiently large $g$ and sufficient winding, any such $\phi'$ satisfies \[D(\phi')\subset \Sigma.\] Suppose then that $g$ is large enough and we have wound sufficiently that this holds, and let $\phi$ be as above. Then $D(\phi)$ has a negative multiplicity in the region $\ul R_S$ by \eqref{eqn:dphi-3}, since $B-\ul\Sigma$ does and $D(\phi')$ has multiplicity zero in this region, a contradiction. We conclude that for sufficiently large $g$ and sufficient winding, there is no Whitney disk $\phi$ as in the statement of the lemma when $s=v_{\gamma\alpha}.$
 \end{proof}

We now explain how Lemma \ref{lem:stddisk} implies Lemma \ref{lem:fchainmap}.

\begin{proof}[Proof of Lemma \ref{lem:fchainmap}]
Suppose   $g$ is large enough and the winding sufficient for the conclusion of Lemma \ref{lem:stddisk} to hold. It suffices to show (by Lemma \ref{lem:bot}), for each pair $\bx, \by$ as in \eqref{eqn:xy} that the coefficient of $[\by\cup \{s\}\cup \bx_{\gamma\alpha},0]$ in $f_{\gamma\alpha}(d \bx\otimes 1)$ is the same as its coefficient in $\partial f_{\gamma\alpha}(\bx\otimes 1)$, for $s = u_{\gamma\alpha}$ or $v_{\gamma\alpha}$, where $d$ is the differential on \[\sfc(\Sigma,\{\gamma_1,\dots,\gamma_n\},\{\alpha_1,\dots,\alpha_n\})\] and $\partial$ is the differential on $\cfp(\ul\Sigma,\boldsymbol\gamma,\boldsymbol\alpha,z|{-}R_{\gamma\alpha};\Gamma_{\eta})$. Note that both coefficients 
\[ \langle f_{\gamma\alpha}(d \bx\otimes 1), [\by\cup \{s\}\cup \bx_{\gamma\alpha},0]\rangle\,\textrm{ and } \langle \partial f_{\gamma\alpha}(\bx\otimes 1), [\by\cup \{s\}\cup \bx_{\gamma\alpha},0]\rangle\] are zero if $s=v_{\gamma\alpha}$. This is by definition for the first and by Lemma \ref{lem:stddisk} in the case $k=1$ for the second. We therefore only need consider the case $s=u_{\gamma\alpha}$. By definition, we have that  \begin{align}
\label{eqn:c1}\langle f_{\gamma\alpha}(d \bx\otimes 1), [\by\cup \{u_{\gamma\alpha}\}\cup \bx_{\gamma\alpha},0]\rangle &= \langle d \bx,\by\rangle,\\
\label{eqn:c2}\langle \partial f_{\gamma\alpha}(\bx\otimes 1), [\by\cup \{u_{\gamma\alpha}\}\cup \bx_{\gamma\alpha},0]\rangle &= \langle \partial([\bx\cup \{u_{\gamma\alpha}\}\cup \bx_{\gamma\alpha},0]), [\by\cup \{u_{\gamma\alpha}\}\cup \bx_{\gamma\alpha},0]\rangle.
\end{align} But it follows immediately from Lemma \ref{lem:stddisk} that the coefficients on the right hand sides of \eqref{eqn:c1} and \eqref{eqn:c2} are equal: any Whitney disk contributing to the coefficient in \eqref{eqn:c1} contributes the same amount to the coefficient in \eqref{eqn:c2}, and Lemma \ref{lem:stddisk} tells us that the converse is true (note that any domain contained in $\Sigma$ is disjoint from $\eta$).
\end{proof}

We will henceforth assume that $g$ is sufficiently large and that we have wound sufficiently to guarantee that $f_{\gamma\alpha}$ is a chain map. 

To show that $f_{\gamma\alpha}$ is a quasi-isomorphism (assuming sufficiently large $g$ and sufficient winding), we will show that it is a filtered chain map for some filtrations on the domain and codomain complexes, and that it induces an isomorphism between $E_1$ pages of the spectral sequences associated to these filtrations. We will first define a  filtration on the codomain  $\cfp(\boldsymbol\gamma,\boldsymbol\alpha|{-}R_{\gamma\alpha};\Gamma_{\eta})$ in each $\Sc$ structure.

Let $E$ and $F$ be the points in $\ul R_S$ and $\ul R_\Sigma$ shown in Figures \ref{fig:uvuv2} and \ref{fig:periodic}. Given generators \[[\bx\cup\{s\}\cup \bx_{\gamma\alpha}, i] \textrm{ and } [\by\cup\{s'\}\cup \bx_{\gamma\alpha},j] \textrm{ of } \cfp(\ul\Sigma,\boldsymbol\gamma,\boldsymbol\alpha,z|{-}R_{\gamma\alpha};\Gamma_{\eta})\] representing the same $\Sc$ structure, for $\bx,\by$ as in \eqref{eqn:xy} and $s,s'\in\{u_{\gamma\alpha},v_{\gamma\alpha}\}$,  choose a Whitney disk \[\phi\in\pi_2(\bx\cup\{s\}\cup \bx_{\gamma\alpha},\by\cup\{s'\}\cup \bx_{\gamma\alpha})\] with $n_z(\phi) = i-j$, and define the relative grading \begin{equation*}\label{eqn:relfilt} \mathcal{F}([\bx\cup\{s\}\cup \bx_{\gamma\alpha}, i],[\by\cup\{s'\}\cup \bx_{\gamma\alpha},j]) = 2n_E(\phi)-n_F(\phi). \end{equation*}
This relative grading is well-defined since any two such $\phi$ differ by a periodic domain $P$ and \[2n_E(P){-}n_F(P)=0\] for all $(\boldsymbol\gamma,\boldsymbol\alpha)$-periodic domains, by \eqref{eqn:EFzero}.  Moreover, we have the following.

 \begin{lemma}
 \label{lem:efphi}
 For $\bx$ and $\by$ as above, we have that  \[2n_E(\phi){-}n_F(\phi)=0\] for any \[\phi\in\pi_2(\bx\cup\{s\}\cup \bx_{\gamma\alpha},\by\cup\{s'\}\cup \bx_{\gamma\alpha})\] with  $s=s'$ and $n_z(\phi)=0$.
 \end{lemma}
 
 \begin{proof}
 As in \eqref{eqn:Dphi},  we can write \[D(\phi)=D+P\] for some domain $D$ with $n_z(D)=0$ whose boundary is disjoint from $\gamma_{n+2g+1}$ and $\alpha_{n+2g+1}$, and some  linear combination $P$ of the elements of the basis in \eqref{eqn:basisDextended}.
We then have  \[2n_E(P)-n_F(P)=0=2n_E(D)-n_F(D).\] The first equality follows  from \eqref{eqn:EFzero}, while the second follows from the facts that $n_z(D)=0$ and that the boundary of $D$ is disjoint from $\gamma_{n+2g+1}$ and $\alpha_{n+2g+1}$. These equalities then imply that $2n_E(\phi)-n_F(\phi)=0$. \end{proof}
 In addition, we note here that \begin{align*}
2n_E(\ul\Sigma)-n_F(\ul\Sigma)&=1,\\
2n_E(B)-n_F(B)&=0,
\end{align*}
 where $B$ is the bigon in Figure \ref{fig:bigon}. These facts will be useful for Lemmas \ref{lem:filt} and \ref{lem:d0} below. For each $\Sc$ structure, let us choose some lift of this relative grading to an absolute grading, which we also denote by $\mathcal{F}$.

\begin{lemma}
\label{lem:filt}
Sufficiently large $g$ and sufficient winding guarantee that the absolute grading $\mathcal{F}$ defines a filtration.
\end{lemma}

\begin{proof} Fix $\bx,\by$ as in \eqref{eqn:xy}. We must show that for sufficiently large $g$ and sufficient winding, the following is true: if the coefficient \begin{equation}\label{eqn:coeff}\langle \partial([\bx\cup\{s\}\cup \bx_{\gamma\alpha}, i]),[\by\cup\{s'\}\cup \bx_{\gamma\alpha},j]\rangle\end{equation} is nonzero, then \[\mathcal{F}([\bx\cup\{s\}\cup \bx_{\gamma\alpha}, i])\geq\mathcal{F}([\by\cup\{s'\}\cup \bx_{\gamma\alpha},j]).\] We will break the proof into three cases.

\noindent \emph{\underline{Case 1: $s=s'$}.} Suppose the coefficient in \eqref{eqn:coeff} is nonzero.
Then there is a Whitney disk \[\phi\in\pi_2(\bx\cup\{s\}\cup \bx_{\gamma\alpha},\by\cup\{s\}\cup \bx_{\gamma\alpha})\] with  $n_z(\phi) = i-j\geq 0$.  We can  write   \[D(\phi) = D(\phi')+ (i-j)\ul\Sigma,\] for some \[\phi'\in\pi_2(\bx\cup\{s\}\cup \bx_{\gamma\alpha},\by\cup\{s\}\cup \bx_{\gamma\alpha})\] with $n_z(\phi')=0$. We  therefore have that \begin{align*}\mathcal{F}([\bx\cup\{s\}\cup \bx_{\gamma\alpha}, i])-\mathcal{F}([\by\cup\{s\}\cup \bx_{\gamma\alpha},j])&=2n_E(\phi)-n_F(\phi) \\
&= 2n_E(\phi')-n_F(\phi')+ i-j\\
&=i-j \geq 0,
\end{align*} where the last equality follows from Lemma \ref{lem:efphi}.

\noindent \emph{\underline{Case 2: $(s,s')=(u_{\gamma\alpha},v_{\gamma\alpha})$}}. 
Suppose the coefficient in \eqref{eqn:coeff} is nonzero.
Then there is a Whitney disk \[\phi\in\pi_2(\bx\cup\{u_{\gamma\alpha}\}\cup \bx_{\gamma\alpha},\by\cup\{v_{\gamma\alpha}\}\cup \bx_{\gamma\alpha})\] with $\mu(\phi)=1$,  no negative multiplicities, and  $n_z(\phi) = i-j\geq 0$. We can  write     \[D(\phi) = D(\phi')+ B+ (i-j-1)\ul\Sigma,\] for some  \[\phi'\in\pi_2(\bx\cup\{u_{\gamma\alpha}\}\cup \bx_{\gamma\alpha},\by\cup\{u_{\gamma\alpha}\}\cup \bx_{\gamma\alpha})\] with $n_z(\phi')=0$ and \[\mu(\phi')=-2(i-j-1).\] If $i-j-1<0$, which since $i-j\geq 0$  means that \[i-j-1=-1,\] then $D(\phi')$  has no negative multiplicities; otherwise, $D(\phi)$  clearly would as well, a contradiction. But  Lemma \ref{lem:stddisk}  in the case \[k=-2(i-j-1)=2\] shows that for sufficiently large $g$  and sufficient winding,  $D(\phi')\subset \Sigma$ for  all such $\phi'$. However, in this case, $D(\phi)$ has negative multiplicities in the region $\ul R_S$ since \[B+ (i-j-1)\ul\Sigma\] does and $D(\phi')$ has multiplicity zero in this region, another contradiction. We conclude that for sufficiently large $g$ and sufficient winding, $i-j-1\geq 0$, in which case \begin{align*}
\mathcal{F}([\bx\cup\{u_{\gamma\alpha}\}\cup \bx_{\gamma\alpha}, i])-\mathcal{F}([\by\cup\{v_{\gamma\alpha}\}\cup \bx_{\gamma\alpha},j])&= 2n_E(\phi')-n_F(\phi') + 2n_E(B)-n_F(B)\\
&+ i-j-1\\
&=i-j-1\geq 0,\end{align*} as desired. 

\noindent \emph{\underline{Case 3: $(s,s')=(v_{\gamma\alpha},u_{\gamma\alpha})$}}. 
Suppose the coefficient in \eqref{eqn:coeff} is nonzero.
Then there is a Whitney disk \[\phi\in\pi_2(\bx\cup\{v_{\gamma\alpha}\}\cup \bx_{\gamma\alpha},\by\cup\{u_{\gamma\alpha}\}\cup \bx_{\gamma\alpha})\] with  $n_z(\phi) = i-j\geq 0$. We can then write  \[D(\phi) = D(\phi')-B + (i-j+1)\ul\Sigma\] for some \[\phi'\in\pi_2(\bx\cup\{u_{\gamma\alpha}\}\cup \bx_{\gamma\alpha},\by\cup\{u_{\gamma\alpha}\}\cup \bx_{\gamma\alpha})\] with $n_z(\phi')=0$. We therefore have that \begin{align*}\mathcal{F}([\bx\cup\{v_{\gamma\alpha}\}\cup \bx_{\gamma\alpha}, i])-\mathcal{F}([\by\cup\{u_{\gamma\alpha}\}\cup \bx_{\gamma\alpha},j])&=2n_E(\phi')-n_F(\phi') - 2n_E(B)+n_F(B)\\
&+i-j+1\\
&= i-j+1> 0,\end{align*} as desired.
\end{proof}

Let $\partial_0$ denote the component of the differential $\partial$ on $\cfp(\ul\Sigma,\boldsymbol\gamma,\boldsymbol\alpha,z|{-}R_{\gamma\alpha};\Gamma_{\eta})$ which preserves the grading $\mathcal{F}$, and let $d$ denote the differential on \[\sfc(\Sigma,\{\gamma_1,\dots,\gamma_n\},\{\alpha_1,\dots,\alpha_n\})\otimes\Lambda,\] as in the proof of Lemma \ref{lem:fchainmap}. We have the following.

\begin{lemma}
\label{lem:d0}
Sufficiently large $g$ and sufficient winding guarantee that for each generator \[[\bx\cup\{s\}\cup \bx_{\gamma\alpha},i]\textrm{ of }\cfp(\ul\Sigma,\boldsymbol\gamma,\boldsymbol\alpha,z|{-}R_{\gamma\alpha};\Gamma_{\eta}),\] we have that 
\begin{equation}\label{eqn:d0}\partial_0([\bx\cup\{s\}\cup \bx_{\gamma\alpha},i]) = 
\begin{cases}
[d(\bx)\cup\{u_{\gamma\alpha}\}\cup \bx_{\gamma\alpha},i] + [\bx\cup\{v_{\gamma\alpha}\}\cup \bx_{\gamma\alpha},i-1],&\textrm{ if } s=u_{\gamma\alpha},\\
[d(\bx)\cup\{v_{\gamma\alpha}\}\cup \bx_{\gamma\alpha},i],&\textrm{ if } s=v_{\gamma\alpha}.
\end{cases}\end{equation}
For $i=0$, the term $[\bx\cup\{v_{\gamma\alpha}\}\cup \bx_{\gamma\alpha},i-1]$ above is interpreted as zero in $\cfp$.
\end{lemma} 

\begin{proof}
Fix $\bx,\by$ as in \eqref{eqn:xy}. We will break the proof into three cases.

\noindent \emph{\underline{Case 1: $s=s'$}}. 
 Suppose the coefficient \begin{equation}\label{eqn:coeffxy}\langle\partial_0([\bx\cup\{s\}\cup \bx_{\gamma\alpha},i]),[\by\cup\{s\}\cup \bx_{\gamma\alpha},j]\rangle\end{equation}   is nonzero.  Then there is a Whitney disk \[\phi\in\pi_2(\bx\cup\{s\}\cup \bx_{\gamma\alpha},\by\cup\{s\}\cup \bx_{\gamma\alpha})\] with  $\mu(\phi)=1$, no negative multiplicities, and $n_z(\phi) = i-j\geq 0$. We can then write   \[D(\phi) = D(\phi')+ (i-j)\ul\Sigma,\] for some \[\phi'\in\pi_2(\bx\cup\{s\}\cup \bx_{\gamma\alpha},\by\cup\{s\}\cup \bx_{\gamma\alpha})\] with $n_z(\phi')=0$. By Lemma \ref{lem:efphi}, we have $2n_E(\phi')-n_F(\phi')=0$. Therefore,
\[2n_E(\phi)-n_F(\phi) =  (i-j)(2n_E(\ul\Sigma)-n_F(\ul\Sigma))=i-j.\]
 But then $i-j=0$ since since $\partial_0$ preserves $\mathcal{F}$.  Thus, \[n_z(\phi) = i-j=0.\] Lemma \ref{lem:stddisk}  in the case $k=1$ then shows that for sufficiently large $g$ and sufficient winding, $D(\phi)\subset\Sigma$ for such $\phi$.  It follows that for sufficiently large $g$ and sufficient winding, \begin{equation}\label{eqn:coeffequality}\langle\partial_0([\bx\cup\{s\}\cup \bx_{\gamma\alpha},i]),[\by\cup\{s\}\cup \bx_{\gamma\alpha},j]\rangle = \begin{cases}
\langle d(\bx),\by\rangle, & \textrm{if } j=i,\\
0,& \textrm{ otherwise,}
\end{cases}
\end{equation} as desired.

\noindent \emph{\underline{Case 2: $(s,s')=(u_{\gamma\alpha},v_{\gamma\alpha})$}}. 
Suppose the coefficient \begin{equation}\label{eqn:coeffxy2}\langle\partial_0([\bx\cup\{u_{\gamma\alpha}\}\cup \bx_{\gamma\alpha},i]),[\by\cup\{v_{\gamma\alpha}\}\cup \bx_{\gamma\alpha},j]\rangle\end{equation}   is nonzero.  Then there is a Whitney disk \[\phi\in\pi_2(\bx\cup\{u_{\gamma\alpha}\}\cup \bx_{\gamma\alpha},\by\cup\{v_{\gamma\alpha}\}\cup \bx_{\gamma\alpha})\] with $\mu(\phi)=1$,  no negative multiplicities, and   $n_z(\phi) = i-j\geq 0$. We can  write \[D(\phi) = D(\phi')+B +(i-j-1)\ul\Sigma\] for some  \[\phi'\in\pi_2(\bx\cup\{u_{\gamma\alpha}\}\cup \bx_{\gamma\alpha},\by\cup\{u_{\gamma\alpha}\}\cup \bx_{\gamma\alpha})\] with $n_z(\phi')=0$ and \[\mu(\phi')=-2(i-j-1).\] By Lemma \ref{lem:efphi}, we have $2n_E(\phi')-n_F(\phi')=0$. This then implies that 
 \begin{align*}2n_E(\phi)-n_F(\phi) &= 2n_E(B)-n_F(B) + (i-j-1)(2n_E(\ul\Sigma)-n_F(\ul\Sigma))\\
 &=i-j-1.\end{align*} But this quantity must be zero since $\partial_0$ preserves $\mathcal{F}$, so we can write \[D(\phi) = D(\phi')+B,\] and $\mu(\phi')=0$. Furthermore, $\phi'$ must have no negative multiplicities; otherwise, $\phi$  clearly would as well, a contradiction. But  Lemma \ref{lem:stddisk} in the case $k=0$ says that for sufficiently large $g$ and sufficient winding, $D(\phi')\subset\Sigma$ for all such $\phi'$. In this  case, the constituent pieces $D(\phi')$ and $B$ of $D(\phi)$ are disjoint domains with holomorphic representatives. Since holomorphic disks of Maslov index zero are constant, we have $D(\phi')=0$, which implies that\[D(\phi)=B.\] In this case, $\phi$ has a unique holomorphic representative, and $\by=\bx$. We may therefore conclude that for sufficiently large $g$ and  sufficient winding, \begin{equation}\label{eqn:coeffequality2}\langle\partial_0([\bx\cup\{u_{\gamma\alpha}\}\cup \bx_{\gamma\alpha},i]),[\by\cup\{v_{\gamma\alpha}\}\cup \bx_{\gamma\alpha},j]\rangle = \begin{cases}
1, & \textrm{if } \by=\bx  \textrm{ and } j=i-1,\\
0,& \textrm{ otherwise,}
\end{cases}
\end{equation} as desired.

\noindent \emph{\underline{Case 3: $(s,s')=(v_{\gamma\alpha},u_{\gamma\alpha})$}}.  
Suppose the coefficient \begin{equation}\label{eqn:coeffxy3}\langle\partial_0([\bx\cup\{v_{\gamma\alpha}\}\cup \bx_{\gamma\alpha},i]),[\by\cup\{u_{\gamma\alpha}\}\cup \bx_{\gamma\alpha},j]\rangle\end{equation}   is nonzero.  Then there is a Whitney disk \[\phi\in\pi_2(\bx\cup\{v_{\gamma\alpha}\}\cup \bx_{\gamma\alpha},\by\cup\{u_{\gamma\alpha}\}\cup \bx_{\gamma\alpha})\] with   $n_z(\phi) = i-j\geq 0$. Note that \[2n_E(\phi)-n_F(\phi) = i-j+1,\] as established in the proof of Lemma \ref{lem:filt} in this case. But this quantity must be zero since $\partial_0$ preserves $\mathcal{F}$, so $i-j = -1.$ But this contradicts the assumption that $i-j\geq 0$. Therefore, \begin{equation}\label{eqn:coeffequality3}\langle\partial_0([\bx\cup\{v_{\gamma\alpha}\}\cup \bx_{\gamma\alpha},i]),[\by\cup\{u_{\gamma\alpha}\}\cup \bx_{\gamma\alpha},j]\rangle = 0\end{equation} for all $i,j$ and $\bx,\by$. Putting the formulae \eqref{eqn:coeffequality}, \eqref{eqn:coeffequality2}, and \eqref{eqn:coeffequality3} together completes the proof of Lemma \ref{lem:d0}.
\end{proof}

Suppose now that $g$ is large enough and that we have wound sufficiently for the conclusions of Lemmas  \ref{lem:fchainmap}, \ref{lem:stddisk}, \ref{lem:filt}, and \ref{lem:d0} to hold. Note that the above filtration on $\cfp(\ul\Sigma,\boldsymbol\gamma,\boldsymbol\alpha,z|{-}R_{\gamma\alpha};\Gamma_{\eta})$  defines a filtration on \[\sfc(\Sigma,\{\gamma_1,\dots,\gamma_n\},\{\alpha_1,\dots,\alpha_n\})\otimes\Lambda\] by simply declaring the filtration grading of a generator $\bx$ to be equal to that of $[\bx\cup\{u_{\gamma\alpha}\}\cup \bx_{\gamma\alpha},0]$. In particular, $d=d_0$, where $d_0$ is the component of $d$ which preserves the filtration grading on the sutured Floer complex, and $f_{\gamma\alpha}$ is a filtered chain map. The $E_1$ page of the spectral sequence associated to the filtration on this sutured Floer complex is therefore simply the homology 
\[H_*(\sfc(\Sigma,\{\gamma_1,\dots,\gamma_n\},\{\alpha_1,\dots,\alpha_n\})\otimes\Lambda, d_0) = \sfh(\Sigma,\{\gamma_1,\dots,\gamma_n\},\{\alpha_1,\dots,\alpha_n\})\otimes\Lambda.\] We claim the following.

\begin{lemma}
\label{lem:qi}
The map between $E_1$ pages induced by $f_{\gamma\alpha}$, \[E_1(f_{\gamma\alpha}): \sfh(\Sigma,\{\gamma_1,\dots,\gamma_n\},\{\alpha_1,\dots,\alpha_n\})\otimes\Lambda\to H_*(\cfp(\ul\Sigma,\boldsymbol\gamma,\boldsymbol\alpha,z|{-}R_{\gamma\alpha};\Gamma_{\eta}), \partial_0),\] is an isomorphism.\end{lemma}

\begin{proof}
We claim that every generator of the homology \[H_*(\cfp(\ul\Sigma,\boldsymbol\gamma,\boldsymbol\alpha,z|{-}R_{\gamma\alpha};\Gamma_{\eta}), \partial_0)\] is represented by a linear combination of generators of the form $[\bx\cup\{u_{\gamma\alpha}\}\cup \bx_{\gamma\alpha},0]$. To see how the lemma follows from this claim, suppose it is true and recall that $E_1(f_{\gamma\alpha})$ is induced by the map which sends a generator $\bx$ to $[\bx\cup\{u_{\gamma\alpha}\}\cup \bx_{\gamma\alpha},0]$. In particular, this map sends a linear combination \begin{equation}\label{eqn:linc1}\bx_1\otimes r_1+\dots + \bx_k\otimes r_k,\end{equation} where the $r_i\in\Lambda$,  to the linear combination \begin{equation}\label{eqn:linc2}[\bx_1\cup\{u_{\gamma\alpha}\}\cup \bx_{\gamma\alpha},0]r_1+\dots + [\bx_k\cup\{u_{\gamma\alpha}\}\cup \bx_{\gamma\alpha},0]r_k.\end{equation} It follows easily from Lemma \ref{lem:d0} that the sum in \eqref{eqn:linc1} is a cycle (resp.\ boundary) with respect to $d=d_0$  iff the sum in \eqref{eqn:linc2} is a cycle (resp.\ boundary) with respect to $\partial_0$. This implies that $E_1(f_{\gamma\alpha})$ is an isomorphism. 

It remains to prove the claim. Given a linear combination \begin{equation}\label{eqn:linc3}w=\bx_1\otimes r_1+\dots + \bx_k\otimes r_k,\end{equation} as in \eqref{eqn:linc1}, let us use the following notation
\begin{align*}
[w,u_{\gamma\alpha},i]&:=[\bx_1\cup\{u_{\gamma\alpha}\}\cup \bx_{\gamma\alpha},i]r_1+\dots + [\bx_k\cup\{u_{\gamma\alpha}\}\cup \bx_{\gamma\alpha},i]r_k,\\
[w,v_{\gamma\alpha},i]&:=[\bx_1\cup\{v_{\gamma\alpha}\}\cup \bx_{\gamma\alpha},i]r_1+\dots + [\bx_k\cup\{v_{\gamma\alpha}\}\cup \bx_{\gamma\alpha},i]r_k.
\end{align*}
Now suppose \[c = [w_i,u_{\gamma\alpha},i] + [z_i,v_{\gamma\alpha},i] + \dots + [w_0,u_{\gamma\alpha},0] + [z_0,v_{\gamma\alpha},0]\] is a cycle with respect to $\partial_0$, where the $w_j$ and $v_j$ are linear combinations as in \eqref{eqn:linc3}. For the claim, it suffices to show that there is some \[b\in \cfp(\ul\Sigma,\boldsymbol\gamma,\boldsymbol\alpha,z|{-}R_{\gamma\alpha};\Gamma_{\eta})\] such that \begin{equation}\label{eqn:db}\partial_0b + c=[w_0,u_{\gamma\alpha},0].\end{equation} Applying the formula for $\partial_0$ in \eqref{eqn:d0}, one easily sees that the fact that $c$ is a cycle implies that \begin{align*}
dz_i&=0,\\
dz_{i-1}&=w_i,\\
dz_{i-2}&=w_{i-1},\\
\vdots&\\
dz_{0} &= w_1.\end{align*}
It therefore follows, after another application of \eqref{eqn:d0}, that  \[b=[z_i,u_{\gamma\alpha},i+1]+ \dots + [z_0,u_{\gamma\alpha},1]\] satisfies \eqref{eqn:db}. This completes the proof of the lemma.
\end{proof}

\begin{proof}[Proof of Theorem \ref{thm:iso}] 
The fact that $f_{\gamma\alpha}$ is  a quasi-isomorphism follows immediately. This is because a filtered chain map between filtered chain complexes which induces an isomorphism between the $E_1$ pages of the associated spectral sequences induces an isomorphism on homology, assuming that the filtrations are bounded from below, which they clearly are in this case (see, e.g., the proof of \cite[Proposition A.6.1]{oszbook}).
\end{proof}

\subsection{Proof of Theorem  \ref{thm:map}} 
\label{ssec:proofmap}
In preparation for the proof of Theorem \ref{thm:map}, we introduce the following notation. Let  $\Delta^i$ be the small triangle with vertices at 
\[\begin{cases}
\Theta^i, c_{\beta\alpha}^i,  c_{\gamma\alpha}^i, & \textrm{for } i=1,\dots,n,\\
\Theta^i, x_{\beta\alpha}^i,  x_{\gamma\alpha}^i, & \textrm{for } i=n{+}1,\dots,n{+}2g,\\
\Theta^{i}, u_{\beta\alpha},  u_{\gamma\alpha}, & \textrm{for } i=n{+}2g{+}1,\\
\end{cases}\]
shown shaded in Figures \ref{fig:HH2}, \ref{fig:TSarcs2}, and  \ref{fig:uvuv2}, and let 
\begin{align*}
\Delta_\Sigma&=\Delta^{1}+ \dots+ \Delta^{n},\\
\Delta_S& = \Delta^{n+1}+ \dots+ \Delta^{n{+}2g{+}1},\\
\Delta_{\ul\Sigma} &= \Delta_\Sigma+ \Delta_S.
\end{align*}
From Lemma \ref{lem:bot}, the image \[f^{+}_{\gamma\beta\alpha;\Gamma_\nu}([\bcc_{\beta\alpha}\cup\{u_{\beta\alpha}\}\cup \bx_{\beta\alpha},0])\] is a linear combination of generators of the form $[\by\cup\{s\}\cup \bx_{\gamma\alpha},0],$ where $s\in\{u_{\gamma\alpha},v_{\gamma\alpha}\}$.

As we shall see, Theorem \ref{thm:map} follows easily from Lemma \ref{lem:stdtri} below, which is a kind of Whitney triangle version of Lemma \ref{lem:stddisk}. 
 
\begin{lemma}
\label{lem:stdtri} Fix an intersection point \begin{equation}\label{eqn:y} \by\in (\gamma_1\times\dots\times\gamma_n)\cap (\alpha_1\times\dots\times\alpha_n)\in \operatorname{Sym}^n(\Sigma)\end{equation} and an integer $k$. For sufficiently large $g$ and sufficient winding, the following is true: for any Whitney triangle \begin{equation*}\label{eqn:whit}\phi\in \pi_2(\Theta, \bcc_{\beta\alpha}\cup\{u_{\gamma\alpha}\}\cup \bx_{\beta\alpha}, \by\cup\{s\}\cup \bx_{\gamma\alpha})\end{equation*} with $s\in \{u_{\gamma\alpha},v_{\gamma\alpha}\}$, where
\begin{enumerate}
\item $\mu(\phi)=k$,
\item $n_z(\phi)=0$,
\item $D(\phi)$ has no negative multiplicities, and
\item $\hfp(-W,\spt;\Gamma_\nu)([[\bcc_{\beta\alpha}\cup\{u_{\gamma\alpha}\}\cup \bx_{\beta\alpha},0]]) \neq 0$ for $\spt=\spc_z(\phi)$,  
\end{enumerate} 
we have that 
\begin{itemize}
\item $s=u_{\gamma\alpha}$,  
\item $\by = \bcc_{\gamma\alpha}$, and  
\item the domain  $D(\phi)=\Delta_{\ul\Sigma}$. 
\end{itemize}

\end{lemma}

The proof of this lemma is very similar to that of Lemma \ref{lem:stddisk}. We therefore omit some of the redundant details in our explanation below.

\begin{proof}[Proof of Lemma \ref{lem:stdtri}]  Fix  $\by$ as in \eqref{eqn:y}. We will break the proof into two cases.

\noindent \emph{\underline{Case 1: $s=u_{\gamma\alpha}$}}.    Suppose \[g\geq 2pL+2\] as in the hypothesis of Lemma \ref{lem:tech1}. Suppose $\spt\in\Sc(-W)$ is a $\Sc$ structure for which \begin{equation}\label{eqn:Wtnonzero}\hfp(-W,\spt;\Gamma_\nu)([[\bcc_{\beta\alpha}\cup\{u_{\gamma\alpha}\}\cup \bx_{\beta\alpha},0]]) \neq 0.\end{equation} In particular, this implies that $\spt\in\SA(-W|{-}R)$. Fix a Whitney triangle in \begin{equation*}\label{eqn:pi2-2}\pi_2(\Theta, \bcc_{\beta\alpha}\cup\{u_{\beta\alpha}\}\cup \bx_{\beta\alpha}, \by\cup\{u_{\gamma\alpha}\}\cup \bx_{\gamma\alpha})\end{equation*} with domain $T$ satisfying $n_z(T)=0$ and $\spc_z(T)=\spt$ (such a   triangle must exist by \eqref{eqn:Wtnonzero}, since the quantity on the left is a linear combination of generators whose coefficients count holomorphic triangles with domains satisfying these properties). The boundary of $T{-}\Delta_{S}$ consists of  
\begin{itemize}
\item integer multiples of complete $\gamma_i$, $\beta_i$,   $\alpha_i$ curves for $i=n{+}1,\dots,n{+}2g{+}1$, and
\item  integer multiples of arcs of the $\gamma_i$, $\beta_i$,  $\alpha_i$ curves for $i=1,\dots,n$.
\end{itemize} We may therefore assume, after adding some integer  linear combination of  the elements in the basis  \eqref{eqn:basisTextended}, \[\{\P_{\gamma\alpha},\D_1,\dots,\D_p,\T_{p+1}\dots,\T_m,\P_{n+1},\dots,\P_{n+2g+1}\},\] that the boundary of $T-\Delta_S$ is  disjoint from 
\begin{itemize}
\item $\beta_i$ for $i=n{+}1,\dots, n{+}2g{+}1$,
\item  $\gamma_{i_j}$ and $\alpha_{i_j}$ for $j=p{+}1,\dots,m$,
\item  $\gamma_{n+2g+1}$ and $\alpha_{n+2g+1}$,
\end{itemize} 
in analogy with the discussion in the proof of Lemma \ref{lem:stddisk}. This $T-\Delta_S$ will serve as a reference domain for the remainder of the proof.

Fix some integer $M$ such that \begin{equation}\label{eqn:choiceN}M-|k|>|\mu(T)| + \max\{|\ul n_{w_{i_j}}(T)| \mid j=1,\dots,m\}.\end{equation} (Recall once more that the  $i_j$ subscripts refer to curves in the boundaries of $\D_j$ and $\T_j $ as in \eqref{eqn:basis2} and \eqref{eqn:basisTstd}.) 
Now suppose $\phi$ is any Whitney triangle, \[\phi\in \pi_2(\Theta, \bcc_{\beta\alpha}\cup\{u_{\beta\alpha}\}\cup \bx_{\beta\alpha}, \by\cup\{u_{\gamma\alpha}\}\cup \bx_{\gamma\alpha}),\]  where $\mu(\phi)=k$, $n_z(\phi)=0$,  $D(\phi)$ has no negative multiplicities, and $\spc_z(\phi)=\spt$, as in the hypothesis of the lemma. The last condition  implies that  $D(\phi)$  differs from $T$ by a sum of doubly-periodic domains \cite[Proposition 8.5]{osz8}. We can therefore write \[D(\phi)=T+P\] for some integer linear combination $P$ of  elements of the basis in \eqref{eqn:basisTextended}. The fact $P$ is a sum of rational doubly-periodic domains implies that $\spc_z(T)=\spc_z(\phi)=\spt$ and also that
\[\mu(\phi) = \mu(T) + \langle c_1(\spt),[P]\rangle.\] Then, exactly as in the proof of Lemma \ref{lem:stddisk} (but appealing to Lemma \ref{lem:tech1} for $N=M$ and Lemma \ref{lem:tech2} for $N=1$ rather than Lemmas \ref{lem:tech13} and \ref{lem:tech23}), we can conclude using  that after a sufficient amount of winding (which does not depend on $P$), in order for $\mu(\phi)=k$ and for $D(\phi)$ to have no negative multiplicities,   $P$ must be an integer linear combination \[P=a_{\gamma\alpha} \P_{\gamma\alpha} + \sum_{i=1}^p b_i \D_i + \sum_{i=n+1}^{n+2g+1} d_i\P_i,\] such that \[T+ \sum_{i=1}^p b_i \D_i- \Delta_S\subset \Sigma\] in direct analogy with the proof of \eqref{eqn:domainsigma}. We can therefore write
 \begin{equation}\label{eqn:domainD1}D(\phi)=\Big(T + \sum_{i=1}^p b_i \D_i -\Delta_S\Big) + \Delta_S+ a_{\gamma\alpha} \P_{\gamma\alpha} +  \sum_{i=n+1}^{n+2g+1} d_i\P_i.\end{equation} Since there are only finitely many $\spt\in\Sc(-W)$ satisfying \eqref{eqn:Wtnonzero}, as discussed in Subsection \ref{ssec:cob}, this conclusion  holds  simultaneously for all such $\spt$ after sufficient winding. In summary,  for $g\geq 2pL+2$ and sufficient winding, any Whitney triangle as in the statement of the lemma has domain \begin{equation}\label{eqn:domainD}D(\phi) = D(\phi') + \Delta_S + a_{\gamma\alpha} \P_{\gamma\alpha} +  \sum_{i=n+1}^{n+2g+1} d_i\P_i,\end{equation} for integers $a_{\gamma\alpha}$ and $d_{n+1},\dots,d_{n+2g+1}$ (note that $a_{\gamma\alpha}$ must be nonnegative), where  \begin{equation}\label{eqn:phi'}\phi'\in\pi_2(\{\Theta^1,\dots,\Theta^n\},\bcc_{\beta\alpha},\by)\end{equation} is a Whitney triangle in the sutured Heegaard triple diagram \[(\Sigma,\{\gamma_1,\dots,\gamma_n\},\{\beta_1,\dots,\beta_n\},\{\alpha_1,\dots,\alpha_n\}).\]  
 
 To complete the proof of the lemma in this case, we next show that for sufficiently large  $g$ and sufficient winding it must also be true that $a_{\gamma\alpha} = d_{n+1}=\dots=d_{n+2g+1}=0$ in \eqref{eqn:domainD1} and 
that $D(\phi')=\Delta_\Sigma$. This will imply that \[D(\phi) = \Delta_\Sigma+\Delta_S = \Delta_{\ul\Sigma},\] which in particular implies that $\by = \bcc_{\gamma\alpha}$, as desired. 

We will first show that that $a_{\gamma\alpha} = d_{n+1}=\dots=d_{n+2g+1}=0$.
For this, fix some \[\phi''\in\pi_2(\{\Theta^1,\dots,\Theta^n\},\bcc_{\beta\alpha},\by).\] Fix some $N$ with \begin{equation*}\label{eqn:choiceN2}N-|k|>|\mu(\phi'')| + \max\{|\ul n_{w_i}(\phi'')| \mid i=1,\dots,n\}.\end{equation*} Let \begin{equation}\label{eqn:genuslem}g\geq \max\{2pL+2,pLN+pL+N+1\},\end{equation} and   wind so that any  Whitney triangle $\phi$ as in the statement of the lemma can be written in the form \eqref{eqn:domainD}.   Suppose $\phi$ is such a  triangle, with $\spc_z(\phi)=\spt$, and write $D(\phi)$ in this form, \[D(\phi) = D(\phi') + \Delta_S + a_{\gamma\alpha} \P_{\gamma\alpha} + \sum_{i=n+1}^{n+2g+1} d_i\P_i.\] The domains $D(\phi')$ and $D(\phi'')$ differ by a triply-periodic domain in the sutured Heegaard triple diagram. From  the equality $\Pi_{\gamma\beta\alpha} = \Pi_{\gamma\alpha}$ established in Subsection \ref{ssec:periodic}, all such domains are actually  $(\boldsymbol\gamma,\boldsymbol\alpha)$-periodic domains. So, we can write \[D(\phi') = D(\phi'') + \sum_{i=1}^p b_i \D_i\] for some integers $b_1,\dots,b_p$. We  therefore have \[D(\phi) = D(\phi'') + \Delta_S + P,\] where  \[P=a_{\gamma\alpha} \P_{\gamma\alpha} +\sum_{i=1}^p b_i \D_i + \sum_{i=n+1}^{n+2g+1} d_i\P_i.\]   Since $P$ is a sum of rational doubly-periodic domains, \[\spc_z(D(\phi'')+\Delta_S)=\spc_z(\phi)=\spt,\] which implies that \[\mu(\phi)=\mu(D(\phi'')+\Delta_S)+\langle c_1(\spt),[P]\rangle=\mu(\phi'')+\langle c_1(\spt),[P]\rangle.\] Then we may proceed exactly as in the proof of Lemma \ref{lem:stddisk}, using Lemma \ref{lem:tech2} rather than Lemma \ref{lem:tech23}, to conclude that $a_{\gamma\alpha}=0$. We therefore have that \[D(\phi) = D(\phi'') + \Delta_S + \sum_{i=1}^p b_i \D_i + \sum_{i=n+1}^{n+2g+1} d_i\P_i.\] Moreover, we must have $d_i=0$ for all $i=n{+}1,\dots,n{+}2g{+}1$ since otherwise $D(\phi)$ would have a negative multiplicity $-|d_i|$ in some thin region between $\gamma_i$ and $\beta_i$. Therefore,  we have \[D(\phi)=D(\phi'') + \Delta_S + \sum_{i=1}^p b_i \D_i=D(\phi')+\Delta_S,\] as desired. It remains to show that $D(\phi')=\Delta_\Sigma$. 

The fact that $D(\phi)$ has no negative multiplicities implies that $D(\phi')$ has no negative multiplicities. We claim  this implies that $D(\phi')=\Delta_\Sigma$. To see this, we refer the diagram in Figure \ref{fig:phiprime} below which shows the possible multiplicities of $D(\phi')$ near $\Theta^i$ and $c_{\beta\alpha}^i$. Either $c_{\gamma\alpha}^i$ is not a component of $\by$ in which case $y=-x$; or $c_{\gamma\alpha}^i$ is a component of $\by$ in which case $y=1-x$. In the first case, we must have $x=0$. But that would imply that $x-1=-1$, which would mean that $D(\phi')$ has a negative multiplicity. Therefore, $y=1-x$. But this forces $x=1$. It follows that the domain of $D(\phi')$ near $\Theta^i$ and $c_{\beta\alpha}^i$ consists just of the triangle $\Delta^i$. This implies that  $D(\phi') =\Delta_\Sigma$, as claimed, completing the proof of the lemma in this case.

\begin{figure}[ht]
\labellist
\tiny

\pinlabel $\Theta^i$ at 37 156
\pinlabel $c_{\beta\alpha}^i$ at 93 183
\pinlabel $c_{\gamma\alpha}^i$ at 115 85
\pinlabel $0$ at 67 95
\pinlabel $0$ at 126 101
\pinlabel $0$ at 12 173
\pinlabel $0$ at 170 173
\pinlabel $x$ at 81 140
\pinlabel $y$ at 97 30
\pinlabel $x-1$ at 52 182

\endlabellist
\centering
\includegraphics[width=3.7cm]{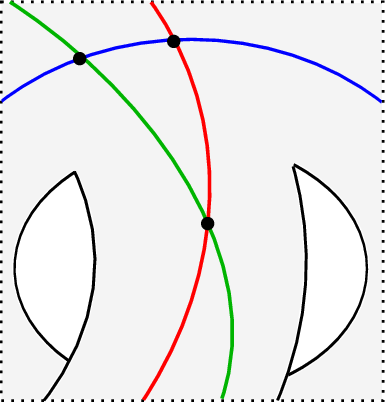}
\caption{The possible multiplicities of $D(\phi')$ near $\Theta^i$ and $c_{\beta\alpha}^i$ and $c_{\gamma\alpha}^i$.}
\label{fig:phiprime}
\end{figure}

\noindent \emph{\underline{Case 2: $s=v_{\gamma\alpha}$}}. As in the proof of Lemma \ref{lem:stddisk}, this case follows easily from the previous case. Let $B$ denote the bigon shown in Figure \ref{fig:bigon}. Observe that for any \[\phi\in\pi_2(\Theta, \bcc_{\beta\alpha}\cup\{u_{\beta\alpha}\}\cup \bx_{\beta\alpha}, \by\cup\{v_{\gamma\alpha}\}\cup \bx_{\gamma\alpha})\] as in the hypothesis of the lemma, where $\mu(\phi)=k$, we can write \begin{equation}\label{eqn:dphi-2}D(\phi) = D(\phi') + B-\ul\Sigma\end{equation} for some Whitney triangle \[\phi'\in\pi_2(\Theta, \bcc_{\beta\alpha}\cup\{u_{\beta\alpha}\}\cup \bx_{\beta\alpha}, \by\cup\{u_{\gamma\alpha}\}\cup \bx_{\gamma\alpha})\] as in the hypothesis of the lemma, where $\mu(\phi')=k+1$. We proved in the previous case that for sufficiently large $g$ and sufficient winding, any such $\phi'$ satisfies \[D(\phi')=\Delta_{\ul\Sigma}.\] Suppose  that $g$ is large enough and we have wound sufficiently that this holds, and let $\phi$ be as above. Then $D(\phi)$ has a negative multiplicity in the region $\ul R_S$ by \eqref{eqn:dphi-2}, since $B-\ul\Sigma$ has multiplicity $-1$ in this region outside of $B$, and $D(\phi')$ has multiplicity zero in this region outside of the small triangles $\Delta_S$, a contradiction. We conclude that for sufficiently large $g$ and sufficient winding, there is no Whitney triangle $\phi$ as in the statement of the lemma when $s=v_{\gamma\alpha}.$
\end{proof}

As mentioned above, Theorem \ref{thm:map} follows easily from Lemma \ref{lem:stdtri}.

\begin{proof}[Proof of Theorem \ref{thm:map}]
Since there are only finitely many $\by$ as in \eqref{eqn:y}, Lemma \ref{lem:stdtri} in the case $k=0$ tells us that for sufficiently large $g$ and sufficient winding,  \[ (f^{+}_{\gamma\beta\alpha;\Gamma_\nu})_*([[\bcc_{\beta\alpha}\cup\{u_{\beta\alpha}\}\cup \bx_{\beta\alpha},0]])=\#\mathcal{M}(\phi)\cdot[[\bcc_{\gamma\alpha}\cup\{u_{\gamma\alpha}\}\cup \bx_{\gamma\alpha},0]]\cdot t^{\partial_\alpha(\phi)\cdot \eta},\] where $\phi$ is the homotopy class of Whitney triangles with domain $\Delta_{\ul\Sigma}$. But this homotopy class has a unique holomorphic representative, and $\Delta_{\ul\Sigma}$ is entirely disjoint from $\eta$. Thus, \[\#\mathcal{M}(\phi)\cdot t^{\partial_\alpha(\phi)\cdot \eta}=1,\] completing the proof of Theorem \ref{thm:map}.
\end{proof}

\section{Obstructing  Lagrangian concordance}
\label{sec:concex}

In this section, we provide further examples which  demonstrate the effectiveness of the invariant $\Theta_{\HF}$ in obstructing Lagrangian concordance, per  Theorem \ref{thm:grid-concordance}. Our main result is the following.

\begin{theorem}
\label{thm:concordance-examples}
There are infinitely many pairs $(K_1,K_2)$ of Legendrian knots in $(S^3,\xi_{std})$ such that:
\begin{itemize}
\item $K_1$ and $K_2$ are smoothly concordant and have the same classical invariants $tb$ and $r$,
\item $K_1$ is a negative stabilization of another Legendrian knot,
\item $K_2$ is neither a positive nor negative stabilization of another Legendrian knot, and 
\item $\Theta_{\HF}(K_1) \neq 0$ while $\Theta_{\HF}(K_2) = 0$.
\end{itemize}
In particular, the last condition  implies by Theorem~\ref{thm:grid-concordance} that there is no Lagrangian concordance from $K_1$ to $K_2$.
\end{theorem}

The fact that $K_1$ in Theorem \ref{thm:concordance-examples} is a stabilization implies that its Legendrian contact homology DGA is trivial. Legendrian contact homology therefore fails to obstruct a Lagrangian concordance from $K_1$ to $K_2$ for these examples.

In our proof of Theorem \ref{thm:concordance-examples} below, we will adopt the  convention that the $(p,q)$-cable of a knot $K$, denoted by $C_{p,q}(K)$, has longitudinal winding $p$ and meridional winding $q$. We will also, for notational convenience, we will denote the $(r,s)$-cable of $C_{p,q}(K)$ simply by \[C_{p,q;r,s}(K):=C_{r,s}(C_{p,q}(K)).\] Given a Legendrian knot $K$, we will use $K^+$ to denote its transverse pushoff which satisfies \[sl(K^+) = tb(K)-r(K).\]  Finally, given a  knot $K$ we will denote by $\maxtb(K)$ and $\maxsl(K)$ the maximal Thurston-Bennequin and self-linking numbers among Legendrian and transverse knots   smoothly isotopic to $K$.

\begin{proof}[Proof of Theorem \ref{thm:concordance-examples}]
In our examples, $K_1$ will be a  Legendrian representative of the iterated torus knot $C_{3,2;3,2}(U)$ and $K_2$ will be a Legendrian representative of \[C_{3,2;3,2}(P(-m,-3,3)) \# 6_1,\] where $P(-m,-3,3)$ is the usual pretzel knot, for any $m \geq 3$. Such $K_1$ and $K_2$ are smoothly concordant as  these pretzel knots and the twist knot $6_1$ are all smoothly slice. 

To define the Legendrian representative $K_1$, we  rely on the following result of Ng, Ozsv{\'a}th, and Thurston from  \cite[Section~3.3]{not}. The Legendrian $K$ in the proposition below was first discovered and studied by Etnyre and Honda in \cite{EH4}.

   \begin{proposition}[Ng--Ozsv{\'a}th--Thurston]
\label{prop:cable-trefoil-example}
There is a Legendrian representative $K$ of  the iterated torus knot $C_{3,2;3,2}(U)$ with $(tb(K), r(K)) = (5,2)$ and $\Theta_{\HF}(K) \neq 0$.
\end{proposition}

We then define $K_1$ to be the Legendrian knot obtained by negatively stabilizing  this knot $K$  three times. It follows that \[(tb(K_1),r(K_1))=(2,-1)\textrm{ and }\Theta_{\HF}(K_1) \neq 0\] since $\Theta_{\HF}$ is preserved by negative stabilization.

We now record some facts that will be relevant in defining $K_2$. First, we record that $\maxsl(C_{3,2;3,2}(U)) = 7$, where this maximal self-linking number is realized by the transverse pushoff of the unique Legendrian representative with $(tb,r)=(6,-1)$; see \cite{EH4} for the full Legendrian classification of $C_{3,2;3,2}(U)$.  Since $g(C_{3,2;3,2}(U)) = 4$, the three inequalities \[ \maxsl(C_{3,2;3,2}(U)) \leq 2\tau(C_{3,2;3,2}(U))-1 \leq 2g_s(C_{3,2;3,2}(U))-1 \leq 2g(C_{3,2;3,2}(U))-1 \] are actually equalities. Here, $\tau$ is the Ozsv\'ath--Szab\'o concordance invariant \cite{osz10} and $g_s$ is the smooth slice genus; see \cite{pla4} for the first inequality and \cite{osz10} for the second.  In particular, we have that  $\tau(C_{3,2;3,2}(U))=4$. We now use these facts to prove the following.

\begin{lemma}
\label{lem:tb-iterated-cable}
Suppose $K$ is a smoothly slice knot with $\maxtb(K)=-1$.  Then there is a Legendrian representative of $C_{3,2;3,2}(K)$ with $(tb,r)=(6,-1)$.  This Legendrian knot achieves the bound $\maxtb(C_{3,2;3,2}(K)) = 6$, and its transverse pushoff achieves the bound $\maxsl(C_{3,2;3,2}(K)) = 7$.
\end{lemma}

\begin{proof}
According to \cite[Corollary~1.17]{lidsiv}, we have $\maxtb(C_{3,2}(K)) \geq 1$.  As $C_{3,2}(K)$ is smoothly concordant to the right-handed trefoil $C_{3,2}(U)$, we have \[\tau(C_{3,2}(K)) = \tau(C_{3,2}(U)) = 1,\] which implies that \[\maxtb(C_{3,2}(K)) \leq 2\tau(C_{3,2}(K))-1 = 1\] by \cite{pla4}.  It follows that $\maxtb(C_{3,2}(K)) = 1$. Applying  \cite[Corollary~1.17]{lidsiv} once more, we may then conclude that $\maxtb(C_{3,2;3,2}(K)) = 6$, as claimed in  the lemma.  Since $C_{3,2;3,2}(K)$ is smoothly concordant to $C_{3,2;3,2}(U)$, we have that \[\tau(C_{3,2;3,2}(K)) = \tau(C_{3,2;3,2}(U))=4,\] and, therefore, that \[\maxsl(C_{3,2;3,2}(K)) \leq 2\tau(C_{3,2;3,2}(K))-1 = 7.\]  This bound is achieved by the transverse pushoff of a $tb$-maximizing Legendrian representative: since $tb=6$ is even for this representative its  rotation number must be odd, and up to reversing orientation we can ensure that $r\leq -1$; hence, the transverse pushoff has $sl=tb-r\geq7$, which implies that this inequality is actually an equality.
\end{proof}

There are infinitely many knots $K$ satisfying the hypothesis of Lemma~\ref{lem:tb-iterated-cable}. These include  the examples of \cite[Theorem~2.10]{cns} and the pretzel knots $P(-m,-3,3)$, for $m \geq 3$, as mentioned in \cite[Section~4.4]{cns}.  Fix any such $K$ and  let $L_1$ be a Legendrian representative of $C_{3,2;3,2}(K)$ with $tb(L_1) = 6$ and $r(L_1) = -1$, whose existence is guaranteed by Lemma \ref{lem:tb-iterated-cable}.

Etnyre, Ng, and V\'ertesi \cite{env} classified Legendrian and transverse representatives of the $6_1$ knot, which in their notation is the twist knot $K_4$. Namely, there is a single $tb$-maximizing Legendrian representative $L_2$ with $(tb,r)=(-5,0)$, and all other representatives are stabilizations of $L_2$, so it follows that $\maxtb(6_1) = -5$ and $\maxsl(6_1) = -5$.

We now define the Legendrian representative $K_2$ of $C_{3,2;3,2}(K)\#6_1$ to be the connected sum $K_2=L_1\#L_2.$ We show below that $K_2$ satisfies the conditions in Theorem \ref{thm:concordance-examples}. 

\begin{proposition}
\label{prop:construct-k-vanishing-grid}
The Legendrian knot $K_2$ has the same classical invariants as $K_1$, it is not a stabilization, and $\Theta_{\HF}(K_2) = 0$.
\end{proposition}

\begin{proof}
Both $L_1$ and $L_2$ maximize $tb$ within their knot types, so we have
\[ tb(K_2) = tb(L_1)+tb(L_2)+1 = \maxtb(L_1)+\maxtb(L_2)+1 = \maxtb(K_2) \]
by Lemma 3.3 and Corollary 3.5 of \cite{EH5}.  More precisely, we compute that $tb(K_2) = 2$ and \[r(K_2) = r(L_1)+r(L_2) = -1.\] This shows that $K_2$ has the same classical invariants as $K_1$. The fact that $K_2$ is a $tb$-maximizer also implies that it is not a stabilization. In order to show that $\Theta_{\HF}(K_2)=0$, we appeal to a result of V\'ertesi \cite[Corollary~1.3]{vera}, which says that there is an isomorphism
\[ \HFKh(m(L_1)) \otimes \HFKh(m(L_2)) \to \HFKh(m(L_1\#L_2)) \]
sending \[\Theta_{\HF}(L_1) \otimes \Theta_{\HF}(L_2)\textrm{ to }\Theta_{\HF}(L_1 \# L_2) = \Theta_{\HF}(K_2).\]  It therefore suffices to show that $\Theta_{\HF}(L_2) = 0$.  But $L_2$ represents an alternating knot type, so it has thin knot Floer homology \cite{osz9}. Therefore, by \cite[Proposition~3.4]{not}, we have $\Theta_{\HF}(L_2) \neq 0$ if and only if $sl(L_2^+) = 2\tau(L_2) - 1$.  The left side is $-5$, but the right side is $-1$ since $L_2$ is smoothly slice, so we have that $\Theta_{\HF}(L_2)=0$, as desired.
\end{proof}

This completes the proof of Theorem \ref{thm:concordance-examples}.
\end{proof}

\bibliographystyle{hplain}
\bibliography{References}

\end{document}